\documentclass{fundam}

\usepackage{amsmath,amsfonts,amssymb, mathtools, latexsym, stmaryrd}
\usepackage[T1]{fontenc}
\usepackage{mathabx}
\usepackage[all,color]{xy}
\usepackage{rotating}
\usepackage{enumitem}

\usepackage{bbm}
\usepackage{cite}

\newcommand{\rad}{\mathbf{rad}}
\newcommand{\Hom}{\mathrm{Hom}}
\newcommand{\End}{\mathrm{End}}
\newcommand{\inc}{\underline{\mathbf{inc}}}

\newcommand{\Ker}{\mathbf{Ker} }

\newcommand{\Img}{\mathbf{Im} }
\newcommand{\Rnk}{\mathbf{rk}}
\newcommand{\tr}{\mathbf{tr}}
\newcommand{\Dyn}{\mathbf{Dyn}}
\newcommand{\Z}{\mathbb{Z}}

\newcommand{\A}{\mathbb{A}}

\newcommand{\D}{\mathbb{D}}
\newcommand{\E}{\mathbb{E}}

\newcommand{\K}{\mathbb{K}}
\newcommand{\CRnk}{\mathbf{cork}} 

\newcommand{\Id}{\mathbf{Id}}
\newcommand{\Quad}{\mathbf{UQuad}}
\newcommand{\quiv}{\mathbf{Quiv}}

\newcommand{\Adj}{\mathrm{Adj}}

\newcommand{\diag}{\mathrm{diag}}

\newcommand{\Star}{\mathbb{S}}
\newcommand{\bas}{\mathbf{e}}

\newcommand{\sou}{\mathbf{s}}
\newcommand{\tar}{\mathbf{t}}

\newcommand{\coxN}{\mathbf{C}}    
\newcommand{\va}{\lambda}

\newcommand{\Part}{\mathcal{P}}
\newcommand{\ct}{\mathbf{ct}}
\newcommand{\lcm}{\mathrm{lcm}}

\newcommand{\bulito}{\mathmiddlescript{\bullet}}
\newcommand{\oper}[1]{\circledast_{#1}}
\newcommand{\wt}[1]{\widetilde{#1}}
\newcommand{\wc}[1]{\widecheck{#1}}
\newcommand{\pseudo}{\doteq}
\newcommand{\dd}{\mathbbm{d}}

\newcommand\mathmiddlescript[1]{\vcenter{\hbox{$\scriptstyle #1$}}}

\newtheorem{theoremMain}{Theorem}

\theoremstyle{definition}
\newtheorem{algorithm}{Algorithm}

   \usepackage{hyperref}

\begin{document}


\setcounter{page}{17}
\publyear{24}
\papernumber{2171}
\volume{191}
\issue{1}

\finalVersionForARXIV


\title{A Strong Gram Classification of  Non-negative Unit Forms of\\  Dynkin Type $\A_{r}$}

\author{Jes\'us Arturo~Jim\'enez Gonz\'alez\thanks{Address of correspondence: Instituto de Matem\'aticas, Mexico City, Mexico. \newline \newline
          \vspace*{-6mm}{\scriptsize{Received April 2022; \ revised February 2024.}}}
\\
 Instituto de Matem\'aticas\\
  UNAM, Mexico\\
jejim@im.unam.mx
}

\maketitle

\hbox{}\hspace{8cm} \emph{To the memory of Prof. Daniel Simson}

\runninghead{J.A. Jim\'enez Gonz\'alez}{A Strong Gram Classification of  Non-negative Unit Forms of Dynkin Type $\A_{r}$}

\begin{abstract}
An integral quadratic form $q:\Z^n \to \Z$ is usually identified with a bilinear form $\wc{\mathbbm{b}}_q: \Z^n \times \Z^n \to \Z$ satisfying $q(x)=\wc{\mathbbm{b}}_q(x,x)$ for any vector $x$ in $\Z^n$, and such that its Gram matrix with respect to the canonical basis of $\Z^n$ is upper triangular. Two integral quadratic forms are called strongly (resp. weakly) Gram congruent if their corresponding upper triangular bilinear forms (resp. their symmetrizations) are equivalent. If $q$ is unitary, such upper triangular bilinear form is unimodular, and one considers the associated Coxeter transformation and its characteristic polynomial, the so-called Coxeter polynomial of $q$ with this identification. Two strongly Gram congruent quadratic unit forms are weakly Gram congruent and have the same Coxeter polynomial.

Here we show that the converse of this statement holds for the connected non-negative case of Dynkin type $\A_{r}$ ($r \geq 1$) and arbitrary corank, and use this characterization to complete a combinatorial classification of such quadratic forms started in~[Fundamenta Informaticae~\textbf{184}(1):49--82, 2021] and~[Fundamenta Informaticae~\textbf{185}(3):221--246, 2022].
\end{abstract}

\begin{keywords}
Integral quadratic form,  Gram congruence,  Dynkin type,  Coxeter polynomial,  edge-bipartite graph, quiver, incidence matrix,  signed line graph.
2020 MSC:  15A63, 15A21, 15B36, 05C22, 05C50, 05C76, 05B20.\\
\end{keywords}

\section*{Introduction} \label{S0}


An integral quadratic form $q(x_1,\ldots,x_n)=\sum_{1\leq i \leq j \leq n}q_{ij}x_ix_j$ is an integer homogeneous polynomial ($q_{ij} \in \Z$) of degree two on $n \geq 1$ integer variables $x_1,\ldots,x_n$, more generally considered as a function $q:\Z^n \to \Z$ whose associated map
\[
\mathbbm{b}_q:\Z^n \times \Z^n \to \Z,
\]
given for (column) vectors $x,y \in \Z^n$ by $\mathbbm{b}_q(x,y)=q(x+y)-q(x)-q(y)$, is a bilinear form, usually called \textbf{polarization} of $q$. The form $q$ is said to be positive (resp. non-negative) if $q(x)>0$ (resp. $q(x) \geq 0$) for any non-zero vector $x \in \Z^n$, that is, whenever the polarization $\mathbbm{b}_q$ is a positive (semi-) definite form, since $q(x)=\frac{1}{2}\mathbbm{b}_q(x,x)$ for any $x$ in $\Z^n$. Recall that two integral bilinear forms $\mathbbm{b}$ and $\mathbbm{b}'$ are called \textbf{equivalent} if there is a $\Z$-invertible matrix $B$ such that $\mathbbm{b}'(x,y)=\mathbbm{b}(Bx,By)$ for any $x,y \in \Z^n$.

\medskip
Integral quadratic forms appear frequently, sometimes implicitly, in Lie theory, in the representation theory of groups, algebras, posets and bocses, in cluster theory, and in the spectral graph theory of signed graphs, to mention some examples. Their usefulness, in representation theory alone, which is our main motivation, has prompted extensive original research for some decades now. For instance:
\begin{itemize}
\itemsep=-.9pt
 \item In the early stages of the representation theory of associative algebras of finite dimension, after the work of Gabriel~\cite{pG72}: Bernstein, Gelfand and Ponomarev~\cite{BGP72}, Ovsienko~\cite{saO79} (see also~\cite{DDGOZP95} and~\cite{cmR}), Dlab and Ringel~\cite{DR76,DR79}, Ringel~\cite{cmR},  Bongartz~\cite{kB83,kB85frt}, de la Pe\~na~\cite{JAP90,JAP90a}, Br{\"u}stle, de la Pe\~na, Skowro\'nski~\cite{BPS11}.

 \item Usually in a graphical context, considering arithmetical properties and classification problems of quadratic forms: Barot~\cite{B99,B01}, Barot and de la Pe\~na~\cite{BP99,BPnonNeg,BP06} von H{\"o}hne~\cite{vH88,vH95,vH96}, Dean and de la Pe\~na~\cite{DP96}, Happel~\cite{H}, Dr{\"a}xler and de la Pe\~na~\cite{DP99,DP00}.

 \item Within Lenzing's Coxeter formalism of bilinear forms~\cite{Le96}: Lenzing and Reiten~\cite{LeRe06}, Mr\'oz~\cite{aM16a,aM16b}, Mr\'oz and de la Pe\~na~\cite{MP14,MP16}.

 \item In a graphical context, considering morsifications, Weyl and isotropy groups, certain mesh geometries of orbits of roots, and classification problems: Simson~\cite{dS11a,dS13a,dS13b,dS13c,dS13,dS16a,dS16b,dS18,dS19,dS20,dS21a,dS21b}, Kosakowska~\cite{jK12}, and Simson and collaborators: Bocian and Felisiak~\cite{BFS13,BFS14}, G\k{a}siorek and Zaj\k{a}c~\cite{GSZ14,GSZ16}, Kasjan~\cite{KS15a,KS15b}, Makuracki and Zyglarski~\cite{MS19,MSZ17}, Zaj\k{a}c~\cite{SZ12,SZ17}.

 \item Within the context of quasi-Cartan matrices, defined by Barot, Geiss and Zelevinsky in~\cite{BGZ06} for the study of cluster algebras: Simson~\cite{dS20}, Makuracki and Mr\'oz~\cite{MM19,MM21} and Perez, Abarca and Rivera~\cite{PAR18}.
\end{itemize}

Let us fix some of the notation and terminology used in the paper. We denote by $\mathbb{M}_n(\Z)$ the set of $n \times n$ matrices with integer coefficients. The identity matrix of size $n$ is denoted by $\Id_n$, or simply by $\Id$ for adequate size. Recall that $M \in \mathbb{M}_n(\Z)$ is $\Z$-invertible if and only if $\det(M)=\pm 1$. The transpose of a matrix $M$ is denoted by $M^{\tr}$, and if $M$ is $\Z$-invertible then $M^{-\tr}:=(M^{-1})^\tr$. Here all matrices have integer coefficients, and as usual, we identify a $m \times n$ matrix $M$ with the linear transformation $M:\Z^n \to \Z^m$ given by $x \mapsto Mx$. We denote by $\Img(M)$ and $\Ker(M)$ the column space of $M$ and the null right space of $M$, respectively. We say that the matrix $K$ is a \textbf{kernel matrix} of $M$ if its columns consists of a basis of $\Ker(M)$. The column vector with $n$ entries, all of them equal to $1$, is denoted by $\mathbbm{1}_n$, or simply by $\mathbbm{1}$ for appropriate size. For matrices $M_t$ of size $m \times n_t$ for $t=1,\ldots,r$ the $m \times n$ matrix with columns those of $M_1,\ldots,M_r$, in that order, is denoted by $[M_1,M_2,\ldots ,M_{r-1},M_r]$, where $n=\sum_{t=1}^rn_t$. For arbitrary matrices $N_1,\ldots,N_r$, take
\[
N_1 \oplus N_2 \oplus \ldots \oplus N_r:=\begin{pmatrix}N_1&0& \cdots &0\\0&N_2&\cdots&0\\ \vdots & \vdots & \ddots & \vdots \\ 0&0&\cdots&N_r \end{pmatrix}.
\]
The canonical basis of $\Z^n$ is denoted by $\bas_1,\ldots,\bas_n$. For a permutation $\rho$ of the set $\{1,\ldots,n\}$, the matrix $P(\rho)$ satisfying $P(\rho)\bas_t=\bas_{\rho(t)}$ for $t=1,\ldots,n$ is called \textbf{permutation matrix} of $\rho$.

\medskip
The matrix with integer coefficients $G_{\mathbbm{b}}=[\mathbbm{b}(\bas_i,\bas_j)]_{i,j=1}^n$ is called \textbf{Gram matrix} of an integral bilinear form $\mathbbm{b}:\Z^n\times \Z^n \to \Z$, with respect to the canonical basis of $\Z^n$. By \textbf{symmetric Gram matrix} $G_q$ of an integral quadratic form $q$ we mean the Gram matrix $G_q=G_{\mathbbm{b}_q}$ of the polarization $\mathbbm{b}_q$ of $q$ (notice that $G_q$ is symmetric and has integer coefficients). We also consider the (unique) bilinear form $\wc{\mathbbm{b}}_q:\Z^n \times \Z^n \to \Z$ such that $q(x)=\wc{\mathbbm{b}}_q(x,x)$ for all $x \in \Z^n$, and such that its Gram matrix with respect to the canonical basis of $\Z^n$, denoted by $\wc{G}_q$, is upper triangular. Note that $G_q=\wc{G}_q+\wc{G}_q^{\tr}$. We say that $q$ is a \textbf{unit form} (or a unitary integral quadratic form) if $q(\bas_i)=1$ for $i=1,\ldots,n$. In that case, $\wc{G}_q$ is a $\Z$-invertible matrix (since it is upper triangular with all diagonal coefficients equal to $1$), and is called the \textbf{standard morsification} of $q$ in Simson's terminology~\cite{dS11a,dS13}. Two unit forms $q$ and $q'$ are called \textbf{weakly Gram congruent} if their polarizations are equivalent, that is, if there is a $\Z$-invertible matrix $B$ such that $G_{q'}=B^{\tr}G_qB$ (or equivalently, $q'=qB$). Similarly, $q$ and $q'$ are called \textbf{strongly Gram congruent} if their standard morsifications are equivalent, that is,  if there is a $\Z$-invertible matrix $B$ such that $\wc{G}_{q'}=B^{\tr}\wc{G}_qB$. Then we write $q \sim^B q'$ and $q \approx^B q'$ for the weak and strong cases respectively (or simply $q \sim q'$ and $q \approx q'$).  The weak Gram classification of connected non-negative unit forms is due to Barot and de la Pe\~na~\cite{BP99} and Simson~\cite{dS16a} (see also~\cite{SZ17}), in terms of a unique pair $(\Delta,c)$ where $\Delta$ is a Dynkin diagram $\A_r$, $\D_s$, $\E_t$ (for $r\geq 1$, $s \geq 4$ or $t \in \{6,7,8\}$) and $c \geq 0$ is the \textbf{corank} of the quadratic form $q$, that is, the rank of the kernel of the symmetric Gram matrix $G_q$. Here we deal with the strong Gram classification problem of connected non-negative unit forms of Dynkin type $\A_{r}$ for $r \geq 1$.

\medskip
For a unit form $q$, consider the matrix with integer coefficients $\Phi_q=-\wc{G}_q^{\tr}\wc{G}_q^{-1}$, called \textbf{Coxeter matrix} of $q$. The characteristic polynomial of $\Phi_{q}$, denoted by $\varphi_q(\va)$, is called \textbf{Coxeter polynomial} of $q$.  It is well known, and can be easily shown, that if $q \approx q'$, then $q \sim q'$ and $\varphi_q=\varphi_{q'}$ (cf.~\cite[Lemma~4.6]{jaJ2020a}). The validity of the converse of this claim in this, or in partial or equivalent forms, is a question raised by Simson for at least a decade (see~\cite{dS13}). Here we give a formulation in terms of non-negative \emph{unit} forms (see also~\cite[Problem~A]{jaJ2020b}), which correspond to non-negative \emph{loop-less} bigraphs as in~\cite{dS13}, or to non-negative \emph{symmetric} quasi-Cartan matrices as in~\cite{BGZ06}. Generalizations of these problems may be found, for instance, in terms of Cox-regular bigraphs in~\cite[Problem~1.3]{aM16b}, or of symmetrizable quasi-Cartan matrices in~\cite{dS20} (see also~\cite{dS13,dS16a,dS16b}).

\medskip
\noindent \textbf{Problem~1 (Simson's Coxeter Spectral Characterization Question).}\\
If two connected non-negative unit forms are weakly Gram congruent and have the same Coxeter polynomial, are they strongly Gram congruent?

\vspace{2mm}

\noindent \textbf{Problem~2 (Simson's Strong Gram Classification Problems).}\
\begin{itemize}
\itemsep=0.8pt
 \item[i)] Classify all non-negative unit forms up to the strong Gram congruence. This includes (and up to Problem~1, is exhausted by) the determination of all Coxeter polynomials per weak Gram congruence class.
 \item[ii)] Given two non-negative unit forms $q$ and $q'$ that are strongly Gram congruent, find a $\Z$-invertible matrix $B$ such that $q \approx^B q'$.
\end{itemize}

Solutions to these problems for special classes of quadratic forms are known. For instance, the positive case was completed recently by Simson~\cite{dS18,dS21a,dS21b} (see further examples and related problems in~\cite{aM16b}). An alternative proof for positive unit forms of Dynkin type $\A_{r}$ was given by the author in~\cite{jaJ2020a}.  Here we present affirmative solutions to Problems~1 and~2 for connected non-negative unit forms of Dynkin type $\A_{r}$ (for $r \geq 1$) and arbitrary corank, with the combinatorial methods presented in~\cite{jaJ2018}, and developed to this end in~\cite{jaJ2020a,jaJ2020b}.

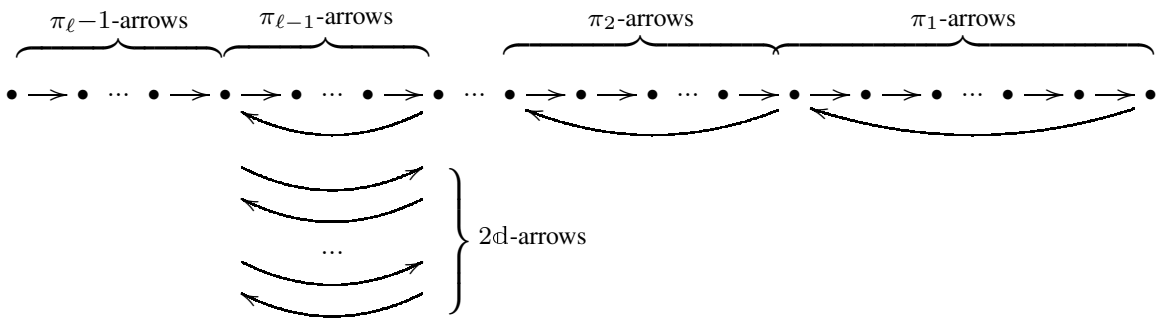
\begin{figure}[!b]
\scalebox{1.2}{
\xy 0;/r.20pc/:
(-68, 10)="Fr1" *{};
( 68, 10)="Fr2" *{};
( 68,-30)="Fr3" *{};
(-68,-30)="Fr4" *{};
(-37, 0)="T1" *{\xymatrix@C=1pc{\bulito \ar[r] & \bulito \ar@{}[r]|-{\cdots} & \bulito \ar[r] & \bulito \ar[r] & \bulito \ar@{}[r]|-{\cdots} & \bulito \ar[r] & \bulito \ar@{}[r]|-{\cdots} \ar@<.5ex>@/^10pt/[lll]
\ar@{<-}@<4ex>@/^10pt/[lll]
\ar@<6ex>@/^10pt/[lll]
\ar@{}@<10ex>[lll]|-{\cdots}
\ar@{<-}@<10ex>@/^10pt/[lll]
\ar@<12ex>@/^10pt/[lll]
 & \bulito \ar[r]  & \bulito \ar[r] & \bulito \ar@{}[r]|-{\cdots} & \bulito \ar[r] & \bulito \ar[r] \ar@<.5ex>@/^10pt/[llll] & \bulito  \ar[r] & \bulito \ar@{}[r]|-{\cdots} & \bulito \ar[r] & \bulito \ar[r] & \bulito \ar@<.5ex>@/^10pt/[lllll] }};
( -16, -19)="B" *{\left. \begin{matrix}\\ \\ \\ \end{matrix} \right\} };
( -6, -19)="Bet" *{{}^{2\dd\text{-arrows}}};
(-60, 8)="F1" *{\overbrace{\qquad \qquad \qquad}^{\pi_{\ell}-1\text{-arrows}}};
(-33, 8)="F2" *{\overbrace{\qquad \qquad \qquad}^{\pi_{\ell-1}\text{-arrows}}};
( 8,   8)="F3" *{\overbrace{\qquad \qquad \qquad \qquad}^{\pi_2\text{-arrows}}};
( 50, 8)="F4" *{\overbrace{\qquad \qquad \qquad \qquad \qquad \quad}^{\pi_1\text{-arrows}}};
\endxy
}
    \caption{For a partition $\pi=(\pi_1,\pi_2,\ldots,\pi_{\ell})$ of $m \geq 2$, and a non-negative integer $\dd$, depiction of the standard $(\pi,\dd)$-extension quiver $\vec{\A}^{\dd}[\pi]$ with $m$ vertices, cycle type $\pi$ and degree of degeneracy $\dd$ (see~Definition~\ref{DStan}). It has $n=m+\ell+2(\dd-1)$ arrows: $m-1$ arrows in the upper row (going from left to right, numbered from $1$ to $m-1$), $\ell-1$ arrows in the second row (going from right to left, numbered from $m$ to $m+\ell-2$), and $2\dd$ alternating parallel arrows, numbered from $m+\ell-1$ to $m+\ell+2(\dd-1)$. The associated quadratic form $q_{\vec{\A}^{\dd}[\pi]}$ is called standard $(\pi,\dd)$-extension of $q_{\A_{m-1}}$, see~(\ref{EQQ0}).}
    \label{Fig}
\end{figure}

Recall that two connected non-negative unit forms are weakly Gram congruent if and only if they have the same Dynkin type and the same corank (see~\cite{BP99} or~\cite{SZ17}), or equivalently, the same Dynkin type and the same number of variables. Simson determined in~\cite{dS16a} representatives of the weak Gram congruence classes of connected non-negative unit forms, the so-called \emph{canonical extensions}, showing that any such form having Dynkin type $\Delta$ is weakly Gram congruent to a unique canonical extension of the unit form $q_{\Delta}$, see~\cite[Theorems~1.12 and~2.12]{dS16a} and~\cite[Theorem~1.8]{SZ17} (recall that the quadratic form $q_{\Delta}$ associated to a graph $\Delta$ is determined by $G_{q_{\Delta}}:=2\Id - \Adj(\Delta)$, where $\Adj(\Delta)$ denotes the symmetric adjacency matrix of $\Delta$). A family of representatives for the corresponding strong Gram classes of Dynkin type $\A_r$ was proposed in~\cite[Definition~5.2]{jaJ2020b} (see Figure~\ref{Fig} and~\ref{SS:stan} below). There they are called \emph{standard extensions} of the unit form $q_{\A_r}$, and here we confirm that they are representatives of strong Gram congruence in the following classification theorem.

\begin{theoremMain}\label{MT0}
Every connected non-negative unit form of Dynkin type $\A_r$ for $r \geq 1$ is strongly Gram congruent to a unique standard extension of the unit form $q_{\A_r}$.
\end{theoremMain}

Alternatively, the following formulation of Theorem~\ref{MT0} answers directly Problem~1.

\begin{theoremMain}\label{MT1}
Let $q$ and $\wt{q}$ be weakly Gram congruent connected non-negative unit forms of Dynkin type $\A_r$ for $r \geq 1$. Then $q$ and $\wt{q}$ are strongly Gram congruent if and only if they have the same Coxeter polynomial.
\end{theoremMain}

Using Theorem~\ref{MT1}, and the results on Coxeter polynomials of~\cite{jaJ2020b} (based on~\cite[Theorem~A]{jaJ2020a}), we complete the following descriptive theorem of non-negative unit forms of Dynkin type $\A_{r}$ up to the strong Gram congruence (cf.~\cite[Problem~B]{jaJ2020b}). We need the following general notions. Let $q:\Z^n \to \Z$ be a unit form, with symmetric Gram matrix $G_q$ and Coxeter matrix $\Phi_q$.
\begin{itemize}
 \item A \textbf{partition} $\pi$ of an integer $m \geq 1$, denoted by $\pi \vdash m$, is a non-increasing sequence of positive integers $\pi=(\pi_1,\ldots,\pi_{\ell(\pi)})$ such that $m=\sum_{t=1}^{\ell(\pi)}\pi_t$. The integer $\ell(\pi)$ is called the \textbf{length} of $\pi$, and the set of partitions of $m$ is denoted by $\Part(m)$.

 \item The kernel of $G_q$ in $\Z^n$ is called \textbf{radical} of $q$, denoted by $\rad(q)$, and its elements are called \textbf{radical vectors} of $q$. The rank of $\rad(q)$ is called \textbf{corank} of $q$, and is denoted by $\CRnk(q)$.  The \textbf{reduced corank} $\CRnk_{re}(q)$ of $q$ is the rank of the kernel of the restriction $\wc{\mathbbm{r}}_q$ of $\wc{\mathbbm{b}}_q$ to its radical (see details in~\ref{reRad} below).

 \item The \textbf{Coxeter number} $\coxN(q)$ of $q$ is the minimal $t >0$ such that $\Phi_q^t=\Id$, if such $t$ exists, and $\coxN(q):=\infty$ otherwise. The \textbf{reduced Coxeter number} $\coxN_{re}(q)$ is the minimal $t>0$ such that $\Phi_q^t-\Id$ is a nilpotent matrix (such $t$ always exists if $q$ is non-negative, cf.~\cite{mS05}).  The \textbf{degree of degeneracy} of $q$ is the integer $\dd_q \geq 0$ such that $2\dd_q=\CRnk(q)-\CRnk_{re}(q)$, which is a non-negative even number since it is the rank of a skew-symmetric form (namely, the restriction $\wc{\mathbbm{r}}_q$ of $\wc{\mathbbm{b}}_q$ to the radical $\rad(q)$ of $q$, see~(\ref{EQQdd}) and further details in~\ref{reRad} below).
\end{itemize}

\begin{theoremMain}\label{MT2}
Let $\Quad_{\A}^c(n)$ denote the set of connected non-negative unit forms in $n \geq 1$ variables having corank $c \geq 0$ and Dynkin type $\A_{n-c}$. Taking $m:=n-c+1\geq 2$, there is a function $\ct:\Quad_{\A}^c(n) \to \Part(m)$, called the \textbf{cycle type} of a unit form in $\Quad_{\A}^c(n)$, which induces a bijection,
\[
\xymatrix{[\Quad_{\A}^c(n)/\approx] \ar[rr]^-{\ct} && \Part^c_1(m)},
\]
where $\Part^c_1(m)$ denotes the set of partitions of $m$ whose length is conditioned by $c$ as follows,
\[
\Part^c_1(m):=\{\pi \vdash m \mid 0 \leq c-(\ell(\pi)-1) \equiv 0 \mod 2\}.
\]
Moreover, if $\ct(q)=(\pi_1,\ldots,\pi_{\ell})$ for $q \in \Quad_{\A}^c(n)$, then the following hold.
\begin{itemize}
 \item[i)] The Coxeter polynomial of $q$ is $\varphi_q(\va)=(\va-1)^{c-1}\prod_{t=1}^{\ell}(\va^{\pi_t}-1)$.
 \item[ii)] The Coxeter number of $q$ is $\coxN(q)=\pi_1$ if $\ell=1$, and $\coxN(q)=\infty$ otherwise.
 \item[iii)] The reduced Coxeter number of $q$ is $\coxN_{re}(q)=\lcm(\pi)$.
 \item[iv)] The geometric multiplicity of $1$ as eigenvalue of $\Phi_q$ is $\CRnk(q)=c$.
 \item[v)] The reduced corank $\CRnk_{re}(q)$ of $q$ is $\ell-1$, and the algebraic multiplicity of $1$ as eigenvalue of $\Phi_q$ is $\CRnk(q)+\CRnk_{re}(q)$.
\end{itemize}
\end{theoremMain}

A solution for the positive and principal cases (coranks zero and one) of Theorem~\ref{MT1}, within our combinatorial framework, was shown in~\cite[Theorems~A and~B]{jaJ2020a} by means of certain (admissible) flations at the level of loop-less quivers. Here we pursue an alternative strategy which we sketch as follows (see details in Section~\ref{S5}).  Let $q$ be a unit form in $\Quad_{\A}^c(n)$. We fix a unique standard extension $\vec{q}$ such that $q \sim \vec{q}$ and $\varphi_{q}=\varphi_{\vec{q}}$ (see Definition~\ref{DEF} and Remark~\ref{RStan} below), and proceed in three main steps, for which, for a $n \times n$ matrix $B$, we consider the matrix $B^*:=\wc{G}_{\vec{q}}^{-1}B^{\tr}\wc{G}_q$.
\begin{itemize}
 \item[] \textit{Step~1.}\label{StepOne} Find a matrix $B$ such that, among other technical conditions (see Definition~\ref{DEF}$(a)$), satisfies
 \[
 \vec{q}=qB, \quad \text{and} \quad q=\vec{q}B^*.
 \]
 \end{itemize}

It can be shown that for corank zero or one, the matrix $B$ of Step~1 determines itself a strong Gram congruence $q \approx^B \vec{q}$ (cf. Lemma~\ref{Lgeneral}). In general, analyzing how far $B$ is from being a strong Gram congruence, we arrive to the following correction steps.

\begin{itemize}
 \item[] \textit{Step~2.} Find a matrix $M$ such that the matrix $B+M$ is $\Z$-invertible, and satisfies the same conditions of Step~1:
 \[
 \vec{q}=q(B+M), \quad \text{and} \quad q=\vec{q}(B+M)^*.
 \]
 \item[] \textit{Step~3.} Find a matrix $C$ such that $[(B+M)C]^*[(B+M)C]=\Id$.
\end{itemize}
Clearly, the condition $N^*N=\Id$ for a square matrix $N$ implies that $N$ is $\Z$-invertible and $q \approx^N \vec{q}$. The goal of this paper is to constructively exhibit the existence of matrices $B$, $M$ and $C$. A solution to Step~1 is given in Section~\ref{S2}, using the specific structure of standard extensions, see Proposition\ref{Ls}. Steps~2 and~3 are simple correction algorithms that work in a much general context (see Remark~\ref{Rgeneral}). However, their justification is long and technical at some points, and requires a special condition that can be easily verified for standard extension (see Lemma~\ref{(L):stan} and Corollary~\ref{CinDet}). Steps~2 and~3 are shown in Sections~\ref{S3} and~\ref{S4}, see Propositions~\ref{(P):inv} and~\ref{(P):strong} and their implementable formulations as Algorithms~\ref{A2} and~\ref{A3}, respectively. The proofs to our main theorems, comments on generalizations and suggestions for an implementation are collected in Section~\ref{S5}.

\section{A combinatorial realization} \label{S2}

In this section we summarize the needed combinatorial notions and results introduced in~\cite{jaJ2018} and developed in~\cite{jaJ2020a,jaJ2020b} for the Coxeter analysis of non-negative unit forms of Dynkin type $\A_{r}$, namely, structural walks, incidence vectors, the inverse of a quiver and standard quivers.

\subsection{Basic notions}

Let $Q=(Q_0,Q_1,\sou,\tar)$ be a \textbf{quiver}, that is, $Q_0$ and $Q_1$ are finite sets (whose elements are called \textbf{vertices} and \textbf{arrows} of $Q$, respectively), and $\sou,\tar:Q_1\to Q_0$ are functions assigning to each arrow $i$ of $Q$ a \textbf{source} vertex $\sou(i) \in Q_0$ and a \textbf{target} vertex $\tar(i) \in Q_0$. Throughout the paper we assume that both sets $Q_0$ and $Q_1$ are totally ordered (see Remark~\ref{RPerm} below). Moreover, taking $m=|Q_0|$ and $n=|Q_1|$ we usually identify the set of vertices $Q_0$ with the set $\{1,\ldots,m\}$, and the set of arrows $Q_1$ with the set $\{1,\ldots,n\}$. The $m \times n$ \textbf{incidence matrix} $I(Q)$ of $Q$ has as $i$-th column the difference $\bas_{\sou(i)}-\bas_{\tar(i)} \in \Z^m$. The \textbf{symmetric Gram matrix} of $Q$ is defined by $G_Q=I(Q)^{\tr}I(Q)$, and its \textbf{upper triangular Gram matrix} $\wc{G}_Q$ is the unique upper triangular matrix satisfying $\wc{G}_Q+\wc{G}_Q^{\tr}=G_Q$. Notice that if $Q$ has no loop, then $\wc{G}_Q$ is $\Z$-invertible, and that if $Q$ is connected then $\Ker I(Q)^{\tr}$ is generated by the single vector $\mathbbm{1}_m$ (see~\cite[Theorem~3.3$(ii)$]{jaJ2020b}).

\medskip
The \textbf{integral quadratic form $q_Q$ associated to a quiver} $Q$ with $m$ vertices and $n$ arrows, as defined in~\cite{jaJ2018}, is given by
\begin{equation}\label{EQQ0}
q_Q(x)=\frac{1}{2}||I(Q)x||^2, \qquad \text{for $x \in \Z^n$,}
\end{equation}
where $||y||^2:=y^{\tr}y$ denotes the squared Euclidean norm of a vector $y \in \Z^m$. By definition, we have $G_{q_Q}=G_Q$, which implies that $\wc{G}_{q_Q}=\wc{G}_Q$.

\begin{remark}\label{RPerm}
Let $Q$ be a connected loop-less quiver. For any permutation $\rho$ of the set of vertices $Q_0$, denote by $\rho \cdot Q=(Q_0,Q_1,\wt{\sou},\wt{\tar})$ the quiver obtained from $Q$ by permuting its vertices via $\rho$ (that is, $\wt{\sou}(i)=\rho(\sou(i))$ and $\wt{\tar}(i)=\rho(\tar(i))$ for any arrow $i \in Q_1$). Denote by $Q^{op}=(Q_0,Q_1,\sou^{op},\tar^{op})$ the quiver obtained from $Q$ by reversing the orientation of all of its arrows (that is, $\sou^{op}=\tar$ and $\tar^{op}=\sou$).
\begin{itemize}
 \item[i)] Then $I(\rho \cdot Q)=P(\rho)I(Q)$ and $I(Q^{op})=(-1)I(Q)$.
 \item[ii)] We have $q_Q=q_{\rho \cdot Q}=q_{\rho \cdot Q^{op}}$, and consequently $\wc{G}_Q=\wc{G}_{\rho \cdot Q}=\wc{Q}_{\rho \cdot Q^{op}}$.
\end{itemize}
\end{remark}
\begin{proof}
Observe that if $i \in Q_1$, then
\begin{eqnarray*}
I(\rho \cdot Q)\bas_i &=& \bas_{\wt{\sou}(i)}-\bas_{\wt{\tar}(i)}=\bas_{\rho \sou(i)}-\bas_{\rho \tar(i)} \\
&=& P(\rho)(\bas_{\sou(i)}-\bas_{\tar((i))})=P(\rho)I(Q)\bas_i,
\end{eqnarray*}
and
\[
I(Q^{op})\bas_i=\bas_{\sou^{op}(i)}-\bas_{\tar^{op}(i)}=(-1)(\bas_{\sou(i)}-\bas_{\tar(i)})=(-1)I(Q).
\]
Then $(i)$ holds. The claim on quadratic forms in $(ii)$ follows from $(i)$ and~(\ref{EQQ0}), since
\[
q_Q(x)=\frac{1}{2}||I(Q)x||^2=\frac{1}{2}||P(\rho)I(Q)x||^2=\frac{1}{2}||(-1)P(\rho)I(Q)x||^2.
\]
The claim on standard morsifications is clear from the first part of $(ii)$.
\end{proof}

In the following result we interpret the weak and strong Gram congruences within this combinatorial realization (compare with~\cite[Lemma~6.1]{jaJ2020b}).

\eject

\begin{theorem}\label{MAINlem}
If $Q$ is a connected loop-less quiver with $m$ vertices and $n$ arrows, then $q_Q$ is a connected non-negative unit form of Dynkin type $\A_{m-1}$ and corank $c=n-m+1$. Moreover, if $q$ is a connected non-negative unit form in $n \geq 1$ variables, with Dynkin type $\A_{m-1}$ for $m \geq 2$ and corank $c \geq 0$, then there is a connected loop-less quiver $Q$ with $n$ arrows and $m=n-c+1$ vertices such that $q=q_Q$. Assuming that $\wt{Q}$ is also a connected loop-less quiver:
\begin{itemize}
 \item[i)] We have $q_Q=q_{\wt{Q}}$ if and only if there is a permutation matrix $P$ and a sign $\epsilon \in \{\pm 1\}$ such that $I(\wt{Q})=\epsilon PI(Q)$.
 \item[ii)] We have $q_Q \sim q_{\wt{Q}}$ if and only if $(|Q_0|,|Q_1|)=(|\wt{Q}_0|,|\wt{Q}_1|)$, which holds if and only if there is a permutation matrix $P$, a sign $\epsilon \in \{\pm 1\}$ and a $\Z$-invertible matrix $B$ satisfying $I(\wt{Q})=\epsilon PI(Q)B$.
 \item[iii)] We have $q_Q \approx q_{\wt{Q}}$ if and only if there is a permutation matrix $P$, a sign $\epsilon \in \{\pm 1\}$ and a ($\Z$-invertible) matrix $B$ satisfying $I(\wt{Q})=\epsilon PI(Q)B$ and $\wc{G}_{\wt{Q}}=B^{\tr}\wc{G}_QB$.
\end{itemize}
\end{theorem}
\begin{proof}
The main claim follows from~\cite[Theorem~5.5]{jaJ2018}, see also~\cite[Proposition~3.15 and Corollary~3.6]{jaJ2020a}. The existence of $P$ and $\epsilon$ in claim $(i)$ follows from~\cite[Corollary~7.3]{jaJ2018}, see~\cite[Lemma~6.1$(i)$]{jaJ2020b} and Remark~\ref{RPerm}$(i)$. The converse follows from Remark~\ref{RPerm}$(ii)$ above.

\medskip
To show $(ii)$, recall first that two non-negative connected unit forms are weakly Gram congruent if and only if they have the same Dynkin type and same corank (see the main corollary in~\cite{BP99}), or equivalently, the same number of variables and the same Dynkin type. By the main part of the theorem, $q_Q$ is a connected unit form on $|Q_1|$ variables and Dynkin type $\A_{|Q_0|-1}$. This shows that $q_Q \sim q_{\wt{Q}}$ if and only if $(|Q_0|,|Q_1|)=(|\wt{Q}_0|,|\wt{Q}_1|)$. For the second equivalence in $(ii)$, it was shown in~\cite[Lemma~6.1$(ii)$]{jaJ2020b} that if $q_Q \sim q_{\wt{Q}}$, then there is a permutation matrix $P$, a sign $\epsilon \in \{\pm 1\}$ and a $\Z$-invertible matrix $B$ satisfying $I(\wt{Q})=\epsilon PI(Q)B$ (note that the sign $\epsilon$ might be ``included'' in matrix $B$, as in~\cite[Lemma~6.1$(ii)$]{jaJ2020b}). The converse is clear, since $\epsilon^2=1$ and $P^{\tr}P=\Id$, and therefore
\[
G_{q_{\wt{Q}}}=G_{\wt{Q}}=I(\wt{Q})^{\tr}I(\wt{Q})=[\epsilon PI(Q)B]^{\tr}[\epsilon PI(Q)B]=B^{\tr}I(Q)^{\tr}I(Q)B=B^{\tr}G_{q_Q}B.
\]
Similarly, the necessity in claim $(iii)$ was shown in~\cite[Lemma~6.1$(iii)$]{jaJ2020b}, and the sufficiency is clear from definition.
\end{proof}

Based on Theorem~\ref{MAINlem}, we propose two characterizations of strong Gram congruence between quadratic forms associated to connected loop-less quivers in Theorem~\ref{MTT} below. The proof and hints for its implementation, which will take the rest of the paper to complete, depend on the following matrices (cf.~\cite[Theorem~3.3]{jaJ2020b}). If $Q$ has $m$ vertices and $n$ arrows, take
\begin{equation}\label{EQQCL}
\Phi_{Q}:=\Id_n-I(Q)^{\tr}I(Q)\wc{G}^{-1}_Q \quad \text{and} \quad \Lambda_Q:=\Id_m-I(Q)\wc{G}_Q^{-1}I(Q)^{\tr}.
\end{equation}
These are called the \textbf{Coxeter-Gram} matrix of $Q$ and the \textbf{Coxeter-Laplacian} of $Q$, respectively. Basic properties of $\Phi_Q$ and $\Lambda_Q$, and their relation with the Coxeter matrix $\Phi_{q_Q}$ of $q_Q$, are collected in the following observation.

\begin{remark}\label{LAMBDA}
Let $Q$ be a connected loop-less quiver with $n$ arrows and $m$ vertices, and with incidence matrix $I(Q)$. Then,
\begin{itemize}
 \item[a)] $\Phi_Q=\Phi_{q_Q}$.
 \item[b)] $\Lambda_Q$ is a permutation matrix.
 \item[c)] $\varphi_{q_Q}(\va)=(\va-1)^{n-m}p_{\Lambda_Q}(\va)$, where $p_{\Lambda_Q}(\va)$ denotes the characteristic polynomial of $\Lambda_Q$.
 \item[d)] $\Lambda_Q^{\tr}I(Q)=I(Q)\Phi_{Q}^{\tr}$.
 \item[e)] $\Lambda_QI(Q)=I(Q)\Phi_{Q}^{-\tr}$.
\end{itemize}
\end{remark}
\begin{proof}
Claims $(a,b,c)$ were shown in~\cite[Theorem~3.3 and Corollary~4.3]{jaJ2020b}. Claim $(d)$ is clear, since
\[
\Lambda_Q^{\tr}I(Q)=I(Q)-I(Q)[I(Q)\wc{G}_Q^{-1}]^{\tr}I(Q)=I(Q)\Phi_{Q}^{\tr}.
\]
Claim $(e)$ follows from $(b)$ and $(d)$, since
\[
\Lambda_QI(Q)=\Lambda_Q[I(Q)\Phi_{Q}^{\tr}]\Phi_{Q}^{-\tr}=\Lambda_Q\Lambda_Q^{\tr}I(Q)\Phi_Q^{-\tr}=I(Q)\Phi_Q^{-\tr}.
\]

\vspace*{-6mm}
\end{proof}

Loosely speaking, due to Remark~\ref{LAMBDA}$(b)$ and $(d,e)$, the (inverse transpose) Coxeter-Gram matrix acts on the columns of the incidence matrix $I(Q)$, and such action is recorded in the Coxeter-Laplacian $\Lambda_Q$ of $Q$.

\subsection{Walks and incidence vectors}

By \textbf{walk} of a quiver $Q$ we mean an alternating sequence of vertices and arrows of $Q$,
\[
w=(v_0,i_1,v_1,i_2,v_2,\ldots,v_{\ell-1},i_{\ell},v_{\ell}),
\]
starting with a vertex $\sou(w):=v_0$ called the \textbf{source} of $w$, and ending with a vertex $\tar(w):=v_{\ell}$ called the \textbf{target} of $w$, and satisfying $\{\sou(i_t),\tar(i_t)\}=\{v_{t-1},v_t\}$ for $t=1,\ldots,\ell$. The integer $\ell(w):=\ell \geq 0$ is the \textbf{length} of $w$, and if $\ell(w)=0$ (that is, if $w=(v_0)$), then $w$ is called a \textbf{trivial walk}. A walk $w$ in $Q$ of length $\ell(w)=1$ has either the shape $w=(\sou(i),i,\tar(i))$ or $w=(\tar(i),i,\sou(i))$ for an arrow $i \in Q_1$. In the first case we use the notation $w=i^{+1}$, and in the second $w=i^{-1}$. Viewing an arbitrary walk $w$ of positive length as concatenation of walks of length one, we use the notation
\[
w=i_1^{\epsilon_1}i_2^{\epsilon_2}\cdots i_{\ell}^{\epsilon_{\ell}}, \quad \text{with $\epsilon_t \in \{\pm 1\}$, for $t=1,\ldots,\ell$,}
\]
where $\tar(i_t^{\epsilon_t})=\sou(i_{t+1}^{\epsilon_{t+1}})$ for $t=1\ldots,\ell-1$. A walk $w$ as above is called \textbf{minimally descending}, if for $t=1,\ldots,\ell-1$ the difference $i_t-i_{t+1}$ is positive and it is the minimal positive difference possible, that is,
\[
i_t-i_{t+1}=\min \limits_j \{ i_t-j \mid \text{$(i_t-j)>0$ and $i_t^{\epsilon_t}j$ is a walk in $Q$} \}.
\]
A minimally descending walk $w$ is called a \textbf{(descending) structural walk}, if whenever a concatenation of the form $w'ww''$ is minimally descending, then both $w'$ and $w''$ are trivial walks. Such walks are determined by their sources (or targets), and we will use the notation $\alpha^-_Q(v)$ for the structural walk having vertex $v$ as source. Take $\xi^-_Q:Q_0 \to Q_0$ given by
\begin{equation}\label{DEFperm}
\xi^-_Q(v):=\tar(\alpha^-_Q(v)).
\end{equation}
The definitions of \textbf{minimally ascending walk} and \textbf{(ascending) structural walk} are analogous, and so are the notions of $\alpha^+_Q(v)$ and $\xi^+_Q$. It can be easily shown that $\xi_Q^-$ and $\xi^+_Q$ are inverse to each other, since
\begin{equation}\label{EQQ15}
\alpha^+_Q(\xi^-_Q(v))=\alpha^-_Q(v)^{-1},
\end{equation}
where $w^{-1}$ denotes the reverse of walk $w$ (see~\cite[Lemma~3.1]{jaJ2020b}). The mapping $\xi^-_Q$ is referred to as \textbf{permutation of vertices} of the quiver $Q$ determined by the ordering of its arrows.  The \textbf{cycle type} of $Q$ is given by the sequence
\begin{equation}\label{EQQct}
\ct(Q)=(\pi_1,\ldots,\pi_{\ell}),
\end{equation}
where $\pi_1 \geq \pi_2 \geq \ldots,\pi_{\ell}>0$ are the cardinalities of the $\xi^-_Q$-orbits on $Q_0$. Then $\ct(Q)$ is a partition of $|Q_0|=m$, and we take $\ct(q_Q):=\ct(Q)$, which is well-defined by Theorem~\ref{MAINlem}$(i)$. We stress that the cycle type $\ct(Q)$ depends on the numbering of the arrows in $Q$, see Example~\ref{EXAfour} below. For a walk $w=i_1^{\epsilon_1}i_2^{\epsilon_2}\cdots i_{\ell}^{\epsilon_{\ell}}$ of $Q$, define the \textbf{incidence vector} $\inc(w) \in \Z^n$ of $w$ as
\begin{equation}
\inc(w)=\sum_{t=1}^{\ell}\epsilon_{t} \bas_{i_t}.
\end{equation}
The following simple identity is fundamental for our analysis (cf.~\cite[Remark~3.2]{jaJ2020b}),
\begin{equation}\label{EQQ2}
I(Q)\inc(w)=\bas_{\sou(w)}-\bas_{\tar(w)}.
\end{equation}
By~(\ref{EQQ0}), it implies that $q(\inc(w)) \in \{0,1\}$ (that is, $\inc(w)$ is a \textbf{$\{0,1\}$-root} of $q_Q$) for any walk $w$ of $Q$, and the converse also holds (cf.~\cite[Lemma~6.1]{jaJ2018}).  Consequently, $\{0,1\}$-roots of $q_Q$ can be treated combinatorially via the walks of quiver $Q$. This also implies that $I(Q)\wc{G}_Q^{-1}$ is also the incidence matrix of a quiver, called the \textbf{(standard) inverse} of $Q$ and denoted by $Q^{\dagger}$. In~\cite{jaJ2020b} we take a constructive route, and derive $Q^{\dagger}$ (denoted by $Q^{-1}$ in~\cite{jaJ2020a} and~\cite{jaJ2020b}) directly from the structural walks of $Q$, as follows.

\medskip
For every arrow $i$ in $Q$, there are exactly two descending structural walks containing arrow $i$, one in the positive direction $i^{+1}$, and the other one in the opposite direction $i^{-1}$. Denote them respectively by $\alpha^-_Q(v)$ and $\alpha^-_Q(w)$ for some vertices $v$ and $w$, and define $\tar^{\dagger}(i):=\tar(\alpha^-_Q(v))$ and $\sou^{\dagger}(i):=\tar(\alpha^-_Q(w))$. In~\cite[Proposition~4.4 and Corollary~4.5]{jaJ2020a} we show that $Q^{\dagger}=(Q_0,Q_1,\sou^{\dagger},\tar^{\dagger})$ is also a connected loop-less quiver, satisfying
\begin{equation}\label{EQQ3}
I(Q^{\dagger})=I(Q)\wc{G}_Q^{-1}.
\end{equation}
This approach is useful for several reasons. Among others, the results of~\cite{jaJ2020b} depend on the following facts.

\begin{lemma}\label{La}
Let $Q$ be a connected loop-less quiver with $m \geq 2$ vertices and $n \geq 1$ arrows. Consider the structural walks $\alpha^{\pm}_Q(v)$ of $Q$ (for $v \in Q_0$), and take $a_v^{\pm}=\inc(\alpha^{\pm}_Q(v))$.
\begin{itemize}
 \item[i)] For any $v \in Q_0$ we have $-a_v^-=a_{\xi^-_Q(v)}^+$.
 \item[ii)] If $Q^{\dagger}$ is the inverse quiver of $Q$, then $I(Q^{\dagger})^{\tr}=[a_1^+,a_2^+,\cdots,a_{m-1}^+,a_m^+]$.
 \item[iii)] We have $\sum_{v=1}^ma_v^-=0=\sum_{v=1}^ma_v^+$.
\end{itemize}
\end{lemma}
\begin{proof}
Claim $(i)$ was shown in~\cite[Remark~2.1]{jaJ2020b}, since $\xi^-_Q(v)=\tar(\alpha^-_Q(v))$. Claim $(ii)$ was shown in~\cite[Lemma~2.3]{jaJ2020b}. Since $I(Q)\bas_i=\bas_{\sou(i)}-\bas_{\tar(i)}$, then $\mathbbm{1}_m^{\tr}I(Q)=0$ for any quiver $Q$, that is, $I(Q)^{\tr}\mathbbm{1}_m=0$. By $(ii)$, this implies that $\sum_{v=1}^ma_v^+=0$, and using $(i)$ we get $\sum_{v=1}^ma_v^-=0$. This shows $(iii)$.
\end{proof}

Combining equations~(\ref{EQQ2}), (\ref{EQQ3}) and Lemma~\ref{La}$(ii)$, with a straightforward calculation we get the following combinatorial expression for the Coxeter-Laplacian $\Lambda_Q$ of Q,
\begin{equation}\label{EQQ5}
\Lambda_Q=P(\xi^-_Q),
\end{equation}
see details in~\cite[Theorem~3.3]{jaJ2020b}. Some basic notions and result will be illustrated with a couple of running examples.
\begin{example}\label{EXAzero}
Consider the following integral quadratic forms $q_1,q_0$ on four variables:
\[
q_1(x_1,x_2,x_3,x_4)=x_1(x_1-x_2-x_3+2x_4)+x_2(x_2-x_3-x_4)+x_3(x_3-x_4)+x_4(x_4),
\]
\[
q_0(x_1,x_2,x_3,x_4)=x_1(x_1-x_2+2x_3-x_4)+x_2(x_2-x_3-x_4)+x_3(x_3-x_4)+x_4(x_4),
\]
with corresponding standard morsifications given by the upper triangular matrices
\[
\wc{G}_{q_1}=\begin{pmatrix}1&\widehat{1}&\widehat{1}&2\\0&1&\widehat{1}&\widehat{1}\\0&0&1&\widehat{1}\\0&0&0&1\end{pmatrix} \qquad \text{and} \qquad
\wc{G}_{q_0}=\begin{pmatrix}1&\widehat{1}&2&\widehat{1}\\0&1&\widehat{1}&\widehat{1}\\0&0&1&\widehat{1}\\0&0&0&1\end{pmatrix}
\]
Consider also the following (connected, loop-less) quivers with corresponding incidence matrices,
\[
Q^1=\xymatrix{\bulito_1 \ar[d]_-{3} & \bulito_3 \ar@<.5ex>[l]^-{1} \ar@<-.5ex>[l]_-{4} \\ \bulito_2 \ar[ru]_-{2} } \quad I(Q^1)=\raisebox{-.5em}{$\left(\begin{smallmatrix} \widehat{1}&0&1&\widehat{1}\\0&1&\widehat{1}&0\\1&\widehat{1}&0&1 \end{smallmatrix}\right)$} \quad
Q^0=\xymatrix{\bulito_1 \ar[d]_-{4} & \bulito_3 \ar@<.5ex>[l]^-{1} \ar@<-.5ex>[l]_-{3} \\ \bulito_2 \ar[ru]_-{2} } \quad I(Q^0)=\raisebox{-.5em}{$\left(\begin{smallmatrix} \widehat{1}&0&\widehat{1}&1\\0&1&0&\widehat{1}\\1&\widehat{1}&1&0 \end{smallmatrix}\right)$}
\]
A direct calculation shows that $q_i=q_{Q^i}$ for $i=1,0$. By Theorem~\ref{MAINlem}, the quadratic forms $q_i$ are connected non-negative unit forms of Dynkin type $\A_2$ and corank $2$, satisfying $q_1 \sim q_0$ (indeed, if $\sigma_{3,4}$ is the transposition of $3$ and $4$ in $\{1,2,3,4\}$ then $q_1 \sim^{P(0\sigma_{3,4})} q_0$). The corresponding inverse quivers~(\ref{EQQ3}) are given as follows,
\[
(Q^1)^{\dagger}=\xymatrix{\bulito_1 & \bulito_3 \ar@<-.5ex>[l]_-{1} \ar@<.5ex>[l]^-{3}  \\ \bulito_2 \ar@<.5ex>[u]^-{4} \ar@<-.5ex>[u]_-{2}  } \quad  I((Q^1)^{\dagger})=\raisebox{-.5em}{$\left(\begin{smallmatrix} \widehat{1}&\widehat{1}&\widehat{1}&\widehat{1}\\0&1&0&1\\1&0&1&0 \end{smallmatrix}\right)$} \quad
(Q^0)^{\dagger}=\xymatrix{\bulito_1 & \bulito_3 \ar[l]_-{1}  \\ \bulito_2 \ar@<.5ex>[u]^-{4} \ar@<-.5ex>[u]_-{2} \ar[ru]_-{3} } \quad  I((Q^0)^{\dagger})=\raisebox{-.5em}{$\left(\begin{smallmatrix} \widehat{1}&\widehat{1}&0&\widehat{1}\\0&1&1&1\\1&0&\widehat{1}&0 \end{smallmatrix}\right).$}
\]
\end{example}

\subsection{Standard quivers}\label{SS:stan}

For any partition $\pi=(\pi_1,\ldots,\pi_{\ell})$ of an integer $m \geq 2$, and any non-negative integer $\dd$, consider the connected loop-less quivers $\vec{\A}^{\dd}[\pi]$ and $\vec{\Star}^{\dd}[\pi]$ with $m$ vertices and $n=m+\ell+2(\dd-1)$ arrows, given as follows (see~\cite[Definition~5.2]{jaJ2020b} and Figure~\ref{Fig} above).

\begin{definition}\label{DStan}
Fix $m \geq 2$, a partition $\pi=(\pi_1,\ldots,\pi_{\ell})$ of $m$ and an integer $\dd \geq 0$.  Take $v_t:=m-(\pi_1+\ldots+\pi_t)$ for $t=0,\ldots,\ell$.
\begin{itemize}
 \item[i)] Let $\vec{\A}_m$ be the quiver with $m$ vertices $(\vec{\A}_m)_0=\{1,\ldots,m\}$ and $m-1$ arrows $(\vec{\A}_m)_1=\{i_1,\ldots,i_{m-1}\}$ such that $\sou(i_t)=t$ and $\tar(i_t)={t+1}$ for $t=1,\ldots,m-1$. Then the quiver $\vec{\A}_m^{\dd}[\pi]=\vec{\A}^{\dd}[\pi]$ is obtained from $\vec{\A}_m$ by adding $\ell-1+2\dd$ arrows $j_1,\ldots,j_{\ell-1}, k_1,\ldots, k_{2\dd}$ in the following way. Define $\sou(j_t):=v_{t-1}$ and $\tar(j_t):=v_t$ for $t=1,\ldots,\ell-1$. Moreover, taking $\alpha:=i_{m-1}$ if $\ell=1$, and $\alpha:=j_{\ell-1}$ if $\ell>1$, define $\sou(k_t):=\sou(\alpha^{\epsilon_t})$ and $\tar(k_t):=\tar(\alpha^{\epsilon_t})$ for $t=1,\ldots,2\dd$, where $\epsilon_t=(-1)^t$.  The set of vertices $(\vec{\A}_m)_0$ has the natural order $1<2<\ldots <m$, and the set of arrows $(\vec{\A}_m)_1$ is ordered as follows: $i_1<i_2<\ldots < i_{m-1}<j_1<\ldots <j_{\ell-1}<k_1<\ldots < k_{2\dd}$.

 \item[ii)] Let $\vec{\Star}_m$ be the quiver with $m$ vertices $(\vec{\Star}_m)_0=\{1,\ldots,m\}$ and $m-1$ arrows $(\vec{\Star}_m)_1=\{i_1,\ldots,i_{m-1}\}$ such that $\sou(i_t)=1$ and $\tar(i_t)={t+1}$ for $t=1,\ldots,m-1$. Then the quiver $\vec{\Star}_m^{\dd}[\pi]=\vec{\Star}^{\dd}[\pi]$ is obtained from $\vec{\Star}_m$ by adding $\ell-1+2\dd$ arrows $j_1,\ldots,j_{\ell-1}, k_1,\ldots, k_{2\dd}$ in the following way. Define $\sou(j_t):=1$ and $\tar(j_t):=v_t+1$ for $t=1,\ldots,\ell-1$. Moreover, taking $\alpha:=i_{m-1}$ if $\ell=1$ and $\alpha:=j_{\ell-1}$ if $\ell>1$, define $\sou(k_t):=\sou(\alpha)=1$ and $\tar(k_t):=\tar(\alpha)$ for $t=1,\ldots,2\dd$. As before, the set of vertices $(\vec{\Star}_m)_0$ has the natural order $1<2<\ldots <m$, and the set of arrows $(\vec{\Star}_m)_1$ is ordered as follows: $i_1<i_2<\ldots < i_{m-1}<j_1<\ldots <j_{\ell-1}<k_1<\ldots < k_{2\dd}$.
\end{itemize}

The quiver $\vec{Q}=\vec{\A}_m^{\dd}[\pi]=\vec{\A}^{\dd}[\pi]$ constructed in $(i)$ is called \textbf{standard $(\pi,\dd)$-extension quiver} of cycle type $\pi$ and degeneracy degree $\dd$ (see Corollary~\ref{(C):dd}), or simply \textbf{standard quiver}. The corresponding quadratic form $q_{\vec{Q}}$ is referred to as \textbf{standard $(\pi,\dd)$-extension} of the unit form $q_{\A_{m-1}}$, or simply as \textbf{standard extension} of $q_{\A_{m-1}}$. The quiver $\vec{\Star}_m^{\dd}[\pi]=\vec{\Star}^{\dd}[\pi]$ constructed in $(ii)$ is the inverse of $\vec{Q}$ (see~\cite[Remark~5.3]{jaJ2020b}).
\end{definition}

Since we fixed linear orders on the sets of vertices and the set of arrows of $\vec{\A}^{\dd}[\pi]$ and $\vec{\Star}^{\dd}[\pi]$, these quivers fix incidence matrices, given explicitly in Remark~\ref{RInc} below.

\begin{corollary}\label{coro:STDquad}
For any partition $\pi$ of $m \geq 2$ and any non-negative integer $\dd$, the quadratic forms $q_{\vec{Q}}$ for $\vec{Q}\in \{\vec{\A}^{\dd}[\pi],\vec{\Star}^{\dd}[\pi] \}$ are non-negative unit forms of Dynkin type $\A_{m-1}$ and corank $\ell(\pi)+\dd-1$.
\end{corollary}
\begin{proof}
Apply Theorem~\ref{MAINlem}.
\end{proof}

In the following technical observation we show that the cycle type of $q_{\vec{Q}}$ for a standard quiver $\vec{Q}=\vec{\A}^{\dd}[\pi]$ is precisely $\pi$. One of its consequences, Corollary~\ref{STDunique} below, is used implicitly in the proof of~\cite[Theorem~6.3]{jaJ2020b}.

\begin{remark}\label{Rxi}
Let $\vec{Q}=\vec{\A}^{\dd}[\pi]$ be the standard $(\pi,\dd)$-extension quiver for a partition $\pi$ of $m\geq 2$ and $\dd \geq 0$. Recall from Definition~\ref{DStan} that the number of vertices (resp. arrows) of $\vec{Q}$ is $m$ (resp. $n=m+\ell(\pi)+2(\dd-1)$).
\begin{itemize}
 \item[a)] Consider the permutation of vertices $\xi:=\xi^-_{\vec{Q}}$ determined by $\vec{Q}$, take $\pi=(\pi_1,\ldots,\pi_{\ell})$ and $v_t:=m-(\pi_1+\ldots+\pi_t)$ for $t=0,\ldots,\ell$. Then
  \begin{equation*}
\xi(v)= \left\{ \begin{array}{l l}
v+1, & \text{if $v \notin \{v_0,v_1,\ldots,v_{\ell-2},v_{\ell-1}\}$},\\
v_{t+1}+1, & \text{if $v =v_t$ for some $t \in \{0,\ldots,\ell-1\}$}.
\end{array} \right.
\end{equation*}
 \item[b)] The cycle type of $q_{\vec{Q}}$ is $\pi$.
\end{itemize}
\end{remark}
\begin{proof}
By definition~(\ref{DEFperm}), we have $\xi(v)=\tar(\alpha^-_{\vec{Q}}(v))$.

Assume first that $\dd=0$. Denote by $i_v$ the unique arrow in $\vec{Q}$ with $\sou(i_v)=v$ and $\tar(i_v)=v+1$, and by $j_t$ the unique arrow in $\vec{Q}$ with $\sou(j_t)=v_{t-1}$ and $\tar(j_t)=v_t$ for $t=1,\ldots,\ell-1$ (these are all the arrows of $\vec{Q}$, since $\dd=0$). Then a direct calculation shows that if $v \notin \{v_0,v_1,\ldots,v_{\ell-1}\}$, then $\alpha_{\vec{Q}}^-(v)=i_v^{+1}$, and therefore $\xi(v)=\tar(i_v)=v+1$. Now, if $v=v_t$ for some $t\in \{0,1,\ldots,\ell-2\}$, then $\alpha_{\vec{Q}}^-(v)=j_{t+1}^{+1}i_{v_{t+1}}^{+1}$, and therefore $\xi(v_t)=\tar(i_{v_{t+1}})=v_{t+1}+1$. Moreover, if $v=v_{\ell-1}$, then
\[
\alpha_{\vec{Q}}^-(v)=j_{\ell-1}^{-1}j_{\ell-2}^{-1}\cdots j_{1}^{-1}i_{m-1}^{-1}i_{m-2}^{-1}\cdots i_2^{-1}i_1^{-1}.
\]
In particular, $\xi(v_{\ell})=\tar(i_1^{-1})=\sou(i_1)=1=v_{\ell}+1$. This shows $(a)$ for $\dd \geq 0$, since $\xi$ is unchanged by adding pairs of (anti-) parallel arrows (see~\cite[Remark~5.1]{jaJ2020b}).

\medskip
To show $(b)$, recall that $\ct(Q)$ is defined as the sequence of cardinalities of the $\xi$-orbits on $\vec{Q}_0$, ordered non-increasingly~(\ref{EQQct}), which equals $\pi$ by $(a)$.
\end{proof}

\begin{corollary}\label{STDunique}
Let $\pi$ and $\pi'$ be partitions of $m\geq 2$ and $m' \geq 2$ respectively, and take $\dd,\dd' \geq 0$.
\begin{itemize}
 \item[a)] We have $q_{\vec{\A}_m^{\dd}[\pi]} \sim q_{\vec{\A}_{m'}^{\dd'}[\pi']}$ if and only if $(m,\ell(\pi)+2\dd)=(m',\ell(\pi')+2\dd')$.
 \item[b)] The following conditions are equivalent:
  \begin{itemize}
   \item[b1)] $(\pi,\dd)=(\pi',\dd')$.
   \item[b2)] $\vec{\A}_m^{\dd}[\pi]=\vec{\A}_{m'}^{\dd'}[\pi']$.
   \item[b3)] $q_{\vec{\A}_m^{\dd}[\pi]}=q_{\vec{\A}_{m'}^{\dd'}[\pi']}$.
   \item[b4)] $q_{\vec{\A}_m^{\dd}[\pi]} \approx q_{\vec{\A}_{m'}^{\dd'}[\pi']}$
  \end{itemize}
\end{itemize}
\end{corollary}
\begin{proof}
Take $\vec{Q}=\vec{\A}_m^{\dd}[\pi]$ and $\vec{Q}'=\vec{\A}_{m'}^{\dd'}[\pi']$, and recall that $|\vec{Q}_0|=m$ and $|\vec{Q}_1|=m+\ell(\pi)+2(\dd-1)$ (and similarly, $|\vec{Q}'_0|=m'$ and $|\vec{Q}'_1|=m'+\ell(\pi')+2(\dd'-1)$, see Definition~\ref{DStan}). Then $(a)$ holds, since $q_{\vec{Q}} \sim q_{\vec{Q}'}$ if and only if $(|\vec{Q}_0|,|\vec{Q}_1|)=(|\vec{Q}'_0|,|\vec{Q}'_1|)$ by Theorem~\ref{MAINlem}$(ii)$, since both $\vec{Q}$ and $\vec{Q}'$ are loop-less and connected quivers.

\medskip
To show $(b)$, observe that $(b1)$ implies $(b2)$ by Defnition~\ref{DStan}, and that evidently $(b4)$ follows from $(b3)$, which follows from $(b2)$.  To complete the proof assume that $(b4)$ holds. In particular, $q_{\vec{Q}} \sim q_{\vec{Q}'}$, and by $(a)$ we have $m=m'$ and $\ell(\pi)+2\dd=\ell(\pi')+2\dd'$. Moreover, $\ct(q_{\vec{Q}})=\ct(q_{\vec{Q}'})$ by~\cite[Theorem~6.3]{jaJ2020b}, and by Remark~\ref{Rxi} we have $\pi=\ct(\vec{Q})=\ct(\vec{Q}')=\pi'$. Then $\ell(\pi)=\ell(\pi')$, which shows that $(\pi,\dd)=(\pi',\dd')$, that is, $(b1)$ holds.
\end{proof}

Similar claims, not needed for our discussion, hold for the inverse quivers $\vec{\Star}^{\dd}[\pi]$. In the following result, the mentioned non-negative integer $\dd$ is the so-called \emph{degree of degeneracy} of $q_Q$, cf.~\ref{reRad}  and Corollaries~\ref{CreC},~\ref{(C):dd} below.

\begin{remark}\label{RStan}
Let $Q$ be a connected loop-less quiver. Then there is a unique standard quiver $\vec{Q}$ with the same number of vertices and arrows as $Q$ such that $\ct(\vec{Q})=\ct(Q)$, namely, $\vec{Q}=\vec{\A}^{\dd}[\ct(Q)]$ for some integer $\dd>0$. Moreover, in this case there is a permutation $\rho$ of $Q_0$ such that $\Lambda_{\rho\cdot Q}=\Lambda_{\vec{Q}}$.
\end{remark}
\begin{proof}
Take $\pi=\ct(Q)$ and $\ell=\ell(\pi)$. Recall that if $n=|Q_1|$ and $m=|Q_0|$ then the corank of $q_Q$ is given by $c=n-(m-1)$ (see Theorem~\ref{MAINlem}).  Since $\pi \in \Part_1^c(m)$ by~\cite[Proposition~4.5]{jaJ2020b}, then there is a non-negative integer $\dd$ such that $2\dd=c-(\ell-1)=n-(m-1)-(\ell-1)$. Take $\vec{Q}=\vec{\A}^{\dd}[\pi]$. Since $\pi$ is a partition of $m$, the standard quiver $\vec{Q}$ has $m$ vertices, and by Definition~\ref{DStan} it also has $m+\ell+2(\dd-1)=n$ arrows. By Remark~\ref{Rxi}, $\ct(\vec{Q})=\pi=\ct(Q)$. The uniqueness of $\vec{Q}$ follows from Corollary~\ref{STDunique}$(b)$.

\medskip
To show the existence of such permutation $\rho$, recall from Remark~\ref{LAMBDA}$(b)$ that there are permutations $\xi$ and $\vec{\xi}$ of the sets $Q_0$ and $\vec{Q}_0$ such that $\Lambda_Q=P(\xi)$ and $\Lambda_{\vec{Q}}=P(\vec{\xi})$, and therefore $\ct(\xi)=\ct(Q)=\ct(\vec{Q})=\ct(\vec{\xi})$, cf.~(\ref{EQQct}). By Lemma~\ref{permutation}, $\xi$ and $\vec{\xi}$ are conjugate permutations, that is, there is a permutation $\rho$ such that $\vec{\xi}=\rho\xi \rho^{-1}$, or matricially,
\[
\Lambda_{\vec{Q}}=P(\vec{\xi})=P(\rho)P(\xi)P(\rho^{-1})=P(\rho)\Lambda_{Q}P(\rho^{-1}).
\]
Using Remark~\ref{RPerm}$(i,ii)$ and the definition of Coxeter-Laplacian~(\ref{EQQCL}), we conclude that
\begin{eqnarray*}
\Lambda_{\rho \cdot Q} &=& \Id_m-I(\rho \cdot Q)\wc{G}_{\rho \cdot Q}I(\rho \cdot Q)^{\tr}=\Id_m-P(\rho) I(Q)\wc{G}_{Q}I(Q)^{\tr}P(\rho)^{\tr} \\
&=& P(\rho)[\Id_m-I(Q)\wc{G}_{Q}I(Q)^{\tr}]P(\rho^{-1})=P(\rho)\Lambda_QP(\rho^{-1})=\Lambda_{\vec{Q}}.
\end{eqnarray*}

\vspace*{-6mm}
\end{proof}

\begin{example}\label{EXAone}
The Coxeter-Laplacians $\Lambda_{Q^i}$ of the quivers $Q^1$ and $Q^0$ of Example~\ref{EXAzero} are given by
\[
\Lambda_{Q^1}=\begin{pmatrix} 0&0&1\\1&0&0\\0&1&0\end{pmatrix} \qquad \text{and} \qquad
\Lambda_{Q^0}=\begin{pmatrix} 1&0&0\\0&1&0\\0&0&1\end{pmatrix},
\]
see~(\ref{EQQCL}). In particular, the cycle types of the quadratic forms $q_i$ are $\ct(q_1)=(3)$ and $\ct(q_0)=(1,1,1)$, cf.~(\ref{EQQct}) and~(\ref{EQQ5}). The standard quivers associated to $Q^1$ and $Q^0$, as in Remark~\ref{RStan}, are given as follows,
\[
\vec{Q}^1=\xymatrix{\bulito_1 \ar[r]^-{1} & \bulito_2 \ar[r]^-{2} \ar@<-1ex>@/_7pt/@{<-}[r]^-3 \ar@<-2.5ex>@/_7pt/[r]_-4 & \bulito_3 } \quad I(\vec{Q}^1)=\raisebox{-.5em}{$\left(\begin{smallmatrix} 1&0&0&0\\\widehat{1}&1&\widehat{1}&1\\0&\widehat{1}&1&\widehat{1} \end{smallmatrix}\right)$} \quad
\vec{Q}^0=\xymatrix{\bulito_1 \ar[r]^-{1} & \bulito_2 \ar[r]^-{2} \ar@<-.5ex>@/_7pt/@{<-}[r]_-3 \ar@<.7ex>@/^7pt/[l]^-4 & \bulito_3 } \quad I(\vec{Q}^0)=\raisebox{-.5em}{$\left(\begin{smallmatrix} 1&0&0&\widehat{1}\\\widehat{1}&1&\widehat{1}&1\\0&\widehat{1}&1&0 \end{smallmatrix}\right)$}
\]
Note that $\Lambda_{Q^i}=\Lambda_{\vec{Q}^i}$ for $i=1,0$. The corresponding inverse quivers~(\ref{EQQ3}) are given by
\[
(\vec{Q}^1)^{\dagger}=\xymatrix@R=1pc{& \bulito_2 \\ \bulito_1 \ar[ru]^-1 \ar[r]^-2 \ar@<-1ex>@/_5pt/[r]^-{3} \ar@<-2ex>@/_5pt/[r]_-{4} & \bulito_3 } \quad I((\vec{Q}^1)^{\dagger})=\raisebox{-.5em}{$\left(\begin{smallmatrix} 1&1&1&1\\\widehat{1}&0&0&0\\0&\widehat{1}&\widehat{1}&\widehat{1} \end{smallmatrix}\right)$} \quad
(\vec{Q}^0)^{\dagger}=\xymatrix{& \bulito_2 \\ \bulito_1 \ar[ru]^-1 \ar[r]^-2 \ar@<-1ex>@/_5pt/[r]^-{3} \ar@<1ex>@/^7pt/[ru]^-{4} & \bulito_3 } \quad I((\vec{Q}^0)^{\dagger})=\raisebox{-.5em}{$\left(\begin{smallmatrix}1&1&1&1\\\widehat{1}&0&0&\widehat{1}\\0&\widehat{1}&\widehat{1}&0 \end{smallmatrix}\right)$}
\]
\end{example}

It will be convenient to have explicit formulas for the incidence matrices of the standard quivers and their inverses. The proof of the following remark is clear from Definition~\ref{DStan}.

\begin{remark}\label{RInc}
 Take a partition $\pi=(\pi_1,\ldots, \pi_{\ell})$ of $m \geq 2$, an integer $\dd \geq 0$, and denote by $\vec{Q}$ the standard quiver $\vec{\A}^{\dd}[\pi]$. Take $n=m-1+2\dd +(\ell-1)$ and $v_t:=m-(\pi_1+\ldots+\pi_t)$ for $t=0,\ldots,\ell$.
\begin{itemize}
 \item[a)] If $\ell=1$, then $n=m-1+2\dd$. Taking $x:=\bas_{m-1}-\bas_{m}$ we have
 \[
 I(\vec{Q})=[I(\vec{\A}_m),\underbrace{-x,x,\cdots ,-x,x}_\text{$2\dd$ columns}].
 \]
 If $\ell >1$, taking $x_t:=\bas_{v_{t-1}}-\bas_{v_t}$ for $t=1,\ldots,\ell-1$ we have
 \[
 I(\vec{Q})=[I(\vec{\A}_m),x_1,x_2,\cdots,x_{\ell-1},\underbrace{-x_{\ell-1},x_{\ell-1},\cdots ,-x_{\ell-1},x_{\ell-1}}_\text{$2\dd$ columns}].
 \]

 \item[b)] If $\ell=1$, for $v \in \{1,\ldots,m\}$ we have
  \begin{equation*}
I(\vec{Q}^{\dagger})^{\tr}\bas_v= \left\{ \begin{array}{l l}
\mathbbm{1}_n, & \text{if $v=1$},\\
-\bas_{v-1}, & \text{if $1<v<m$},\\
-\bas_{m-1}-\sum_{j=m}^n\bas_j, & \text{if $v=m$}.
\end{array} \right.
\end{equation*}
  If $\ell>1$, take $z_t:=\bas_{v_t}+\bas_{m-1+t}$ for $t=1,\ldots,\ell-1$. For $v \in \{1,\ldots,m \}$ we have
 \begin{equation*}
I(\vec{Q}^{\dagger})^{\tr}\bas_v= \left\{ \begin{array}{l l}
\mathbbm{1}_n, & \text{if $v=1$},\\
-\bas_{v-1}, & \text{if $v \notin \{v_t+1\}_{t=1,\ldots,\ell}$},\\
-z_t, & \text{if $v=v_t+1$ for some $t=1,\ldots,\ell-2$}, \\
-z_{\ell-1}- \sum_{j=m+\ell-1}^{n}\bas_{j}, & \text{if $v=v_{\ell-1}+1$}.
\end{array} \right.
\end{equation*}
\end{itemize}
\end{remark}

For instance, the standard quivers include all (inverse) generalized Kronecker quivers, as indicated in the following useful observation.

\begin{remark}\label{RKro}
For $n \geq 1$, consider the (generalized) Kronecker quiver $\K_n$ with two vertices and $n$ arrows in the same direction,
\[
\K_n=\xymatrix{\bulito \ar@<5ex>[r]^-{1} \ar@<2.5ex>[r]^-{2} \ar@{}[r]|-{\cdots} \ar@<-2.5ex>[r]^-{n-1} \ar@<-5ex>[r]^-{n} & \bulito}; \qquad \K_{2n'}^{\dagger}=\xymatrix{\bulito \ar@<5ex>[r]^-{1} \ar@<2.5ex>@{<-}[r]^-{2} \ar@{}[r]|-{\cdots} \ar@<-2.5ex>[r]^-{2n'-1} \ar@<-5ex>@{<-}[r]^-{2n'} & \bulito}; \qquad \K_{2n'+1}^{\dagger}=\xymatrix{\bulito \ar@<5ex>[r]^-{1} \ar@<2.5ex>@{<-}[r]^-{2} \ar@{}[r]|-{\cdots} \ar@<-2.5ex>@{<-}[r]^-{2n'} \ar@<-5ex>[r]^-{2n'+1} & \bulito}.
\]
Observe that if $n$ is even ($n=2n'$ for $n' >0$) then $\K_n=\vec{\Star}^{n'-1}[(1,1)]$ and its inverse is given by $\K_n^{\dagger}=\vec{\A}^{n'-1}[(1,1)]$. On the other hand, if $n$ is odd ($n=2n'+1$ for $n' \geq 0$) then $\K_n=\vec{\Star}^{n'}[(2)]$ and its inverse is given by $\K_n^{\dagger}=\vec{\A}^{n'}[(2)]$. Moreover, assume that $n=2\dd+1$ for some integer $\dd\geq 1$.
\begin{itemize}
 \item[a)] Take $b_t=\bas_t+\bas_{t+1}\in \Z^n$ and $b'_t=-\bas_t+\bas_{t+1}\in \Z^n$ for $t=1,\dots,2\dd$. Then the set $\{b_1,\ldots b_{2\dd}\}$ is a basis of the kernel of $I(\K_n^{\dagger})$, and for $1\leq t,t'\leq n$ we have
 \begin{equation*}
b^{\tr}_tb'_{t'}= \left\{ \begin{array}{l l}
1, & \text{if $t'=t-1$},\\
-1, & \text{if $t'=t+1$},\\
0, & \text{otherwise}.
\end{array} \right.
\end{equation*}

 \item[b)] Take $c_t=(-1)b_t \in \Z^n$ if $t \in \{1,\ldots,2\dd\}$ is even, and $c_{2u+1}=\sum_{r=0}^ub_{2r+1} \in \Z^n$ if $t=2u+1 \in \{1,\ldots,2\dd\}$ is odd, and let $K=[c_1,\dots,c_{2\dd}]$ be the $n \times 2\dd$ matrix with columns the vectors $c_1,\ldots,c_{2\dd}$. Then $K$ is a kernel matrix for $I(\K_n^{\dagger})$, and
 \[
 K^{\tr}\wc{G}_{\K_n^{\dagger}}K=\underbrace{W_1 \oplus \ldots \oplus W_1}_{\text{$\dd$ times}}, \qquad \text{where $W_1=\left( \begin{matrix} 0&1\\-1&0\end{matrix}\right)$.}
 \]
\end{itemize}
\end{remark}
\begin{proof}
The shape of the corresponding standard quiver $\vec{\A}^{n'-1}[(1,1)]$ and $\vec{\A}^{n'}[(2)]$ is clear from definition. That they are the inverse of the Kronecker quiver $\K_n$ (for $n=2n'$ and $n=2n'+1$ respectively), follows from~\cite[Remark~5.3]{jaJ2020b}.

\medskip
Claim $(a)$ is straightforward. In particular, the matrix $K$ of the point $(b)$ is a kernel matrix of $I(\K_n^{\dagger})$.  Observe that the $(2\dd+1) \times (2\dd+1)$ matrix $\wc{G}_{\K_n^{\dagger}}$ is given by
\[
\wc{G}_{\K_n^{\dagger}}=\left( \begin{matrix} 1&\widehat{2}&2&\widehat{2}&2&\cdots & 2 \\ 0&1&\widehat{2}&2&\widehat{2}&\cdots & \widehat{2} \\ 0&0&1&\widehat{2}&2&\cdots &2 \\ 0&0&0&1&\widehat{2}&\cdots &\widehat{2} \\ 0&0&0&0&1&\cdots &2 \\ \vdots & \vdots & \vdots & \vdots & \vdots & \ddots & \vdots \\ 0 & 0 & 0 & 0 & 0 & \cdots & 1 \end{matrix}\right),
\]
where $\widehat{2}:=-2$, and that $b'_t=\wc{G}_{\K_n^{\dagger}}b_t$ for $t=1,\ldots,2\dd$. Then, for $1\leq t,t'\leq n$ we have
 \begin{equation}\label{Kone}
c^{\tr}_t\wc{G}_{\K_n^{\dagger}}c_{t'}= \left\{ \begin{array}{l l}
1, & \text{if $t$ is odd and $t'=t+1$},\\
-1, & \text{if $t$ is even and $t'=t-1$},\\
0, & \text{otherwise}.
\end{array} \right.
\end{equation}
Indeed, if both $t$ and $t'$ are even, then by the definition $c_t=(-1)b_t$ and using $(a)$, we get $c^{\tr}_t\wc{G}_{\K_n^{\dagger}}c_{t'}=0$. A similarly claim holds if both $t$ and $t'$ are odd. Assume that $t$ is even and $t'$ is odd. Using $(a)$ we get
 \begin{equation*}
c^{\tr}_t\wc{G}_{\K_n^{\dagger}}c_{t'}=(-1)b_t^{\tr}(b'_1+b'_3+\ldots +b'_{t'})= \left\{ \begin{array}{l l}
(-1)b_t^{\tr}b'_{t-1}+(-1)b_t^{\tr}b'_{t+1}=0, & \text{if $t<t'$},\\
(-1)b_t^{\tr}b'_{t-1}=-1, & \text{if $t'=t-1$},\\
0, & \text{otherwise}.
\end{array} \right.
\end{equation*}
Similarly, if $t$ is odd and $t'$ is even, then
 \begin{equation*}
c^{\tr}_t\wc{G}_{\K_n^{\dagger}}c_{t'}=(b_1+b_3+\ldots +b_{t})^{\tr}(-1)b'_{t'}= \left\{ \begin{array}{l l}
b_{t'-1}^{\tr}(-1)b'_{t'}+b_{t'+1}^{\tr}(-1)b'_{t'}=0, & \text{if $t>t'$},\\
b_t^{\tr}(-1)b'_{t+1}=1, & \text{if $t'=t+1$},\\
0, & \text{otherwise}.
\end{array} \right.
\end{equation*}
These identities show equation~(\ref{Kone}), which is a coefficientwise expression of the direct sum of $\dd$ copies of $W_1$. This completes the proof.
\end{proof}

\subsection{The Coxeter-Laplacian}\label{CoxLap}

Our proposed solution to Problem~2$(ii)$ starts with an explicit combinatorial construction, that uses the structural walks of $Q$ given above, cf. Step~1 on page~\pageref{StepOne}.

\begin{proposition}\label{Ls}
Let $Q$ be a connected loop-less quiver with $m \geq 2$ vertices and $n \geq 1$ arrows, and such that $\Lambda_Q=\Lambda_{\vec{Q}}$ for a standard quiver $\vec{Q}$ with the same number of vertices and arrows as $Q$. Then there is a (not necessarily $\Z$-invertible) $n \times n$ matrix $B$ such that
\[
I(Q)B=I(\vec{Q}) \qquad \text{and} \qquad I(\vec{Q}^{\dagger})B^{\tr}=I(Q^{\dagger}).
\]
\end{proposition}
\begin{proof}
If such standard quiver $\vec{Q}$ exists then $\ct(Q)=\ct(\vec{Q})$, see~(\ref{EQQct}) and~(\ref{EQQ5}). Therefore, $\vec{Q}$ is the unique standard quiver $\vec{Q}=\vec{\A}^{\dd}[\pi]$ with $\pi=\ct(\vec{Q})$, see Remark~\ref{RStan} and Corollary~\ref{STDunique}. Consider the vectors $a_v^{\pm}:=\inc(\alpha_Q^{\pm}(v))$ in $\Z^n$ as in Lemma~\ref{La}, with $n=m-1+2\dd+(\ell-1)$. By hypothesis we have $\Lambda_Q=\Lambda_{\vec{Q}}$, which implies that $\xi:=\xi^-_Q=\xi^-_{\vec{Q}}$ by~(\ref{EQQ5}). Consider the description of $\xi$ given in Remark~\ref{Rxi}.

\noindent \textbf{Case $\ell=1$.} In this case, we have necessarily $\pi=(m)$. Take the matrices
\begin{eqnarray}
B_1^1&=&[a_1^-,a_2^-,\cdots,a_{m-1}^-], \nonumber \\
B_3&=&[\underbrace{-a_{m-1}^-,a_{m-1}^-,\cdots,-a_{m-1}^-,a_{m-1}^-}_\text{$2\dd$ columns}], \quad\text{and} \nonumber \\
B&=&[B_1^1,B_3]. \label{EQQa}
\end{eqnarray}
Clearly, $B$ is a $n \times n$ matrix since $n=m-1+2\dd$. In the following steps we show that $B$ satisfies the wanted conditions.
\begin{itemize}
 \item[] \textbf{Step~1.} We have $I(Q)B=I(\vec{Q})$. Indeed, recall from Remark~\ref{RInc}$(a)$ that $I(\vec{Q})$ has the following shape
\[
I(\vec{Q})=[I(\vec{\A}_m),\underbrace{-x,x,\cdots,-x,x}_\text{$2\dd$ columns}],
\]
where $x=\bas_{m-1}-\bas_m$.

 By definition of $\alpha_Q^-(v)$ we have $\sou(\alpha_Q^-(v))=v$, and by~(\ref{DEFperm}), $\tar(\alpha_Q^-(v))=\xi(v)$ for any $v \in \{1,\ldots,m\}$. By Remark~\ref{Rxi}, $\xi(v)=v+1$ if $v<m$, and $\xi(m)=1$. Thus, using~(\ref{EQQ2}), for $1 \leq i=v < m$ we get
 \[
I(Q)B\bas_i= I(Q)a^-_v=I(Q)\inc(\alpha^-_Q(v))=\bas_v-\bas_{v+1}=I(\vec{\A}_m)\bas_i.
 \]
Moreover, $I(Q)B\bas_{m-1}=I(Q)a^-_{m-1}=\bas_{m-1}-\bas_{m}=x$, and for $m \leq i \leq n=m-1+2\dd$ we have
\[
I(Q)B\bas_i= \pm I(Q)a^-_{m-1}=\pm x,
\]
with signs corresponding to the parity of $i$. Altogether, for $1 \leq i \leq n$ we have $I(Q)B\bas_i=I(\vec{Q})\bas_i$, hence our claim.

 \item[] \textbf{Step~2.} We have $BI(\vec{Q}^{\dagger})^{\tr}=I(Q^{\dagger})^{\tr}$. Indeed, using Remark~\ref{RInc}$(b)$ and Lemma~\ref{La}, we get
 \[
 BI(\vec{Q}^{\dagger})^{\tr}\bas_1=B\mathbbm{1}_n=\sum_{i=1}^{m-1}a^-_i+\sum_{j=1}^{\dd}(a^-_{m-1}-a^-_{m-1})=-a_m^-=a_{\xi(m)}^+=a_1^+.
 \]
 Moreover, for $1<i=v<m$ we have
 \[
 BI(\vec{Q}^{\dagger})^{\tr}\bas_i=-B\bas_{v-1}=-a_{v-1}^-=a_{\xi(v-1)}^+=a_v^+.
 \]
Using again Lemma~\ref{La}$(i)$ we get
 \[
 BI(\vec{Q}^{\dagger})^{\tr}\bas_m=-B\bas_{m-1}-\sum_{j=m}^{n}B\bas_j=-a_{m-1}^-=a_{\xi(m-1)}^+=a_m^+.
 \]
By Lemma~\ref{La}$(ii)$ we conclude that $BI(\vec{Q}^{\dagger})^{\tr}\bas_i=a^+_i=I(Q^{\dagger})^{\tr}\bas_i$ for $1 \leq i \leq n$. Hence, the case $\ell=1$ holds.
\end{itemize}

\noindent \textbf{Case $\ell>1$.} For $t=1,\ldots,\ell-1$ fix arbitrary walks $\delta(t)$ from $v_{t}$ to $v_{t}+1$ (recall that $v_t:=m-(\pi_1+\ldots+\pi_t)$ for $t=0,\ldots,m$), and take $d_t=\inc(\delta(t))$. Consider the matrices
\begin{eqnarray}
B_1^t&=&[\underbrace{a_{v_{\ell+1-t}+1}^-,a_{v_{\ell+1-t}+2}^-,\cdots,a_{v_{\ell-t}-1}^-}_\text{$\pi_{\ell+1-t}-1$ columns}], \qquad \text{for $t=1,\ldots,\ell$,} \nonumber \\
B_1&=&[\underbrace{B_1^1,d_{\ell-1},B_1^2,d_{\ell-2},B_1^3,\cdots ,B_1^{\ell-1},d_{1},B_1^{\ell}}_\text{$m-1$ columns}], \nonumber \\
B_2&=&[\underbrace{y_1,y_2,\cdots,y_{\ell-1}}_\text{$\ell-1$ columns}], \nonumber \\
B_3&=&[\underbrace{-y_{\ell-1},y_{\ell-1},\cdots,-y_{\ell-1},y_{\ell-1}}_\text{$2\dd$ columns}], \quad\text{and} \nonumber \\
B&=&[B_1,B_2,B_3], \label{EQQb}
\end{eqnarray}
where $y_t:=a_{v_{t-1}}^--d_t=\inc(\alpha_Q^-(v_{t-1})\delta(t)^{-1})$ for $t=1,\ldots,\ell-1$. Clearly, $B$ is a $n \times n$ matrix since $n=m-1+2\dd+(\ell-1)$. Again, we show that $B$ satisfies the wanted conditions in two steps.
\begin{itemize}
 \item[] \textbf{Step~1.} We have $I(Q)B=I(\vec{Q})$. By construction, the $v$-th column of $B_1$ (for $1\leq v<m$) is given by $\inc(\gamma)$ for a walk $\gamma$ in $Q$ with $\sou(\gamma)=v$ and $\tar(\gamma)=v+1$ (indeed, if $v=v_t$ for some $t=1,\ldots,\ell-1$, then $\gamma=\delta(t)$, otherwise $\gamma=\alpha_Q^-(v)$). In particular, using~(\ref{EQQ2}),
 \[
 I(Q)B_1=I(\vec{\A}_m).
 \]
 Again by~(\ref{EQQ2}), for $t=1,\ldots,\ell-1$ we have $I(Q)y_t=I(Q)\inc(\alpha_Q^-(v_{t-1})\delta(t)^{-1})=\bas_{v_{t-1}}-\bas_{v_t}=x_t$, which implies that
 \[
 I(Q)B=[I(\vec{\A}_m),x_1,\cdots,x_{\ell-1},\underbrace{-x_{\ell-1},x_{\ell-1},\cdots,-x_{\ell-1},x_{\ell-1}}_\text{$2\dd$ columns}]=I(\vec{Q}),
 \]
 by Remark~\ref{RInc}$(a)$.

 \item[] \textbf{Step~2.} We have $BI(\vec{Q}^{\dagger})^{\tr}=I(Q^{\dagger})^{\tr}$. Indeed, using Remark~\ref{RInc}$(b)$ and Lemma~\ref{La} we get
 \begin{eqnarray*}
 BI(\vec{Q}^{\dagger})^{\tr}\bas_1&=&B\mathbbm{1}_n=\sum_{\substack{v=1,\ldots,m\\ v \neq v_1,\ldots,v_{\ell}}}a^-_v+\sum_{t=1}^{\ell-1}d_t+\sum_{t=1}^{\ell-1}(a_{v_{t-1}}^--d_t)+\sum_{j=1}^{\dd}(y_{\ell-1}-y_{\ell-1}) \\
 &=& \sum_{\substack{v=1,\ldots,m\\v \neq v_{\ell-1}}}a_v^-=-a_{v_{\ell-1}}^-=a_{\xi(v_{\ell-1})}^+=a_{v_{\ell}+1}^+=a_1^+,
 \end{eqnarray*}
since $v_{\ell}=0$. Moreover, for $v \in \{1,\ldots,m\}- \{v_t+1\}_{t=1,\ldots,\ell}$ we have
 \[
 BI(\vec{Q}^{\dagger})^{\tr}\bas_v=-B\bas_{v-1}=-a_{v-1}^-=a_{\xi(v-1)}^+=a_v^+.
 \]
If $v=v_{t}+1$ for some $t \in \{1,\ldots,{\ell-2}\}$, we have (recall that $z_t:=\bas_{v_t}+\bas_{m-1+t}$ for $t=1,\ldots,\ell-1$),
 \begin{eqnarray*}
 BI(\vec{Q}^{\dagger})^{\tr}\bas_v &=& -Bz_t=-B(\bas_{v_{t-1}}+\bas_{m-1+t})=-d_{t}-(a_{v_{t-1}}^--d_t)\\
 &=&-a_{v_{t-1}}^-=a_{\xi(v_{t-1})}^+=a_{v_t+1}^+=a_v^+.
 \end{eqnarray*}
Finally, \vspace*{-2mm}
 \begin{eqnarray*}
 BI(\vec{Q}^{\dagger})^{\tr}\bas_{v_{\ell-1}+1}&=&-Bz_{\ell-1}-\sum_{j=m+\ell-1}^{n}B\bas_j\\[-5pt]
 &=&-B(\bas_{v_{\ell-1}}+\bas_{m+\ell-2})-\sum_{j=1}^{\dd}(y_{\ell-1}-y_{\ell-1}) \\[-3pt]
 &=&-d_{\ell-1}-(a_{v_{\ell-2}}^--d_{\ell-1})=a_{v_{\ell-2}}^-=a_{\xi(v_{\ell-2})}^+=a_{v_{\ell-1}+1}^+.
 \end{eqnarray*}
We conclude that $BI(\vec{Q}^{\dagger})^{\tr}\bas_v=a^+_v=I(Q^{\dagger})^{\tr}\bas_v$ for all $1 \leq v \leq m$, by Lemma~\ref{La}$(ii)$.
\end{itemize}

\vspace*{-6mm}
\end{proof}

\begin{example}\label{EXAtwo}
Let us apply Proposition~\ref{Ls} (keeping the notation of its proof) to the running Example~\ref{EXAzero}.  First we take the descending structural walks of quiver $Q^1$:
\[
\xy 0;/r.20pc/:
(-68, 10)="Fr1" *{};
( 68, 10)="Fr2" *{};
( 68,-10)="Fr3" *{};
(-68,-10)="Fr4" *{};
(-33, 8)="F1" *{\alpha^-_{Q^1}(1)=4^{-1}2^{-1},};
(-33, 0)="F2" *{\alpha^-_{Q^1}(2)=3^{-1}1^{-1},};
(-26, -8)="F3" *{\alpha^-_{Q^1}(3)=4^{+1}3^{+1}2^{+1}1^{+1},};
(40, 0)="T1" *{B'_1=\begin{pmatrix} 0&\widehat{1}&1&\widehat{1}\\\widehat{1}&0&0&0\\ 0&\widehat{1}&1&\widehat{1}\\\widehat{1}&0&0&0 \end{pmatrix}.};
\endxy
\]
The cycle type of $Q^1$ has length $\ell(\ct(Q^1))=1$ (cf.~Example~\ref{EXAone}). Then the matrix $B_1'$ obtained in Proposition~\ref{Ls}, given by $B_1'=[a^-_1,a_2^-,-a_2^-,a_2^-]$, is shown explicitly above. On the other hand, the cycle type of $Q^0$ has length $\ell(\ct(Q^0))=3$, and its descending structural walks are given by
\[
\xy 0;/r.20pc/:
(-68, 10)="Fr1" *{};
( 68, 10)="Fr2" *{};
( 68,-10)="Fr3" *{};
(-68,-10)="Fr4" *{};
(-30, 8)="F1" *{\alpha^-_{Q^0}(1)=4^{+1}2^{+1}1^{+1},};
(-30, 0)="F2" *{\alpha^-_{Q^0}(2)=4^{-1}3^{-1}2^{-1},};
(-33, -8)="F3" *{\alpha^-_{Q^0}(3)=3^{+1}1^{-1},};
(40, 0)="T1" *{B'_0=\begin{pmatrix} 0&0&\widehat{1}&0\\ 0&1&\widehat{1}&\widehat{1}\\ 0&0&1&\widehat{1}\\1&0&0&\widehat{2} \end{pmatrix}.};
\endxy
\]
We need arbitrary walks $\delta(t):v_t \to v_t+1$ in $Q^0$ for $t=1,\ldots,\ell-1=2$, which we choose to be $\delta(1)=2^{+1}$ and $\delta(2)=4^{+1}$ (note that $v_0=3$, $v_1=2$, $v_2=1$ and $v_3=0$, since $\ct(Q^0)=(1,1,1)$). The matrix $B_0'$ obtained in Proposition~\ref{Ls}, given by $B_0'=[d_2,d_1,a^-_3-d_1,a^-_2-d_2]$, is shown above. A direct computation shows that $I(Q^i)B'_i=I(\vec{Q}^i)$ and $I((\vec{Q}^i)^{\dagger})(B'_i)^{\tr}=I((Q^i)^{\dagger})$ for $i=1,0$, see Example~\ref{EXAone} for the description of quivers $\vec{Q}^i$ and their inverses.
\end{example}

\section{Radicals and invertibility assumption} \label{S3}

In this section we analyze some strong Gram invariants within the radical of a non-negative unit form (Lemmas~\ref{(P):radre} and~\ref{(P):basis}) in order to prove our invertibility-correction algorithm Proposition~\ref{(P):inv} (see Algorithm~\ref{A2}).

\subsection{The reduced radical}\label{reRad}

Recall that a subgroup $X$ of $\Z^m$ is called \textbf{pure} if whenever $ax \in X$ for some $x \in \Z^m$ and some non-zero $a \in \Z$, then $x \in X$.  For any unit form $q$ consider the following subgroups of $\Z^n$,
\begin{eqnarray}
\rad(q) & = & \{ x \in \Z^n \mid  \text{$y^{\tr}G_qx= 0$ for all $y$ in $\Z^n$} \}=\{x \in \Z^n \mid  \text{$G_qx= 0$}\}, \nonumber \\
\rad_{re}(q) & = & \{ x \in \rad(q) \mid \text{$y^{\tr}\wc{G}_qx= 0$ for all $y \in \rad(q)$} \}. \nonumber
\end{eqnarray}

These are pure subgroups of $\Z^n$. The group $\rad(q)$ is called the \textbf{radical} of $q$, and we will refer to $\rad_{re}(q)$ as the \textbf{reduced radical} of $q$. The rank $\CRnk_{re}(q)$ of the reduced radical $\rad_{re}(q)$ of $q$ will be called \textbf{reduced corank} of $q$. The restriction of the standard morsification $\wc{\mathbbm{b}}_q$ to the radical of $q$ is denoted by $\wc{\mathbbm{r}}_q$. To be precise, let $k:\rad(q) \to \Z^n$ be the inclusion of the radical of $q$ in $\Z^n$, and take
\[
\wc{\mathbbm{r}}_q(x,y):=\wc{\mathbbm{b}}_q(k(x),k(y)) \quad \text{for $x,y \in \rad(q)$.}
\]
Clearly, $\wc{\mathbbm{r}}_q$ is a skew-symmetric bilinear form. In particular, its rank is an even non-negative number, $\Rnk(\wc{\mathbbm{r}}_q):=2\dd$ for some $\dd  \geq 0$ (cf.~\cite[XI, \S 4]{frG}), and we call $\dd_q:=\dd$ the \textbf{degree of degeneracy of $q$}.\label{LABdd} Note that
\begin{equation}\label{EQQdd}
\CRnk(q)=2\dd_q+\CRnk_{re}(q).
\end{equation}
Observe also that the reduced corank $\CRnk_{re}(q)$ and the degree of degeneracy $\dd_q$ are strong Gram invariants of $q$ (see Lemma~\ref{LinDet} below). To the best of the author's knowledge, these notions are new in the literature on integral bilinear and quadratic forms. Fixing a $\Z$-basis of $\rad(q)$ we get matricial forms $K$ of $k$ and $W_q$ of $\wc{\mathbbm{r}}_q$ given by
\begin{equation}\label{EQQW}
W_q=K^{\tr}\wc{G}_qK.
\end{equation}

Let us illustrate these notions.

\begin{example}\label{EXAthree}
Consider the quadratic forms $q^1$ and $q^0$ of Example~\ref{EXAzero}. A basis of the radical $\rad(q_i)$ is given in the columns of the matrix $K_i$ below, and the restriction $\wc{\mathbbm{r}}_{q_i}$ of the standard morsification of $q_i$ to its radical, under such basis, is given by the matrix $W_{q_i}$ as in~(\ref{EQQW}), for $i=0,1$,
\[
K_1=\begin{pmatrix}1&1\\1&0\\1&0\\0&\widehat{1}\end{pmatrix}, \qquad W_{q_1}=\begin{pmatrix}0&1\\\widehat{1}&0\end{pmatrix}, \qquad \begin{matrix} \CRnk_{re}(q_1)=0, \\ \dd_{q_1}=1. \end{matrix}
\]
\[
K_0=\begin{pmatrix}1&\widehat{1}\\1&0\\0&1\\1&0\end{pmatrix}, \qquad W_{q_0}=\begin{pmatrix}0&0\\0&0\end{pmatrix}, \qquad \begin{matrix} \CRnk_{re}(q_0)=2, \\ \dd_{q_0}=0. \end{matrix}
\]
Being the nullity of $W_{q_i}$, the respective reduced coranks are given by $\CRnk_{re}(q_1)=0$ and $\CRnk_{re}(q_0)=2$, and by~(\ref{EQQdd}) the corresponding degrees of degeneracy are $\dd_{q_1}=1$ and $\dd_{q_0}=0$.
\end{example}

Recall that $\mathbbm{1}_m$ denotes the vector in $\Z^m$ with all entries equal to $1$.

\begin{lemma}\label{(P):tech1}
Let $Q$ be a connected loop-less quiver with $m$ vertices. Then the image of the incidence matrix $I(Q)$, as a linear transformation $I(Q):\Z^n \to \Z^m$, is the set
\[
\Img I(Q) =\{x \in \Z^m \mid \mathbbm{1}^{\tr}_mx=0\},
\]
which is a pure subgroup of $\Z^m$.
\end{lemma}
\begin{proof}
Assume first that $Q$ is a tree, that is, that $Q$ has $m-1$ arrows. By~\cite[Propositions~3.13 and 3.8]{jaJ2020a}, there is a $\Z$-invertible $(m-1) \times (m-1)$ matrix $B$ such that $I(Q)=I(S)B$, where $S$ is a maximal star with center a vertex $v$, that is, the columns of $I(S)$ are given by $\pm(\bas_v-\bas_{v'})$ for $v' \in \{1,\ldots,m\}-\{v\}$. These columns are a basis of the group $C:=\{x \in \Z^m \mid \mathbbm{1}^{\tr}_mx=0 \}$, which shows that $\Img I(Q)=C$.

\medskip
Now, take an arbitrary loop-less quiver $Q$ with $m$ vertices. Since $\mathbbm{1}_m^{\tr}I(Q)=0$, we have $\Img I(Q) \subseteq C$. Choose a spanning tree $\wt{Q}$ of $Q$. Clearly, the image of $I(\wt{Q})$ is a subset of the image of $I(Q)$, since $I(\wt{Q})$ is obtained from $I(Q)$ by deleting those columns indexed by the elements of the set $Q_1-\wt{Q}_1$. Thus, by the first part of the proof, we have $C \subseteq \Img I(Q)$, hence $\Img I(Q)=C$. Clearly, if $ax \in C$ for some non-zero $a \in \Z$, then $\mathbbm{1}^{\tr}_m(ax)=a(\mathbbm{1}_m^{\tr}x)=0$. Since $a \neq 0$, we have $x \in C$, that is, $C$ is a pure subgroup of $\Z^m$.
\end{proof}

The following is a useful characterization of the reduced radical in case of non-negative unit forms of Dynkin type $\A_r$.

\begin{lemma}\label{(P):radre}
Let $Q$ be a connected loop-less quiver with inverse quiver $Q^{\dagger}$, Coxeter-Gram matrix $\Phi_Q$ and Coxeter-Laplacian $\Lambda_Q$. Then
\begin{eqnarray*}
\rad(q_Q) &=& \{x \in \Z^n \mid \Phi^{\tr}_Qx=x \}, \\
\rad_{re}(q_Q) &=& \{ I(Q^{\dagger})^{\tr}y \in \Z^n \mid \text{$y \in \Z^m$ and $\Lambda^{\tr}_Qy=y$} \}.
\end{eqnarray*}
\end{lemma}
\begin{proof}
The first identity is easy to verify, since $G_Qx=(\wc{G}_Q+\wc{G}_Q^{\tr})x=0$ if and only if $\wc{G}_Qx=-\wc{G}_Q^{\tr}$, that is, if and only if $x=\Phi_Q^{\tr}x$.

\medskip
For the second identity, recall from~(\ref{EQQ3}) that $I(Q^{\dagger})=I(Q)\wc{G}_Q^{-1}$, and observe that if $\Lambda_Q^{\tr}y=y$ then
\[
I(Q)I(Q^{\dagger})^{\tr}y=(\Id-\Lambda^{\tr}_Q)y=0,
\]
that is, $I(Q^{\dagger})^{\tr}y \in \Ker I(Q)$. Moreover, for any $x \in \rad(q_Q)=\Ker I(Q)$ we have
\begin{eqnarray*}
\langle x, I(Q^{\dagger})^{\tr}y\rangle_Q &=&x^{\tr}\wc{G}_QI(Q^{\dagger})^{\tr}y=-x^{\tr}\wc{G}^{\tr}_QI(Q^{\dagger})^{\tr}y \\ &=& -[I(Q^{\dagger})\wc{G}_Qx]^{\tr}y=-[I(Q)x]^{\tr}y=0,
\end{eqnarray*}
which means that $I(Q^{\dagger})^{\tr}y \in \rad_{re}(q_Q)$.

\medskip
On the other hand, take an arbitrary $x$ in $\rad_{re}(q_Q)$ and choose a kernel matrix $K$ of $I(Q)$. Then $x^{\tr}\wc{G}_QK=0$, and since $\Img I(Q)$ is pure (Lemma~\ref{(P):tech1}), by Remark~\ref{Rpure}(e) below there is a vector $y \in \Z^m$ such that
\[
x^{\tr}\wc{G}_Q=y^{\tr}I(Q), \quad \text{that is,} \quad x=I(Q^{\dagger})^{\tr}y.
\]
Moreover, we have $I(Q)x=0$, which implies that $(\Id-\Lambda_Q^{\tr})y=I(Q)I(Q^{\dagger})^{\tr}y=0$. We conclude that $x=I(Q^{\dagger})^{\tr}y$ for some $y \in \Z^m$ with $\Lambda_Q^{\tr}y=y$.
\end{proof}

\subsection{Bases for the reduced radical}

Let $v^1_1,\ldots,v^1_{\ell}$ be representative vertices of the $\xi^-_Q$-orbits in $Q_0$, with orbit sizes $\pi_1 \geq \pi_2 \geq \ldots \geq \pi_{\ell}$. The cycle type of $Q$, as given in~(\ref{EQQct}), is the partition $\ct(Q)=(\pi_1,\ldots,\pi_{\ell})$ of the integer $m=|Q_0|$ (cf.~\cite[Definition~4.2]{jaJ2020b}). For $t=1,\ldots,\ell$ consider the concatenated walks
\begin{equation}\label{EQQwalks}
\beta_t=\alpha^-_Q(v^1_t)\alpha^-_Q(v^2_t) \cdots \alpha^-_Q(v^{\pi_t}_t),
\end{equation}
where $v^{r+1}_t=\xi^-_Q(v^r_t)$ for $r=1,\ldots,\pi_t-1$, and $v^1_t=\xi^-_Q(v^{\pi_t}_t)$. This is indeed a walk, since $\tar(\alpha^-_Q(v_t^r))=\xi^-_Q(v_t^r)=v_t^{r+1}=\sou(\alpha^-_Q(v_t^{r+1}))$ if $1 \leq r <\pi_t$. Alternatively, taking $\mathbbm{1}_{[t]}$ as the vector in $\Z^m$ with entry in position $v$ given by $1$ if $v$ is in the $\xi^-_Q$-orbit of $v_t^1$, and $0$ otherwise (for $t=1,\ldots,\ell$), then using Lemma~\ref{La}$(i,ii)$ we have
\begin{equation}\label{EQQbasis}
\inc(\beta_t)=-I(Q^{\dagger})^{\tr}\mathbbm{1}_{[t]}, \quad \text{for $t=1,\ldots,\ell$}.
\end{equation}

\begin{lemma}\label{(P):basis}
The set $\beta\{t_0\}:=\{\inc(\beta_1),\ldots,\inc(\beta_{t_0-1}),\inc(\beta_{t_0+1}),\ldots,\inc(\beta_{\ell})\}$ is a basis of $\rad_{re}(q_Q)$ for any $t_0\in \{1,\ldots,\ell\}$.
\end{lemma}
\begin{proof}
First observe that the vectors $\mathbbm{1}_{[1]},\ldots,\mathbbm{1}_{[\ell]}$ of $\Z^m$ are a basis of the eigenspace of $\Lambda_Q^{\tr}$ corresponding to the eigenvalue $1$, since $\Lambda_Q$ is the permutation matrix of $\xi^-_Q$, see~(\ref{EQQ5}). Hence, by~(\ref{EQQbasis}), Lemma~\ref{(P):radre}, and since
\[
\sum_{t=1}^{\ell}\inc(\beta_t)=-I(Q^{\dagger})^{\tr}\mathbbm{1}_m=0,
\]
the set $\beta\{t_0\}$ generates $\rad_{re}(q_Q)$, for any $t_0 \in \{1,\ldots,\ell\}$.

\medskip
Take now integers $\lambda_1,\ldots,\lambda_{\ell}$ such that $\sum_{t=1}^{\ell}\lambda_t\inc(\beta_t)=0$. That is, if $y:=\sum_{t=1}^{\ell}\lambda_t\mathbbm{1}_{[t]}$ then $0=y^{\tr}I(Q^{\dagger})$. In particular, $\sum_{t=1}^{\ell}\lambda_t\inc(\beta_t)=0$ if and only if all $\lambda_t$ are equal, since the left null space of $I(Q^{\dagger})$ is generated by $\mathbbm{1}_m^{\tr}$ (see~\cite[Theorem~3.3$(ii)$]{jaJ2020b}). This shows that if $\lambda'_{t_0}=0$, and
\[
0 = \sum_{\substack{i=1,\ldots,\ell\\t \neq t_0}}\lambda_t' \inc(\beta_t)=\sum_{t=1}^{\ell}\lambda'_t \inc(\beta_t),
\]
then $\lambda_t'=0$ for $t=1,\ldots,\ell$, which completes the proof.
\end{proof}

\begin{example}\label{EXAfour}
Recall the description of the Coxeter-Laplacians $\Lambda_{Q^1}$ and $\Lambda_{Q^0}$ of the quivers $Q^1$ and $Q^0$ of Example~\ref{EXAzero}, given in Example~\ref{EXAone}. In the first case we have $\Lambda_{Q^1}^{\tr}y=y$ iff $y=a\mathbbm{1}_3$ for some $a\in \Z$, and in this case
\[
I((Q^1)^{\dagger})^{\tr}y=a\wc{G}_{Q^1}^{-\tr}I(Q^1)^{\tr}\mathbbm{1}_3=0.
\]
By Lemma~\ref{(P):radre}, we have $\rad_{re}(q_1)=0$ (see Example~\ref{EXAthree}). In the second case $\Lambda_{Q^0}^{\tr}y=y$ for all $y \in \Z^3$. Following the procedure~(\ref{EQQwalks}), we find the walks
\[
\beta_1=\alpha_{Q^0}^-(1)=4^{+1}2^{+1}1^{+1}, \qquad \beta_2=\alpha_{Q^0}^-(2)=4^{-1}3^{-1}2^{-1} \quad \text{and} \quad \beta_3=\alpha_{Q^0}^-(3)=3^{+1}1^{-1},
\]
(cf. Example~\ref{EXAtwo}). Alternatively, consider the inverse quiver $(Q^0)^{\dagger}$ of $(Q^0)$ given in Example~\ref{EXAone}, and observe that the columns of the matrix $-I((Q^0)^{\dagger})^{\tr}$ are precisely the incidence vectors $\inc(\beta_t)$ of the walks $\beta_1,\beta_2,\beta_3$, see~(\ref{EQQbasis}). Note that the basis of $\rad(q_0)=\rad_{re}(q_0)$ chosen in Example~\ref{EXAthree} (as columns of $K_0$) corresponds to the basis $\beta\{2\}=\{\inc(\beta_1),\inc(\beta_3)\}$ of Lemma~\ref{(P):basis}.
\end{example}

As direct consequence of Lemma~\ref{(P):basis} we have the following results.

\begin{corollary}\label{CreC}
Let $q$ be a connected non-negative unit form of Dynkin type $\A_r$, corank $c$ and cycle type $\ct(q)$. Then the reduced corank of $q$ is $\CRnk_{re}(q)=\ell(\ct(q))-1$, and the degree of degeneracy of $q$ is $\dd_q=\frac{1}{2}[c-\ell(\ct(q))+1]$.
\end{corollary}
\begin{proof}
By Lemma~\ref{(P):basis}, the rank of $\rad_{re}(q)$, the so-called reduced corank of $q$, is $\CRnk_{re}(q)=\ell(\ct(q))-1$. Then the claim on the degree of degeneracy follows from~(\ref{EQQdd}).
\end{proof}

\begin{corollary}\label{(C):dd}
The degree of degeneracy of the standard $(\pi,\dd)$-extension of $q_{\A_{m-1}}$ is $\dd$.
\end{corollary}
\begin{proof}
Take $\vec{Q}:=\vec{\A}_m^{\dd}[\pi]$ for a partition $\pi$ of $m \geq 2$, and $\dd \geq 0$. By Remark~\ref{Rxi}, the form $q:=q_{\vec{Q}}$ has cycle type $\ct(q)=\pi$ and corank $\CRnk(q)=\ell(\pi)+2\dd-1$. By Corollary~\ref{CreC} we have $\dd_q=\frac{1}{2}[\CRnk(q)-\ell(\pi)+1]=\dd$, as claimed.
\end{proof}

Recall that in the text, matrices, linear transformations and their images and kernels are taken over $\Z$. A matrix or linear transformation $M$ will be called \textbf{pure} if so is the group $\Img(M)$. Similarly, a bilinear form $\mathbbm{b}:\Z^n \times \Z^n \to \Z$ will be called \textbf{pure} if so is the adjoint transformation $x \mapsto [\mathbbm{b}(x,-):\Z^n \to \Z]$, or equivalently, if its Gram matrix $G_{\mathbbm{b}}$ is pure (under any choice of basis). The following result is a simple and important observation: \emph{the upper triangular bilinear form of any standard extension of $q_{\A_r}$ has pure restriction to its radical}. After completing the proof of our main results we will be able to prove that this observation holds for any connected non-negative unit form of Dynkin type $\A_r$ (Corollary~\ref{CinDet}).

\begin{lemma}\label{(L):stan}
The standard morsification of every standard extension of a unit form of Dynkin type $\A_r$ ($r \geq 1$) has pure restriction to its radical.
\end{lemma}
\begin{proof}
Consider a standard quiver $\vec{Q}=\vec{A}_m^{\dd}[\pi]$ with $\pi$ a partition of $m \geq 2$ with length $\ell=\ell(\pi)$, and take $q=q_{\vec{Q}}$. Then $q$ is a connected non-negative unit form of Dynkin type $\A_{m-1}$ in $n=m+\ell+2(\dd-1)$ variables with corank $\CRnk(q)=2\dd+\ell-1$ (Corollary~\ref{coro:STDquad}) and reduced corank $\CRnk_{re}(q)=\ell-1$ (Lemma~\ref{(P):basis}). We fix a kernel matrix $K=[K',K'']$ of $I(\vec{Q})$ in the following way. Take the $n \times (\ell-1)$ matrix $K''$ whose columns are a basis of $\rad_{re}(q)$, see for instance Lemma~\ref{(P):basis}. If $\dd=0$, then $K=K''$. If $\dd>0$, observe that the last $2\dd+1$ arrows of the standard quiver $\vec{Q}$ determine the inverse Kronecker quiver $\K^{\dagger}_{2\dd+1}$. Consider the inclusion $\iota: \Z^{2\dd+1} \to \Z^n$ of $\Z^{2\dd+1}$ into the last $2\dd+1$ entries of $\Z^n$, and take the $n \times 2\dd$ matrix $K'=[\iota(c_1),\ldots,\iota(c_{2\dd})]$ with $c_1,\ldots,c_{\dd}$ the vectors constructed in Remark~\ref{RKro}$(b)$, where we showed the second equality identity in~(\ref{EQQlocal}) below (the first identity can be easily shown, since $\K^{\dagger}_{2\dd+1}$ is a subquiver of $\vec{Q}$):
\begin{equation}\label{EQQlocal}
(K')^{\tr}\wc{G}_{\vec{Q}}(K')=[c_1,\ldots,c_{\dd}]^{\tr}\wc{G}_{\K^{\dagger}_{2\dd+1}}[c_1,\ldots,c_{\dd}]=\bigoplus_{t=1}^{\dd}\begin{pmatrix} 0&1\\-1&0 \end{pmatrix}.
\end{equation}
Then the subspace of $\Z^n$ generated by the columns of $K'$ has zero intersection with the reduced radical of $q$, which implies that the $n \times \CRnk(q)$ matrix $K$ is indeed a kernel matrix of $I(\vec{Q})$.

\medskip
Now, relative to the fixed basis $K$, the restriction $\wc{\mathbbm{r}}_q$ of $\wc{\mathbbm{b}}_q$ to the radical $\rad(q)$ has Gram matrix $W_q=K^{\tr}\wc{G}_{\vec{Q}}K$, see~(\ref{EQQW}). Since the kernel of $\wc{\mathbbm{r}}_q$ is the reduced radical of $q$, see~(\ref{EQQdd}), the quotient of $\wc{\mathbbm{r}}_q$ by its kernel has Gram matrix $(K')^{\tr}\wc{G}_{\vec{Q}}(K')$, which is $\Z$-invertible by~(\ref{EQQlocal}). That is, $\wc{\mathbbm{r}}_q$ is a pure bilinear form, as claimed.
\end{proof}

\subsection{Adding for invertibility}

Here we show how to correct the non-invertibility of the matrices obtained in Proposition~\ref{Ls}. We need the following preliminary observation.

\begin{lemma}\label{(P):isore}
Let $Q$ and $\wt{Q}$ be connected loop-less quivers with $m$ vertices and $n$ arrows. Assume that $\Lambda_Q=\Lambda_{\wt{Q}}$, and that there is a matrix $B$ such that $I(\wt{Q}^{\dagger})B^{\tr}=I(Q^{\dagger})$. Then $B$ restricts to an isomorphism $B|_{\rad_{re}}:\rad_{re}(q_{\wt{Q}}) \to \rad_{re}(q_Q)$.
\end{lemma}
\begin{proof}
Since $\Lambda_Q=\Lambda_{\wt{Q}}$, and since the matrix $I(Q^{\dagger})^{\tr}$ has as columns the vectors $\inc(\alpha^+_Q(v))$ for $v \in Q_0$ (see Lemma~\ref{La}$(ii)$), by Lemma~\ref{(P):radre} and~(\ref{EQQbasis}), the equation $BI(\wt{Q}^{\dagger})^{\tr}=I(Q^{\dagger})^{\tr}$ implies that the transformation $B$ sends the bases for $\rad_{re}(q_{\wt{Q}})$ constructed in Lemma~\ref{(P):basis}, to the corresponding bases of $\rad_{re}(q_Q)$. Hence the claim.
\end{proof}

For a slightly more general version of the following proposition, see Remark~\ref{Rgeneral} below.

\begin{proposition}\label{(P):inv}
Let $Q$ be a connected loop-less quiver. If there is a square matrix $B$ such that
\[
I(Q)B=I(\vec{Q}) \qquad \text{and} \qquad I(\vec{Q}^{\dagger})B^{\tr}=I(Q^{\dagger}),
\]
where $\vec{Q}$ is a standard quiver with same number of vertices and arrows as $Q$, then there is a matrix $M$ such that $B+M$ is $\Z$-invertible and satisfies
\[
I(Q)[B+M]=I(\vec{Q}) \qquad \text{and} \qquad I(\vec{Q}^{\dagger})[B+M]^{\tr}=I(Q^{\dagger}).
\]
\end{proposition}
\begin{proof}
Let $K$ be a kernel matrix for $I(Q)$, and assume that $K=[K' , K'']$ where the last $\ell-1$ columns $K''$ of $K$ are a basis of the reduced radical $\rad_{re}(q_Q)$. Take similarly a kernel matrix $\vec{K}$ of $I(\vec{Q})$ written as $\vec{K}=[\vec{K}' , \vec{K}'']$, and define $K^{\dagger}:=\wc{G}_QK$ and $\vec{K}^{\dagger}:=\wc{G}_{\vec{Q}}\vec{K}$, which are kernel matrices of $I(Q^{\dagger})$ and $I(\vec{Q}^{\dagger})$ respectively. Since $I(Q)(B\vec{K})=I(\vec{Q})\vec{K}=0$, there is a unique matrix $L$ such that $B\vec{K}=KL$. The matrices $L$ and $\vec{W}=\vec{K}^{\tr}\wc{G}_{\vec{Q}}\vec{K}$ have the following shapes
\[
L=\begin{pmatrix} L_1&0\\L_2&L_3 \end{pmatrix} \quad \text{and} \quad \vec{W}=\begin{pmatrix} \vec{W}'&0\\0&0 \end{pmatrix},
\]
where $L_3$ is $\Z$-invertible (Lemma~\ref{(P):isore}), and by~(\ref{EQQlocal}) in the proof of Lemma~\ref{(L):stan}, $\vec{W}'$ is a $\Z$-invertible skew-symmetric matrix (since it corresponds to the restriction of $\wc{\mathbbm{b}}_{q_{\vec{Q}}}$ to its radical, modulo its kernel: the reduced radical). Define $M=KY\vec{K}^{\dagger\tr}$ for some matrix $Y=\left(\begin{smallmatrix} Y_1&Y_3\\Y_2&Y_4 \end{smallmatrix}\right)$, and notice that
\[
I(Q)(B+M)=I(Q)B+I(Q)KY\vec{K}^{\dagger\tr}=I(\vec{Q}),
\]
and
\[
I(\vec{Q}^{\dagger})(B+M)^{\tr}=I(\vec{Q}^{\dagger})B^{\tr}+I(\vec{Q}^{\dagger})K^{\dagger}Y^{\tr}K^{\tr}=I(Q^{\dagger}).
\]
Since $\vec{K}^{\dagger\tr}=-\vec{K}^{\tr}\wc{G}_{\vec{Q}}$ and $\vec{K}^{\dagger\tr}\vec{K}=-\vec{K}^{\tr}\wc{G}_{\vec{Q}}\vec{K}=-\vec{W}$, we have
\[
(B+M)\vec{K}=B\vec{K}+KY\vec{K}^{\dagger\tr}\vec{K}=K(L-Y\vec{W})=K\begin{pmatrix} L_1-Y_1\vec{W}'&0\\L_2-Y_2\vec{W}'&L_3 \end{pmatrix}.
\]
Taking $Y_1=(L_1-\Id)(\vec{W}')^{-1}$ and $Y_2=L_2(\vec{W}')^{-1}$ (with $Y_3$ and $Y_4$ arbitrary, say equal to zero), we get
\[
(B+M)\vec{K}=K(L-Y\vec{W}),
\]
with $(L-Y\vec{W})=\left( \begin{smallmatrix} \Id&0\\0&L_3\end{smallmatrix} \right)$ a $\Z$-invertible matrix. Then the restriction
\[
(B+M)|_{\rad}: \rad(q_{\vec{Q}}) \to \rad(q_Q),
\]
is an isomorphism, which implies that $(B+M)$ is $\Z$-invertible by Corollary~\ref{Cinv}$(ii)$.
\end{proof}

\begin{example}\label{EXAfive}
Consider the quivers $Q^i$ of Example~\ref{EXAzero}, and the matrices $B'_i$ satisfying the assumptions of Proposition~\ref{(P):inv}, as given in Example~\ref{EXAtwo} (for $i=1,0$). Observe that $\det(B_1')=0$ and $\det(B_0')=1$, so we apply Proposition~\ref{(P):inv} only for the case $i=1$. Besides the kernel matrix $K_1$ of $I(Q_1)$ given in Example~\ref{EXAthree}, we fix the following kernel matrix for the corresponding standard quiver $\vec{Q}^1$ (cf. Example~\ref{EXAone}),
\[
\vec{K}_1=\begin{pmatrix} 0&0\\1&0\\1&\widehat{1}\\0&\widehat{1} \end{pmatrix}, \qquad B_1'\vec{K}_1=K_1\begin{pmatrix} 0&0\\0&0 \end{pmatrix}=K_1L.
\]
Observe that $B'_1\vec{K}_1=0$, and therefore the matrix $L$ (using the notation of the proof of Proposition~\ref{(P):inv}) is the $2 \times 2$ zero matrix. Note also that $\vec{W}:=\vec{K}_1^{\tr}\wc{G}_{\vec{Q}^1}\vec{K}_1=\left(\begin{smallmatrix} 0&1\\\widehat{1}&0 \end{smallmatrix}\right)$. Then, as defined in the proof of Proposition~\ref{(P):inv}, the matrix $Y$ is given by $Y=(L-\Id_2)\vec{W}^{-1}=-\vec{W}^{\tr}=\vec{W}$. Taking $M_1:=K_1Y(\vec{K}_1^{\dagger})^{\tr}$ where $\vec{K}_1^{\dagger}:=\wc{G}_{\vec{Q^1}}\vec{K}_1$, we get
\[
M_1=\begin{pmatrix} 0&1&0&\widehat{1}\\ 0&0&1&\widehat{1}\\0&0&1&\widehat{1}\\0&\widehat{1}&1&0 \end{pmatrix} \qquad \text{and} \qquad B_1:=B_1'+M_1=\begin{pmatrix} 0&0&1&\widehat{2}\\ \widehat{1}&0&1&\widehat{1}\\0&\widehat{1}&2&\widehat{2}\\\widehat{1}&\widehat{1}&1&0 \end{pmatrix}.
\]
Note that $\det(B_1)=\pm 1$, as claimed in Proposition~\ref{(P):inv}. In fact, a direct computation shows that $q_1 \approx^{B_1} \vec{q}_1$.
\end{example}

\section{A model for pseudo-endomorphisms} \label{S4}

In this section we provide a method to modify a $\Z$-invertible matrix satisfying the equations of Proposition~\ref{Ls} (called \emph{pseudo-morphisms}) into a strong Gram congruence. It depends on a special decomposition of skew-symmetric matrices, presented first in Lemma~\ref{(I):skewsym}$(b)$.

\subsection{Decompositions of skew-symmetric matrices}

Recall that if $Z$ is a skew-symmetric $c \times c$ matrix, then there exists a $\Z$-invertible matrix $P$ such that $P^{\tr}ZP$ is in canonical form, that is, there are positive integers $d_1,d_2,\ldots,d_r$ such that
\[
P^{\tr}ZP=d_1\left(\begin{smallmatrix}0&1\\-1&0 \end{smallmatrix}\right) \oplus d_2\left(\begin{smallmatrix}0&1\\-1&0 \end{smallmatrix}\right) \oplus \ldots \oplus d_r\left(\begin{smallmatrix}0&1\\-1&0 \end{smallmatrix}\right) \oplus \mathbf{0},
\]
for a square zero matrix $\mathbf{0}$ of size $c-2r$, and where $d_t$ divides $d_{t+1}$ for $t=1,\ldots, r-1$ (see for instance~\cite[Theorem~IV.1]{Newman:1972}). Such expression is usually called the \textbf{skew normal form} of $Z$. Observe that $Z$ is pure if and only if $d_t=1$ for $t=1,\ldots,r$.

\begin{lemma}\label{(I):skewsym}
Let $Z$ and $W$ be skew-symmetric $c \times c$ matrices.
\begin{itemize}
 \item[a)] If $W$ is $\Z$-invertible, then there is a matrix $Y$ such that $Z=Y^{\tr}WY$.
 \item[b)] If $W$ is pure, then there is a matrix $Y$ such that $Z=Y-Y^{\tr} + Y^{\tr}WY$.
\end{itemize}
\end{lemma}
\begin{proof}
Take $\wt{W}=R^{\tr}WR$ for some $\Z$-invertible matrix $R$. If claims $(a)$ and $(b)$ hold for $\wt{W}$ and arbitrary $Z$, then they also hold for $W$ and arbitrary $Z$. Indeed, if $\wt{Z}:=R^{-1}ZR^{-\tr}$, then for each case $(a)$ or $(b)$ there is a matrix $\wt{Y}$ such that either
\[
\wt{Z}=\wt{Y}^{\tr}\wt{W}\wt{Y}, \quad \text{or} \quad \wt{Z}=\wt{Y}-\wt{Y}^{\tr}+\wt{Y}^{\tr}\wt{W}\wt{Y},
\]
respectively. Take $Y:=R\wt{Y}R^{\tr}$. Then $\wt{Y}=R^{-1}YR^{-\tr}$, and
\begin{eqnarray*}
(a) &Z&=R\wt{Z}R^{\tr}=R[\wt{Y}^{\tr}\wt{W}\wt{Y}]R^{\tr}=Y^{\tr}WY, \\
(b) &Z&=R\wt{Z}R^{\tr}=R[\wt{Y}-\wt{Y}^{\tr}+\wt{Y}^{\tr}\wt{W}\wt{Y}]R^{\tr}=Y-Y^{\tr}+Y^{\tr}WY.
\end{eqnarray*}
Therefore, we may assume that $W$ is in skew normal form. In particular, if $W$ is $\Z$-invertible then $WW^{\tr}=W^{\tr}W=\Id_c=-WW$.

\medskip
To show $(a)$, observe that if $W$ is $\Z$-invertible then $c=2r$ for some $r\geq 1$ and $W=\bigoplus_{i=1}^rW_1$, where $W_1=\left( \begin{smallmatrix}0&1\\-1&0 \end{smallmatrix} \right)$. Take a $\Z$-invertible matrix $P$ such that $P^{\tr}ZP=\bigoplus_{i=1}^r d_iW_1$ (not necessarily in skew normal form), and consider the diagonal $c \times c$ matrix $S=\diag(1,d_1,1,d_2,\ldots,1,d_r)$. Then
\[
P^{\tr}ZP=\bigoplus_{i=1}^r d_iW_1=SWS, \quad \text{that is,} \quad Z=Y^{\tr}WY \quad \text{where $Y=SP^{-1}$}.
\]

To show $(b)$, assume first that $W$ is $\Z$-invertible and take $\wt{Z}:=Z+W$, which is also skew-symmetric. By part $(a)$ there is a matrix $\wt{Y}$ such that $\wt{Z}=\wt{Y}^{\tr}W\wt{Y}$. Taking $Y:=\wt{Y}-W$, we get\vspace*{-2mm}
\begin{eqnarray*}
Z&=&\wt{Z}-W=(Y+W)^{\tr}W(Y+W)-W\\
&=&Y^{\tr}WY+Y^{\tr}WW+W^{\tr}WY+W^{\tr}WW-W\\
&=&Y-Y^{\tr}+Y^{\tr}WY,
\end{eqnarray*}
as wanted. In the pure case, $W=\begin{pmatrix} W'&0\\0&0 \end{pmatrix}$ where $W'$ is $\Z$-invertible (since $W$ is in its skew normal form). Taking
\[
Z:=\begin{pmatrix} Z_1&-Z_2^{\tr}\\Z_2&Z_3 \end{pmatrix} \quad \text{and} \quad Y:=\begin{pmatrix} Y_1&0\\Y_2&Y_3 \end{pmatrix},
\]
then $Y-Y^{\tr}+Y^{\tr}WY=\begin{pmatrix}Y_1-Y_1^{\tr}+ Y_1^{\tr}W'Y_1&-Y_2^{\tr}\\Y_2&Y_3-Y_3^{\tr} \end{pmatrix}$. Using the first part of the proof of $(b)$, since $Z_1$ is skew-symmetric and $W'$ is $\Z$-invertible, we may find $Y_1$ such that $Z_1=Y_1-Y_1^{\tr}+ Y_1^{\tr}W'Y_1$. Take $Y_2=Z_2$, and take any matrix $Y_3$ such that $Z_3=Y_3-Y_3^{\tr}$. Then $Z=Y-Y^{\tr}+Y^{\tr}WY$, which completes the proof.
\end{proof}

\subsection{Pseudo-morphisms}

Let us formalize some of the notions already used in the paper.

\eject

\begin{definition}\label{DEF}
Consider connected loop-less quivers $Q$ and $\wt{Q}$ with $m$ vertices and $n$ arrows.
\begin{itemize}
 \item[a)] A square matrix $B$ satisfying
 \[
 I(Q)B=I(\wt{Q}) \qquad \text{and} \qquad I(\wt{Q}^{\dagger})B^{\tr}=I(Q^{\dagger}),
 \]
 is referred to as \textbf{pseudo-morphism} from $Q$ to $\wt{Q}$. These relations are expressed with the notation $Q \pseudo^{B} \wt{Q}$, or simply by $Q \pseudo \wt{Q}$ if such matrix $B$ exists.
 \item[b)] If $Q \pseudo^B Q$ we call $B$ a \textbf{pseudo-endomorphism} of $Q$. The set of pseudo-endomorphisms of $Q$ will be denoted by $\End_{\pseudo}(Q)$, that is,
\begin{equation*}
\End_{\pseudo}(Q)=\{B \in \mathbb{M}_n(\Z) \mid I(Q)B=I(Q)=I(Q)B^* \}, \quad \text{where $B^*:=\wc{G}_Q^{-1}B^{\tr}\wc{G}_Q$, see~(\ref{EQQ3})}.
\end{equation*}
 \item[c)] For matrices $Z$, $Z'$ and $W$ in $\mathbb{M}_c(\Z)$ we will use the notation
\begin{equation*}
Z \oper{W} Z' := Z+Z'-ZWZ'.
\end{equation*}
\end{itemize}
\end{definition}

\begin{remark}\label{Rcat}
Let $\quiv$ denote the set of connected loop-less quivers with at least two vertices. For $Q$ and $\wt{Q}$ in $\quiv$, denote by $\Hom_{\pseudo}(Q,\wt{Q})$ the set of pseudo-morphisms from $Q$ to $\wt{Q}$. Then $(\quiv,\Hom_{\pseudo})$, together with the product of matrices, is a category. Moreover, the following hold if $Q$ and $\wt{Q}$ have the same number of vertices and arrows.
\begin{itemize}
\itemsep=0.9pt
 \item[a)] If $Q \pseudo \wt{Q}$ then $\Lambda_Q=\Lambda_{\wt{Q}}$.
 \item[b)] If $Q \pseudo^B \wt{Q}$ then $\wt{Q} \pseudo^{B^*} Q$, where $B^*=\wc{G}_{\wt{Q}}^{-1}B^{\tr}\wc{G}_Q$.
\end{itemize}
In particular, $\pseudo$ is an equivalence relation on the set $\quiv$.
\end{remark}
\begin{proof}
Let $Q \pseudo^B \wt{Q}$ and $\wt{Q} \pseudo^C \wt{\wt{Q}}$ be pseudo-morphisms. Then
\[
I(Q)(BC)=(I(Q)B)C=I(\wt{Q})C=I(\wt{\wt{Q}}),
\]
and
\[
I(\wt{\wt{Q}}^{\dagger})(BC)^{\tr}=(I(\wt{\wt{Q}}^{\dagger})C^{\tr})B^{\tr}=I(\wt{Q}^{\dagger})B^{\tr}=I(Q^{\dagger}),
\]
that is, $Q \pseudo^{BC} \wt{\wt{Q}}$. Note that the identity matrix serves as neutral element of the monoid $\Hom_{\pseudo}(Q,Q)$, and the associativity of the matrix product shows that $(\quiv,\Hom_{\pseudo})$ is a category.

\medskip
To show $(a)$, assume that $Q \pseudo^B \wt{Q}$ for a matrix $B$. By definition of $\Lambda_{\wt{Q}}$, see~(\ref{EQQCL}), and since $I(\wt{Q}^{\dagger})=I(\wt{Q})\wc{G}_{\wt{Q}}^{-1}$, see~(\ref{EQQ3}), we have
\begin{eqnarray*}
\Lambda_{\wt{Q}}&=&\Id-I(\wt{Q}^{\dagger})I(\wt{Q})^{\tr}=\Id-I(\wt{Q}^{\dagger})[I(Q)B]^{\tr}\\
&=& \Id-[I(\wt{Q}^{\dagger})B^{\tr}]I(Q)^{\tr}=\Id-I(Q^{\dagger})I(Q)^{\tr}=\Lambda_Q.
\end{eqnarray*}

To prove $(b)$, using~(\ref{EQQ3}) we get
\[
I(\wt{Q})B^*=I(\wt{Q}^{\dagger})\wc{G}_{\wt{Q}}(\wc{G}_{\wt{Q}}^{-1}B^{\tr}\wc{G}_Q)=I(\wt{Q}^{\dagger})B^{\tr}\wc{G}_Q=I(Q^{\dagger})\wc{G}_Q=I(Q).
\]

\eject

Moreover, using $(a)$, (\ref{EQQ3}) and Remark~\ref{LAMBDA}$(e)$, we have
\begin{eqnarray*}
I(Q^{\dagger})(B^*)^{\tr} &=& I(Q)\wc{G}_Q^{-1}(\wc{G}_Q^{\tr}B\wc{G}_{\wt{Q}}^{-\tr})=-[I(Q)\Phi_Q^{-\tr}]B\wc{G}_{\wt{Q}}^{-\tr}= -\Lambda_QI(Q)B\wc{G}_{\wt{Q}}^{-\tr} \\
&=& -\Lambda_{\wt{Q}}I(\wt{Q})\wc{G}_{\wt{Q}}^{-\tr}=-I(\wt{Q})\Phi_{\wt{Q}}^{-\tr}\wc{G}_{\wt{Q}}^{-\tr}=I(\wt{Q})\wc{G}_{\wt{Q}}^{-1}\wc{G}_{\wt{Q}}^{\tr}\wc{G}_{\wt{Q}}^{-\tr}= I(\wt{Q}^{\dagger}).
\end{eqnarray*}
That is, $\wt{Q} \pseudo^{B^*} Q$.

\medskip
Note that $\pseudo$ is a reflective relation since $Q \pseudo^{\Id} Q$, and it is transitive by the first part of the proof. Since the symmetry of $\pseudo$ was shown in $(b)$, we conclude that $\pseudo$ is an equivalence relation.
\end{proof}

With the notation introduced in Definition~\ref{DEF}, Propositions~\ref{Ls} and~\ref{(P):inv} may be summarized as follows:

\begin{corollary}\label{C:PseudoStandard}
If $Q$ is a connected loop-less quiver with $\Lambda_Q=\Lambda_{\vec{Q}}$ for a standard quiver $\vec{Q}$ with the same number of arrows as $Q$, then there is a $\Z$-invertible matrix $B$ such that $Q \pseudo^B \vec{Q}$.
\end{corollary}

\begin{lemma}\label{(M):endo}
Let $Q$ be a connected loop-less quiver with $m$ vertices and $n$ arrows, fix a kernel matrix $K$ for $I(Q)$, and take $K^{\dagger}:=\wc{G}_QK$. If $c:=n-m+1$, then the function
\[
\xymatrix@R=.5pc{ \mathbb{M}_c(\Z) \ar[rr]^-{\Upsilon} && \End_{\pseudo}(Q) \\ Z \ar@{|->}[rr] && \Upsilon(Z):=\Id+KZK^{\dagger\tr},}
\]
is a bijection, with inverse denoted by $\Xi$. Moreover, if $B$ and $B'$ are pseudo-endomorphisms of $Q$, and $W:=K^{\tr}\wc{G}_QK$, then
\begin{itemize}
\item[a)] $\Xi(BB')=\Xi(B) \oper{W} \Xi(B')$, where $\oper{W}$ is as in Definition~\ref{DEF}$(c)$.
\item[b)] We have $\Xi(\Id_n)=0$, and $B$ is $\Z$-invertible if and only if there is a matrix $Z$ with $\Xi(B)\oper{W}Z=0$ (or equivalently, $Z\oper{W}\Xi(B)=0$).
\end{itemize}
\end{lemma}
\begin{proof}
Let $c=n-m+1$ be the corank of $q_Q$. Then $K$ and $K^{\dagger}$ are $n \times c$ matrices, and for any $Z \in \mathbb{M}_c(\Z)$, we have
\[
I(Q)\Upsilon(Z)=I(Q)[\Id+KZK^{\dagger\tr}]=I(Q)+I(Q)K(ZK^{\dagger\tr})=I(Q),
\]
and
\[
I(Q^{\dagger})\Upsilon(Z)^{\tr}=I(Q^{\dagger})[\Id+K^{\dagger}Z^{\tr}K^{\tr}]=I(Q^{\dagger})+I(Q^{\dagger})K^{\dagger}(Z^{\tr}K^{\tr})=I(Q^{\dagger}),
\]
since $I(Q^{\dagger})=I(Q)\wc{G}_Q^{-1}$ and $K^{\dagger}=\wc{G}_QK$. This shows, as claimed, that $\Upsilon(Z)$ belongs to $\End_{\pseudo}(Q)$.

\medskip
Now, if $B$ is a pseudo-endomorphism of $Q$,  then $B$ may be expressed uniquely as
\[
B=\Id+KZK^{\dagger\tr}=\Upsilon(Z),
\]
for a $c \times c$ matrix $Z$. Indeed, we have $I(Q)(B-\Id)=0$, which implies that there is a unique $c \times n$ matrix $L$ such that $B-\Id=KL$ (since the kernel matrix $K$ of $I(Q)$ has rank $c$). On the other hand, $(KL)I(Q^{\dagger})^{\tr}=(B-\Id)I(Q^{\dagger})^{\tr}=0$ implies that $LI(Q^{\dagger})^{\tr}=0$ since the columns of $K$ are linearly independent, thus there is a unique matrix $Z$ such that $L=ZK^{\dagger\tr}$. Then the mapping $\Xi(B):=Z$ is well defined, by the uniqueness of $Z$. That is, $\Xi(\Upsilon(Z))=Z$, and clearly
\[
\Upsilon(\Xi(B))=\Id+K\Xi(B)K^{\dagger\tr}=B,
\]
which proves that $\Upsilon$ is a bijection with inverse $\Xi$.

\medskip
To show $(a)$ take $\Xi(B):=Z$ and $\Xi(B'):=Z'$. Then
\begin{eqnarray*}
 BB' & = & \Upsilon(Z)\Upsilon(Z') = (\Id+KZK^{\dagger\tr})(\Id+KZ'K^{\dagger\tr}) \\
 & = & \Id+K[Z+Z'+ZK^{\dagger\tr}KZ']K^{\dagger\tr} =  \Id+K[Z+Z'-ZWZ']K^{\dagger\tr} \\
 & = & \Id+K[Z\oper{W} Z']K^{\dagger\tr},
\end{eqnarray*}
since $K^{\dagger\tr}K=(K^{\tr}\wc{G}_Q^{\tr})K=-W$ (recall that $(\wc{G}_Q+\wc{G}_Q^{\tr})K=0$). Claim $(b)$ follows from the uniqueness of $\Xi(-)$, since clearly $\Xi(\Id)=0$.
\end{proof}

\begin{remark}\label{assoc}
Let $Q$ be a connected loop-less quiver, and take a kernel matrix $K$ for $I(Q)$. If $W:=K^{\tr}\wc{G}_QK$, then the binary operation $\oper{W}:\mathbb{M}_c(\Z)\times \mathbb{M}_c(\Z) \to \mathbb{M}_c(\Z)$ given in Definition~\ref{DEF}$(c)$ makes $(\mathbb{M}_c(\Z),\oper{W})$ a monoid with identity element the zero matrix $0 \in \mathbb{M}_c(\Z)$.
\end{remark}
\begin{proof}
Use the associativity of $\End_{\pseudo}(Q)$ (Remark~\ref{Rcat}) and apply Lemma~\ref{(M):endo}$(a,b)$.
\end{proof}

\subsection{Multiplying for strong Gram congruence}

For a pseudo-morphism $Q \pseudo^B \wt{Q}$ we have often considered the matrix $B^*:=\wc{G}_{\wt{Q}^{-1}}B^{\tr}\wc{G}_Q$ (see for instance Theorem~\ref{MTT}$(ii)$ or Remark~\ref{Rcat}). Here we consider further properties of this star operation,
\[
\xymatrix{\Hom_{\pseudo}(Q,\wt{Q}) \ar[rr]^-{(-)^*} && \Hom_{\pseudo}(\wt{Q},Q)},
\]
see Remark~\ref{Rcat}$(b)$.

\begin{lemma}\label{Lstar}
Let $Q$ and $\wt{Q}$ be connected loop-less quivers with the same number of vertices and arrows. Assume that $Q \pseudo^B \wt{Q}$, and take $B^*:=\wc{G}_{\wt{Q}}^{-1}B^{\tr}\wc{G}_Q$.
\begin{itemize}
 \item[a)] If $Q=\wt{Q}$ then we have $\Xi(B^*)=(-1)\Xi(B)^{\tr}$.
 \item[b)] Both matrices $\Xi(B^*B)$ and $\Xi(BB^*)$ are skew-symmetric.
 \item[c)] If $B$ is $\Z$-invertible then $\wt{Q} \pseudo^{B^{-1}} Q$ and
 \[
 (B^*)^{-1}=\wc{G}_Q^{-1}B^{-\tr}\wc{G}_{\wt{Q}}=:(B^{-1})^*.
 \]
\end{itemize}
\end{lemma}
\begin{proof}
To show $(a)$ assume that $Q=\wt{Q}$, and take $Z:=\Xi(B)$. Note that
\begin{eqnarray*}
 B^* & = & \wc{G}_Q^{-1}B^{\tr}\wc{G}_Q = \wc{G}_Q^{-1}(\Id+KZK^{\dagger\tr})^{\tr}\wc{G}_Q \\
 & = & \Id+\wc{G}_Q^{-1}K^{\dagger}Z^{\tr}K^{\tr}\wc{G}_Q =\Id +K(-Z^{\tr})K^{\dagger\tr} = \Upsilon(-Z^{\tr}),
\end{eqnarray*}
that is, $\Xi(B^*)=\Xi(\Upsilon(-\Xi(B)^{\tr}))=-\Xi(B)^{\tr}$.

\medskip
To show $(b)$ take $Z:=\Xi(B^*B)$. Since $B^*B=\Id+\wt{K}Z\wt{K}^{\dagger\tr}$ and $\wt{K}^{\dagger}=\wc{G}_{\wt{Q}}K$, we have
\[
\wt{K}^{\dagger}Z\wt{K}^{\dagger\tr}=\wc{G}_{\wt{Q}}\wt{K}Z\wt{K}^{\dagger\tr}=\wc{G}_{\wt{Q}}(B^*B-\Id)=B^{\tr}\wc{G}_QB-\wc{G}_{\wt{Q}},
\]
and therefore,
\[
\wt{K}^{\dagger}Z^{\tr}\wt{K}^{\dagger\tr}=(\wt{K}^{\dagger}Z\wt{K}^{\dagger\tr})^{\tr}=B^{\tr}\wc{G}_Q^{\tr}B-\wc{G}_{\wt{Q}}^{\tr}.
\]
Since $I(Q)B=I(\wt{Q})$ we get
\[
\wt{K}^{\dagger}[Z+Z^{\tr}]\wt{K}^{\dagger\tr}=B^{\tr}G_QB-G_{\wt{Q}}=(I(Q)B)^{\tr}(I(Q)B)-I(\wt{Q})^{\tr}I(\wt{Q})=0.
\]
In particular, $Z+Z^{\tr}=0$ since the columns of $\wt{K}^{\dagger}$ are linearly independent. The case $BB^*$ can be shown in a similar way.

\medskip
For $(c)$, using that $I(Q)B=I(\wt{Q})$ and $I(\wt{Q}^{\dagger})B^{\tr}=I(Q^{\dagger})$, if $B$ is $\Z$-invertible then
\[
I(\wt{Q})B^{-1}=I(Q) \qquad \text{and} \qquad I(Q^{\dagger})B^{-\tr}=I(\wt{Q}^{\dagger}),
\]
that is, $\wt{Q} \pseudo^{B^{-1}} Q$. The last claim of $(c)$ is immediate.
\end{proof}

For a more general version of the following proposition, see Remark~\ref{Rgeneral} below.

\begin{proposition}\label{(P):strong}
Let $Q$ be a connected loop-less quiver. If there is a $\Z$-invertible matrix $B$ such that
\[
I(Q)B=I(\vec{Q}) \qquad \text{and} \qquad I(\vec{Q}^{\dagger})B^{\tr}=I(Q^{\dagger}),
\]
where $\vec{Q}$ is a standard quiver with same number of vertices and arrows as $Q$, then there is a matrix $C$ such that $I(Q)[BC]=I(\vec{Q})$ and $[BC]^*[BC]=\Id$.
\end{proposition}
\begin{proof}
With the notation of Definition~\ref{DEF}, we have $Q \pseudo^B \vec{Q}$. Fix kernel matrices $K$ and $\vec{K}$ of $I(Q)$ and $I(\vec{Q})$, respectively.

\medskip
Take $\overrightarrow{B}:=B^*B$, $Z:=\Xi(\overrightarrow{B})$ and $\vec{Z}:=\Xi(\overrightarrow{B}^{-1})$, where $\Xi$ is the function constructed in Lemma~\ref{(M):endo} with respect to $\vec{K}$. By Lemma~\ref{(M):endo}$(a)$, we have $\vec{Z}\oper{\vec{W}} Z=0$ where $\vec{W}=\vec{K}^{\tr}\wc{G}_{\vec{Q}}\vec{K}$. By Lemma~\ref{Lstar}$(b,c)$, the matrices $Z$ and $\vec{Z}$ are skew-symmetric. By Lemma~\ref{(I):skewsym}$(b)$, there is a matrix $Y$ such that $\vec{Z}=(-Y^{\tr}) \oper{\vec{W}} Y$, since $\vec{W}$ is a pure skew-symmetric matrix by Lemma~\ref{(L):stan}, cf.~(\ref{EQQW}). Consider the pseudo-endomorphism $C=\Upsilon(-Y^{\tr})=\Id-\vec{K}Y^{\tr}\vec{K}^{\dagger\tr}$ of $\vec{Q}$, which is $\Z$-invertible by Lemma~\ref{(M):endo}$(b)$, since
\[
(-Y^{\tr}) \oper{\vec{W}} [Y \oper{\vec{W}}Z]= [(-Y^{\tr})\oper{\vec{W}}Y] \oper{\vec{W}}Z=\vec{Z}\oper{\vec{W}} Z=0,
\]
using the associativity of $\oper{\vec{W}}$ (Remark~\ref{assoc}). Observe that
\[
\overrightarrow{BC}=(BC)^*(BC)=C^*(B^*B)C=C^*\overrightarrow{B}C.
\]
In particular, $\overrightarrow{BC}=\Id$ if and only if $(CC^{*})\overrightarrow{B}=\Id$ since $C$ is $\Z$-invertible. Applying Lemma~\ref{(M):endo}$(a,b)$ we get
\[
\Xi[(CC^*)\overrightarrow{B}]=[(-Y^{\tr})\oper{\vec{W}}Y] \oper{\vec{W}}Z=\vec{Z}\oper{\vec{W}} Z=0,
\]
since $\Xi(C)=(-Y^{\tr})$, and $\Xi(C^*)=Y$ by Lemma~\ref{Lstar}$(a)$. Thus, we have $\overrightarrow{BC}=(BC)^*(BC)=\Id$, as wanted.
\end{proof}

\begin{example}\label{EXAsix}
We now apply Proposition~\ref{(P):strong} to the quiver $Q^0$ of Example~\ref{EXAzero} (recall from Example~\ref{EXAfive} that we have already found a matrix $B_1$ such that $q_1 \approx^{B_1} \vec{q}_1$). Consider the $\Z$-invertible matrix $B_0:=B_0'$ given in Example~\ref{EXAtwo} satisfying $Q^0 \pseudo^{B_0} \vec{Q}^0$. Following the notation of the proof of Proposition~\ref{(P):strong} we get
\[
B_0=\begin{pmatrix} 0&0&\widehat{1}&0\\0&1&\widehat{1}&\widehat{1}\\0&0&1&\widehat{1}\\1&0&0&\widehat{2} \end{pmatrix}, \qquad \vec{B}_0=\begin{pmatrix} 1&\widehat{1}&1&0\\1&0&1&\widehat{1}\\ 1&\widehat{1}&2&\widehat{1} \\ 0&\widehat{1}&1&1 \end{pmatrix} \qquad \text{and} \qquad
\vec{B}^{-1}_0=\begin{pmatrix} 1&1&\widehat{1}&0 \\ \widehat{1}&2&\widehat{1}&1 \\ \widehat{1}&1&0&1 \\ 0&1&\widehat{1}&1 \end{pmatrix}. \qquad
\]
Taking $\vec{Z}:=\left( \begin{smallmatrix}0&\widehat{1}\\1&0 \end{smallmatrix} \right)$ we have
\[
\vec{B}_0^{-1}=\Upsilon(\vec{Z})=\Id_4+\vec{K}_0\vec{Z}\vec{K}_0^{\dagger \tr}=
\begin{pmatrix}1&0\\0&1\\0&1\\1&0 \end{pmatrix}\begin{pmatrix}0&\widehat{1}\\1&0 \end{pmatrix}\begin{pmatrix}\widehat{1}&1&&1\\0&&1&0 \end{pmatrix}.
\]
By Lemma~\ref{(I):skewsym}$(b)$, there is a matrix $Y$ such that $\vec{W}=(-Y^{\tr}) \oper{\vec{W}} Y$. Indeed, note that $\vec{W}:=\vec{K}^{\tr}_0\wc{G}_{\vec{Q}^0}\vec{K}_0=0$, therefore $(-Y^{\tr}) \oper{\vec{W}} Y=Y-Y^{\tr}$, so we may simply take $Y:=\left( \begin{smallmatrix} 0&0\\1&0 \end{smallmatrix} \right)$. As in the proof above, we take
\[
C:=\Upsilon(-Y^{\tr})=\Id_4-\vec{K}_0Y^{\tr}\vec{K}_0^{\dagger \tr}=
\begin{pmatrix} 1&1&\widehat{1}&0 \\ 0&1&0&0 \\ 0&0&1&0 \\ 0&1&-1&1 \end{pmatrix} \quad \text{and} \quad B_0C=\begin{pmatrix} 0&0&\widehat{1}&0 \\ 0&0&0&\widehat{1} \\ 0&\widehat{1}&2&\widehat{1} \\ 1&\widehat{1}&1&\widehat{2} \end{pmatrix}.
\]
A direct calculation shows that $q_0 \approx^{B_0C} \vec{q}_0$, as claimed in Proposition~\ref{(P):strong}.
\end{example}

\section{Main proofs, concluding remarks and algorithms} \label{S5}

This section collects all preliminary results to prove the main technical theorem of the paper (Theorem~\ref{MTT}), which connects the Coxeter-Laplacian with the existence of strong Gram congruences, and suggests implementable algorithms to solve Simson's Problem~2$(ii)$. The section ends with those results of general interest used along the paper, and with some comments on generalizations and future work.

\subsection{Main results}

The following is a combinatorial version of Theorem~\ref{MT1} in terms of the Coxeter-Laplacian of a quiver.

\begin{theorem}\label{MTT}
The following are equivalent for connected loop-less quivers $Q$ and $\wt{Q}$ with the same number of vertices and arrows:
\begin{itemize}
 \item[i)] The Coxeter-Laplacians of $Q$ and $\wt{Q}$ coincide, $\Lambda_Q=\Lambda_{\wt{Q}}$.
 \item[ii)] There is a (not necessarily $\Z$-invertible) matrix $B$ such that $Q \pseudo^B \wt{Q}$.
 \item[iii)] There is a ($\Z$-invertible) matrix $B$ such that $I(Q)B=I(\wt{Q})$ and $\wc{G}_{\wt{Q}}=B^{\tr}\wc{G}_QB$.
\end{itemize}
\end{theorem}
\begin{proof}
Assume first that $B$ satisfies $(iii)$ and recall that $I(Q^{\dagger})=I(Q)\wc{G}_Q^{-1}$ for any connected loop-less quiver $Q$, see~(\ref{EQQ3}). Since $1=\det(\wc{G}_{\wt{Q}})=\det(B^{\tr}\wc{G}_QB)=\det(B)^2$, then $B$ is $\Z$-invertible and
\[
I(\wt{Q}^{\dagger})B^{\tr}=I(\wt{Q})[\wc{G}_{\wt{Q}}^{-1}B^{\tr}\wc{G}_Q]\wc{G}_Q^{-1}=I(\wt{Q})B^{-1}\wc{G}_Q^{-1}=I(Q)\wc{G}_Q^{-1}=I(Q^{\dagger}).
\]
That is, $Q \pseudo^B \wt{Q}$, and $(ii)$ holds. That $(ii)$ implies $(i)$ was shown in Remark~\ref{Rcat}$(a)$.

\medskip
Assume that $\Lambda_Q=\Lambda_{\wt{Q}}$. By Remark~\ref{RStan}, there is a standard quiver $\vec{Q}$ with the same number of vertices and arrows as $Q$, and such that $\ct(Q)=\ct(\vec{Q})$. By permuting the vertices of $Q$ and $\wt{Q}$ if necessary, we may assume that $\Lambda_Q=\Lambda_{\vec{Q}}=\Lambda_{\wt{Q}}$ (see Remarks~\ref{RPerm} and~\ref{RStan}; recall that in Definition~\ref{DStan} we fixed linear orders on $\vec{Q}_0$ and $\vec{Q}_1$).

\medskip
By Proposition~\ref{Ls} there is a matrix $B'$ such that $Q \pseudo^{B'} \vec{Q}$, and by Proposition~\ref{(P):inv} there is a matrix $M$ such that $B'+M$ is $\Z$-invertible and $Q \pseudo^{B'+M} \vec{Q}$. Moreover, using Proposition~\ref{(P):strong} we find a matrix $C$ such that if $B:=(B'+M)C$, then $I(Q)B=I(\vec{Q})$ and $B^*B=\Id$. Similarly, since $\Lambda_{\wt{Q}}=\Lambda_{\vec{Q}}$, we may find a matrix $\wt{B}$ such that $I(\wt{Q})\wt{B}=I(\vec{Q})$ and $\wt{B}^*\wt{B}=\Id$ (in particular, $\wt{B}$ is $\Z$-invertible since it is a square matrix with $\det(\wt{B})=\pm 1$). This means that $\wc{G}_{\vec{Q}}=B^{\tr}\wc{G}_QB$ and $\wc{G}_{\vec{Q}}=\wt{B}^{\tr}\wc{G}_{\wt{Q}}\wt{B}$, and therefore, $I(Q)(B\wt{B}^{-1})=I(\wt{Q})$ and
\[
\wc{G}_{\wt{Q}}=(B\wt{B}^{-1})^{\tr}\wc{G}_Q(B\wt{B}^{-1}),
\]
which completes the proof.
\end{proof}

Let us now derive Theorems~\ref{MT0}, \ref{MT1} and~\ref{MT2} from Theorem~\ref{MTT}.

\medskip
\noindent \textbf{Proof of Theorem~\ref{MT0}:} $ $ \\
Let $q$ be a connected non-negative unit form in $n \geq 1$ variables and of Dynkin type $\A_r$ for $r \geq 1$. Using Theorem~\ref{MAINlem} we find a connected loop-less quiver $Q$ with $m=r+1$ vertices and $n$ arrows such that $q=q_Q$. Consider the cycle type $\ct(q)=\ct(Q)$ and degree of degeneracy $\dd_q$ of $q$. Take $\vec{Q}=\vec{A}^{\dd_q}[\ct(q)]$, which is also a connected loop-less quiver with $m=r+1$ vertices and $n$ arrows (see Definition~\ref{DStan}).  Since $\ct(Q)=\ct(\vec{Q})$ by Remark~\ref{RStan}, we may assume that $\Lambda_Q=\Lambda_{\vec{Q}}$ (by replacing $Q$ by the quiver $\rho \cdot Q$ for some permutation $\rho$ of the set of vertices $Q_0$, see Remark~\ref{RPerm}). By Theorem~\ref{MTT} there is a matrix $B$ such that $I(Q)B=I(\vec{Q})$ and $\wc{G}_{\vec{Q}}=B^{\tr}\wc{G}_QB$. Taking $\vec{q}:=q_{\vec{Q}}$, by definition we have $\wc{G}_Q=\wc{G}_q$ and $\wc{G}_{\vec{Q}}=\wc{G}_{\vec{q}}$, and therefore, $q \approx^B \vec{q}$. The uniqueness of $\vec{q}$ follows from the uniqueness of the standard quiver $\vec{Q}$ with cycle type $\ct(q)$, $m$ vertices and $n$ arrows, cf. Corollary~\ref{STDunique}$(b)$.
\QED

\medskip
\noindent \textbf{Proof of Theorem~\ref{MT1}:} $ $ \\
Let $q$ and $\wt{q}$ be weakly Gram congruent connected non-negative unit forms in $n$ variables and of Dynkin type $\A_r$. If $q$ and $\wt{q}$ are strongly Gram congruent, then $\varphi_q=\varphi_{\wt{q}}$, see for instance~\cite[Lemma~4.6]{jaJ2020a}.

\medskip
Assume conversely that $\varphi_q=\varphi_{\wt{q}}$. Since $q \sim \wt{q}$, then $q$ and $\wt{q}$ have the same corank $c=\CRnk(q)=\CRnk(\wt{q})$. Using Theorem~\ref{MAINlem} we find quivers $Q$ and $\wt{Q}$ such that $q=q_Q$ and $\wt{q}=q_{\wt{Q}}$, both of which have $m=r+1$ vertices and $n$ arrows. By Remark~\ref{LAMBDA}$(b)$ we have
\[
\varphi_{q_Q}(\va)=(\va-1)^{c-1}p_{\Lambda_Q}(\va)=(\va-1)^{c-1}p_{\Lambda_{\wt{Q}}}(\va)=\varphi_{q_{\wt{Q}}}(\va).
\]
Therefore, the Coxeter-Laplacians $\Lambda_{Q}$ and $\Lambda_{\wt{Q}}$ are co-spectral. By Lemma~\ref{permutation} below, the matrices $\Lambda_Q$ and $\Lambda_{\wt{Q}}$ are conjugate, that is, there is a permutation matrix $P(\rho)$ such that $\Lambda_{\wt{Q}}=P(\rho)^{\tr}\Lambda_{Q}P(\rho)$. Thus, by Theorem~\ref{MAINlem}$(i)$ and replacing $Q$ by $\rho\cdot Q$ if necessary, we may assume that $\Lambda_Q=\Lambda_{\wt{Q}}$, see Remark~\ref{RPerm}. We conclude that $q \approx \wt{q}$ by Theorem~\ref{MAINlem}$(iii)$, and using the equivalence of $(i)$ and $(iii)$ in Theorem~\ref{MTT}.\QED

\medskip
\noindent \textbf{Proof of Theorem~\ref{MT2}:} $ $ \\
The construction of the cycle type $\ct$ is given in~\cite[Definition~4.2]{jaJ2020b}, see also~(\ref{EQQct}) above, where it is shown that $\ct:[\Quad_{\A}^c(n)/\approx ] \to \Part_1^c(m)$ is well-defined and surjective. The injectivity of $\ct$ is direct consequence of Theorem~\ref{MT1}, and the properties $(i-iii)$ were shown in~\cite[Theorem~6.3 and Corollary~6.4]{jaJ2020b}. Claim $(iv)$ is clear from definition, see Lemma~\ref{(P):radre}.

\medskip
That the reduced corank of $q$ is $\CRnk_{re}(q)=\ell-1$, where $\ell$ is the length of $\ct(q)$, was shown in Corollary~\ref{CreC}. That the multiplicity of $1$ as a root of $\varphi_q$ is $c+(\ell-1)$ can be easily derived from the shape of $\varphi_q$, cf.~\cite[Remark~7.1$(b)$]{jaJ2020b}. That is, the algebraic multiplicity of $1$ as eigenvalue of $\Phi_q$ is $\CRnk(q)+\CRnk_{re}(q)$, which shows $(v)$.
\QED

\subsection{Some general results} \label{AUX}

\noindent \textbf{On pure subgroups and orthogonal matrices} \label{AUX1}
\smallskip

We collect some well-known facts about pure subgroups, giving a sketch of the proofs.

\begin{lemma}\label{Rpure}
The following are equivalent for a subgroup $X$ of $\Z^m$.
\begin{itemize}
\itemsep=0.85pt
 \item[a)] The group $X$ is pure.
 \item[b)] The quotient $\Z^m/X$ is free.
 \item[c)] There is a subgroup $Y$ of $\Z^m$ such that $\Z^m=X \oplus Y$.
 \item[d)] The canonical inclusion $X \hookrightarrow \Z^m$ has a left inverse.
 \item[e)] If $f:A \to \Z^m$ is any morphism of abelian groups with $\Img(f)=X$, then for any morphism $h:A \to B$ of abelian groups with $h(\Ker(f)) = 0$ there is a morphism $g:\Z^m \to B$ (not necessarily unique) such that $h=gf$.
\end{itemize}
\end{lemma}
\begin{proof}
Assume that $X$ is pure. If $\bar{y}:=y+X \in \Z^m/X$ satisfies $a\bar{y}=0$ for some non-zero $a \in \Z$, then $ay \in X$, hence $y \in X$, that is, $\bar{y}=0$. This shows that $\Z^m/X$ is a torsion free (finitely generated) abelian group, thus free. Assume that $\Z^m/X$ is free, and fix a basis $\bar{y}_1,\ldots,\bar{y}_d$ of $\Z^m/X$ with $\bar{y}_t=y_t+X$ and $y_t \in \Z^m$ for $t=1,\ldots,d$. Take $Y=\langle y_1,\ldots,y_d \rangle$ the subgroup of $\Z^m$ generated by $y_1,\ldots,y_d$, and verify that $\Z^m=X \oplus Y$.  Assume that $\Z^m=X \oplus Y$, and take the inclusion $\sigma_X:X \to \Z^m$. Then the projection $\pi_X:X \oplus Y \to X$ satisfies $\pi_X\sigma_X=\Id_X$, that is, $\sigma_X$ is left invertible. Assume that $\sigma_X$ has a left inverse $\pi:\Z^m \to X$, and that $ax \in X$ for a non-zero $a \in \Z$. Then $ax=\pi(ax)=a\pi(x)$, which implies that $x=\pi(x)$, that is, $x \in X$, and $X$ is a pure subgroup of $\Z^m$. These arguments show that the statements $(a-d)$ are equivalent.

\medskip
Assume now that $(d)$ holds, and take functions $f$ and $h$ as in claim $(e)$. Denote by $\wt{f}$ the restriction of $f$ to its image, $\wt{f}:A \to \Img(f)$. Then $f=\sigma_X \wt{f}$, and therefore $\wt{f}=\pi f$, where $\sigma_X:X \to \Z^m$ is the canonical inclusion, and $\pi$ is a left inverse of $\sigma_X$. Since $\wt{f}$ is a cokernel of the inclusion $\Ker(f) \hookrightarrow A$, there is $\wt{h}:\Img(f) \to B$ such that $h=\wt{h}\wt{f}$. Then the following diagram is commutative,
\[
\xymatrix@R=1.5pc@C=3pc{\Ker(f) \ar@{^(->}[r] & A \ar[r]^-{f} \ar[rd]_-{\wt{f}} \ar@/_10pt/[rdd]_-{h} & \Z^m \ar[d]^-{\pi} \\
& & \Img(f) \ar[d]^-{\wt{h}} \\ & & B }
\]
Take $g:=\wt{h}\pi$ and observe that $gf=\wt{h}\pi f=\wt{h}\wt{f}=h$, as wanted. Finally, assume that $(e)$ holds, and take $f=\sigma_X$ and $h=\Id_X$. Since $\Ker(f)=0$ then $h(\Ker(f))=0$, and by hypothesis there is $g:\Z^m \to X$ such that $\Id_X=gf=f\sigma_X$. This shows the equivalence of claims $(d)$ and $(e)$, completing the proof.
\end{proof}

Recall that two permutation matrices $P$ and $P'$ are called \textbf{conjugate} if there is a permutation matrix $R$ such that $P'=R^{\tr}PR$.

\begin{lemma}\label{permutation}
The following are equivalent for permutations $\xi$ and $\xi'$ of the set $\{1,\ldots,m\}$.
\begin{itemize}
\itemsep=0.85pt
 \item[i)] $P(\xi)$ and $P(\xi')$ are conjugate permutation matrices.
 \item[ii)] The matrices $P(\xi)$ and $P(\xi')$ are co-spectral.
 \eject
 \item[iii)] The cycle types $\ct(\xi)$ and $\ct(\xi')$ coincide.
 \item[iv)] $\xi$ and $\xi'$ are conjugate permutations.
\end{itemize}
\end{lemma}
\begin{proof}
If $P(\xi)$ and $P(\xi')$ are conjugate permutation matrices, then they are similar matrices and their spectra coincide. Recall that the characteristic polynomial of $P(\xi)$ is given by
\[
p_{P(\xi)}(\va)=\prod_{t=1}^{\ell}(\va^{\pi_t}-1),
\]
where $\pi=(\pi_1,\ldots,\pi_{\ell})$ is the cycle type $\ct(\xi)=\pi$ of $\xi$ (cf.~\cite[\S 2.2]{jR06} or~\cite{jaJ2020b}). If $P(\xi)$ and $P(\xi')$ are co-spectral, then the cycle types of $\xi$ and $\xi'$ coincide (cf.~\cite[Algorithm~2 and Remark~7.4]{jaJ2020b}). In this case, $\xi$ and $\xi'$ are conjugate permutations (see~\cite[Proposition~2.33]{jR06}), which in turn shows that $P(\xi)$ and $P(\xi')$ are conjugate permutation matrices.
\end{proof}

\noindent \textbf{On non-negative integral quadratic forms} \label{AUX2}
\smallskip

Let $Z$ be a free finitely generated abelian group. By integral quadratic form $q$ on $Z$ we mean a function $q:Z \to \Z$ satisfying $q(ax)=a^2q(x)$ for any $a \in \Z$ and $x \in Z$, and such that its polarization $\mathbbm{b}_q:Z \times Z \to \Z$, given by $\mathbbm{b}_q(x,y):=q(x+y)-q(x)-q(y)$, is a (symmetric) bilinear form. The symmetric matrix of $\mathbbm{b}_q$ with respect to a fixed basis of $Z$ is denoted by $G_q$. Since the determinant of $G_q$ is independent of the chosen basis, it is referred to as \textbf{determinant} $\det(q)$ of $q$. The \textbf{radical} of $q$ is the kernel (or right null space) of $G_q$.  Since $q(x)=0$ for any $x \in \rad(q)$, the induced function $q/\rad(q):Z/\rad(q) \to \Z$ is well defined.

\begin{lemma}\label{AuxLem}
Let $q:Z \to \Z$ be an integral quadratic form.
\begin{itemize}
 \item[i)] The form $q$ has zero radical (that is, $\rad(q)=0$) if and only if $\det(q)\neq 0$.
 \item[ii)] If $\wt{Z}$ is a free finitely generated abelian group, and $B:\wt{Z} \to Z$ is a linear transformation, then $qB:\wt{Z} \to \Z$ is an integral quadratic form.
 \item[iii)] If $\wt{Z}$ is a free finitely generated abelian group, $\wt{q}:\wt{Z} \to \Z$ is a function and $\wt{B}:Z \to \wt{Z}$ is a surjective linear transformation such that $q=\wt{q}\wt{B}$, then $\wt{q}$ is an integral quadratic form.
 \item[iv)] The radical $\rad(q)$ is a pure subgroup of $Z$, and the function $q/\rad(q):Z/\rad(q) \to \Z$ is an integral quadratic form with radical zero.
 \item[v)] If $q$ is non-negative, then $x \in Z$ is a radical vector of $q$ if and only if $q(x)=0$.
\end{itemize}
\end{lemma}
\begin{proof}
Fix a basis $(\bas_1,\ldots,\bas_n)$ of $Z$, and let $G_q$ be the Gram matrix of $\mathbbm{b}_q$ with respect to such basis. Then $\det(G_q)=0$ if and only if there is a trivial integral linear combination of the rows of $G_q$, that is, if there is a vector $z=(z_1,\ldots,z_n) \in \Z^n$ such that $G_qz=0$ (indeed, use Gaussian elimination over the rational numbers, then multiply any solution by a common multiple of the denominators to get integer coefficients). Then claim $(i)$ is clear, since $x=\sum_{i=1}^nz_i\bas_i \in Z$ is a radical vector of $q$.

\medskip
For $(ii)$ take $\wt{q}:=qB$. The polarization of $\wt{q}$ is given by $\mathbbm{b}_{\wt{q}}(\wt{x},\wt{y})=\mathbbm{b}_q(B\wt{x},B\wt{y})$ for vectors $\wt{x},\wt{y} \in \wt{Z}$. Then the bilinearity of $\mathbbm{b}_{\wt{q}}$ follows from the linearity of $B$ and the bilinearity of $\mathbbm{b}_q$. Moreover, for any $a \in \Z$ and $\wt{x} \in \wt{Z}$ we have
\[
\wt{q}(a\wt{x})=q(Ba\wt{x})=q(aB\wt{x})=a^2q(B\wt{x})=a^2\wt{q}(\wt{x}),
\]
since $B$ is linear and $q$ is an integral quadratic form.

\medskip
To show $(iii)$ observe that the polarization of $\wt{q}$ has the following shape, for vectors $\wt{x},\wt{y} \in \wt{Z}$ and vectors $x,y \in Z$ such that $\wt{B}x=\wt{x}$ and $\wt{B}y=\wt{y}$,
\[
\mathbbm{b}_{\wt{q}}(\wt{x},\wt{y})=\mathbbm{b}_{\wt{q}}(\wt{B}x,\wt{B}y)=\mathbbm{b}_q(x,y).
\]
Then a standard calculation, using that $\wt{B}$ is linear and surjective and that $\mathbbm{b}_q$ is bilinear, shows that $\mathbbm{b}_{\wt{q}}$ is bilinear. Moreover, for arbitrary $a \in \Z$ and $\wt{x} \in \wt{Z}$, there is $x \in Z$ such that $\wt{B}x=\wt{x}$ (since $\wt{B}$ is surjective), and we have
\[
\wt{q}(a\wt{x})=\wt{q}(a\wt{B}x)=\wt{q}(\wt{B}(ax))=q(ax)=a^{2}q(x)=a^2\wt{q}(\wt{B}x)=a^2\wt{q}(\wt{x}),
\]
since $q$ is a quadratic form and $\wt{B}$ is linear.

\medskip
To show $(iv)$, since $\rad(q)$ is the kernel of $\mathbbm{b}_q$ then it is a pure subgroup of $Z$. Moreover, observe that $x \in \rad(q)$ if and only if $q(x+y)=q(y)$ for any $y \in Z$, which shows that the induced function $q/\rad(q)$ is well defined. If $\wt{B}:Z \to Z/\rad(q)$ is the canonical surjection, then $q=(q/\rad(q))(\wt{B})$ and $\wt{q}$ is an integral quadratic form by $(iii)$.

\medskip
The proof of $(v)$ is a simple generalization of~\cite[Lemma~2.1$(b)$]{jaJ2020a}. Indeed, if $x \in \rad(q)$ then $q(x)=\frac{1}{2}x^{\tr}G_qx=0$ (here we identify the elements $x \in Z$ with their coordinate vectors in $\Z^n$ in the fixed basis). Conversely, if $q$ is non-negative and $q(x)=0$, then for an arbitrary $m \in \Z$ and a basis $\bas_1,\ldots,\bas_n$ of $Z$ we have
\[
0 \leq q(mx+\bas_i)=m\mathbbm{b}_q(x,\bas_i)+q(\bas_i).
\]
Since $m$ is arbitrary, then $\mathbbm{b}_q(x,\bas_i)=0$, and this holds for any $i\in \{1,\ldots,n\}$. That is, $x \in \rad(q)$.
\end{proof}

\begin{lemma}\label{Linv}
Let $q:Z \to \Z$ be a non-negative integral quadratic form, $B:\wt{Z} \to Z$ a linear transformation, and take $\wt{q}:=qB$. Then $B$ is an isomorphism if and only if the following conditions hold,
\begin{itemize}
 \item[a)] the transformation $B$ restricts to an isomorphism $B|_{\rad}:\rad(\wt{q}) \to \rad(q)$;
 \item[b)] the determinants of $q/\rad(q)$ and $\wt{q}/\rad(\wt{q})$ coincide.
\end{itemize}
\end{lemma}
\begin{proof}
By Lemma~\ref{AuxLem}$(ii)$, $\wt{q}=qB$ is an integral quadratic form, which is clearly non-negative. By non-negativity, $B$ restricts to a transformation $B|_{\rad}:\rad(\wt{q}) \to \rad(q)$. Indeed, if $\wt{x} \in \rad(\wt{q})$ then $q(B(\wt{x}))=\wt{q}(\wt{x})=0$, and by Lemma~\ref{AuxLem}$(v)$ the vector $B(\wt{x})$ is in the radical of $q$. In particular, $B$ induces a linear transformation $[B]:\wt{Z}/\rad(\wt{q}) \to Z/\rad(q)$ such that the following diagram commutes,
\[
\xymatrix@R=1pc{\wt{Z}/\rad(\wt{q}) \ar[rrd]^-{\wt{q}/\rad(\wt{q})} \ar[dd]_-{[B]} \\ & & \Z. \\ Z/\rad(q)  \ar[rru]_-{q/\rad(q)} }
\]

Assume first that $B$ is an isomorphism. Then $q=\wt{q}B^{-1}$ and since $\wt{q}$ is also non-negative, the restriction $B^{-1}|_{\rad}:\rad(q) \to \rad(\wt{q})$ is inverse of $B|_{\rad}$. Since $\rad(q)$ is a direct summand of $Z$, and so is $\rad(\wt{q})$ of $\wt{Z}$ (Lemma~\ref{Rpure}), then $[B]$ is an isomorphism between free finitely generated abelian groups, which implies that $\det(\wt{q}/\rad(\wt{q}))=\det(q/\rad(q))\det([B])=\det(q/\rad(q))$. This shows that $(a)$ and $(b)$ hold.

\medskip
For the converse, by Lemma~\ref{AuxLem}$(i,v)$ we have $\det(q/\rad(q)) \neq 0$ and $\det(\wt{q}/\rad(\wt{q}))\neq 0$. By $(b)$, this implies that $\det([B])=\pm 1$, that is, that $[B]$ is an isomorphism. Assume now that $0=B\wt{x}$ for some $\wt{x} \in \wt{Z}$. Then $[B](\wt{x}+\rad(\wt{q}))=0$, which implies that $\wt{x} \in \rad(\wt{q})$ since $[B]$ is injective. Since $B|_{\rad}$ is injective by $(a)$, we have $\wt{x}=0$, that is, $B$ is injective. Take now $x \in Z$ arbitrary. Since $[B]$ is surjective, there is $\wt{x} \in \wt{Z}$ such that $B(\wt{x}+\rad(\wt{q}))=x+\rad(q)$, that is, $x-B\wt{x} \in \rad(q)$ since $B(\rad(\wt{q})) \subseteq \rad(q)$. Using that $B|_{\rad}$ is surjective, again by $(a)$ there is $\wt{y} \in \rad(\wt{q})$ such that $B\wt{y}=x-B\wt{x}$, that is, $x=B(\wt{x}+\wt{y})$. This shows the surjectivity of $B$, completing the proof.
\end{proof}

\begin{corollary}\label{Cinv}
Let $q$ and $\wt{q}$ be connected non-negative unit forms in the same number of variables.  Then the following hold:
\begin{itemize}
 \item[i)] The forms $q$ and $\wt{q}$ have the same Dynkin type if and only if the determinants of $q/\rad(q)$ and $\wt{q}/\rad(\wt{q})$ coincide.
 \item[ii)] If $q$ and $\wt{q}$ have the same Dynkin type and $\wt{q}=qB$ for an integer matrix $B$, then $B$ is $\Z$-invertible if and only if $B$ restricts to an isomorphism $B|_{\rad}:\rad(\wt{q}) \to \rad(q)$.
\end{itemize}
\end{corollary}
\begin{proof}
To show $(i)$ assume first that $q$ and $\wt{q}$ have the same Dynkin type. Then they are weakly Gram congruent since they have the same number of variables, that is, there is a $\Z$-invertible matrix $B$ with $\wt{q}=qB$. By Lemma~\ref{Linv}$(b)$ the determinants of $q/\rad(q)$ and $\wt{q}/\rad(\wt{q})$ coincide. Conversely, assume that $q$ and $\wt{q}$ have different Dynkin types, say $\Delta$ and $\wt{\Delta}$. By~\cite[Theorem~3.15]{BJP19}, the forms $q/\rad(q)$ and $q_{\Delta}$ are equivalent, therefore, they have the same determinant. The same holds for $\wt{q}/\rad(\wt{q})$ and $q_{\wt{\Delta}}$. Since $q$ and $\wt{q}$ have the same number of variables, and $\Delta \neq \wt{\Delta}$, then $\det(q_{\Delta}) \neq \det(q_{\wt{\Delta}})$ (see for instance~\cite[Corollary~3.10$(b)$]{dS20}), which shows that the determinants of $q/\rad(q)$ and $\wt{q}/\rad(\wt{q})$ are different.

\medskip
Claim $(ii)$ follows from $(i)$ and Lemma~\ref{Linv}$(b)$.
\end{proof}

\begin{lemma}\label{LinDet}
Let $q$ and $\wt{q}$ be non-negative unit forms with $q \approx^B \wt{q}$ for a $n \times n$ matrix $B$. Then $B$ restricts to isomorphisms
\begin{eqnarray*}
B|_{\rad}&:& \rad(\wt{q}) \to \rad(q), \\
B|_{\rad_{re}}&:& \rad_{re}(\wt{q}) \to \rad_{re}(q).
\end{eqnarray*}
In particular, the standard morsification $\wc{\mathbbm{b}}_q$ has pure restriction $\wc{\mathbbm{r}}_q$ to its radical if and only if $\wc{\mathbbm{b}}_{\wt{q}}$ has pure restriction $\wc{\mathbbm{r}}_{\wt{q}}$ to its radical.
\end{lemma}
\begin{proof}
That $B|_{\rad}$ is an isomorphism was shown in Lemma~\ref{Linv}. The second isomorphism holds since $\wc{\mathbbm{b}}_{\wt{q}}(x,y)=\wc{\mathbbm{b}}_q(Bx,By)$ for any $x,y \in \Z^n$, and $B$ is $\Z$-invertible. The claim on pure restriction is clear since both $B|_{\rad}$ and $B|_{\rad_{re}}$ are isomorphisms, and $\rad_{re}(q)$ is the kernel of the restriction $\wc{\mathbbm{r}}_q$ of $\wc{\mathbbm{b}}_q$ to its radical.
\end{proof}

\noindent \textbf{Some consequences} \label{AUX3}
\smallskip

The following technical observation, basic for our work, seems to hold for more general contexts (for other Dynkin types and for other unimodular morsifications). The author was not able to find any related results in the literature, nor any more general proofs.

\begin{corollary}\label{CinDet}
Let $q$ be a connected non-negative unit form of Dynkin type $\A_{r}$ for $r \geq 1$. Then the associated upper triangular bilinear form $\wc{\mathbbm{b}}_q$ (the standard morsification of $q$) has pure restriction to its radical.
\end{corollary}
\begin{proof}
Follows from Lemma~\ref{LinDet}, Theorem~\ref{MT0} and Lemma~\ref{(L):stan}.
\end{proof}

\begin{remark}\label{Rgeneral}
In the proofs of Propositions~\ref{(P):inv} and~\ref{(P):strong}, the only needed condition of the standard quiver $\vec{Q}$ is that the upper triangular bilinear form $\wc{\mathbbm{b}}_{q_{\vec{Q}}}$ has pure restriction to its radical (Lemma~\ref{(L):stan}). As consequence of Corollary~\ref{CinDet}, the same constructions of Propositions~\ref{(P):inv} or~\ref{(P):strong} hold when replacing $\vec{Q}$ for an arbitrary quiver $\wt{Q}$ satisfying $Q \pseudo^B \wt{Q}$ for a square or a $\Z$-invertible matrix $B$, respectively.
\end{remark}

Observe that a loop-less quiver $Q$ and its inverse $Q^{\dagger}$ have the same Coxeter polynomial. Indeed, using that $I(Q^{\dagger})=I(Q)\wc{G}_Q^{-1}$ and that $\wc{G}_{Q^{\dagger}}=\wc{G}_Q^{-1}$ (cf.~\cite[Proposition~4.4]{jaJ2020a}), a direct calculation yields $\Phi_Q^{\tr}=\Phi_{Q^{\dagger}}$. Then the following result, consequence of Theorem~\ref{MT1}, helps us find explicit congruences between the upper and lower triangular Gram matrices for the class of unit forms considered in this paper (see~\cite{HS04} for related problems).

\begin{corollary}\label{CinvTr}
Let $q$ be a connected non-negative unit form of Dynkin type $\A_r$ for some $r \geq 1$, consider the upper triangular Gram matrix $\wc{G}_q$ and take $q^{\dagger}:=q\wc{G}_q^{-1}$. Then $q^{\dagger}$ is a connected unit form with $q^{\dagger} \approx q$. Moreover, any strong Gram congruence $q \approx^B q^{\dagger}$ determines a congruence
\[
C^{\tr}\wc{G}_qC=\wc{G}_q^{\tr},
\]
by taking $C=B\wc{G}_q$.
\end{corollary}
\begin{proof}
By Theorem~\ref{MAINlem}, there is a connected loop-less quiver $Q$ such that $q=q_Q$. Taking $Q^{\dagger}$ as in~(\ref{EQQ3}), that is, $I(Q^{\dagger})=I(Q)\wc{G}_Q^{-1}$, then
\[
q_{Q^{\dagger}}(x)=\frac{1}{2}||I(Q^{\dagger})x||^2=\frac{1}{2}||I(Q)\wc{G}_Q^{-1}x||^2=q_Q(\wc{G}_Q^{-1}x)=q^{\dagger}(x).
\]
Therefore, using again Theorem~\ref{MAINlem}, $q^{\dagger}$ is a non-negative connected unit form of Dynkin type $\A$. Clearly, $q \sim^{\wc{G}^{-1}_q} q^{\dagger}$, and since $\wc{G}_q$ is upper triangular then so is its inverse, which shows that $\wc{G}_{q^{\dagger}}=\wc{G}_q^{-1}$. Note that
\[
\Phi_{q^{\dagger}}=-\wc{G}_{q^{\dagger}}^{\tr}\wc{G}_{q^{\dagger}}^{-1}=-\wc{G}_q^{-\tr}\wc{G}_q=(-\wc{G}_q^{\tr}\wc{G}_q^{-1})^{\tr}=\Phi_q^{\tr}.
\]
In particular, the Coxeter polynomials of $q$ and $q^{\dagger}$ coincide. Then $q \approx q^{\dagger}$ by Theorem~\ref{MT1}.  Finally, if $q \approx^B q^{\dagger}$ and we take $C=B\wc{G}_q$, then
\[
C^{\tr}\wc{G}_qC=\wc{G}_q^{\tr}B^{\tr}\wc{G}_qB\wc{G}_q=\wc{G}_q^{\tr}\wc{G}_q^{-1}\wc{G}_q=\wc{G}_q^{\tr},
\]
as claimed.
\end{proof}

We end this section with a numerical strong Gram classification of non-negative connected unit forms of Dynkin type $\A_r$ and some comments on the number of corresponding classes. Recall that $\Quad^c_{\A}(n)$ denotes the set of connected non-negative unit forms on $n \geq 1$ variables having Dynkin type $\A_{n-c}$ and corank $c \geq 0$.

\begin{corollary}\label{CnumCla}
Let $q,q':\Z^n \to \Z$ be non-negative connected unit forms of Dynkin type $\A_r$, $r\geq 1$. Then $q \approx q'$ if and only if $\ct(q)=\ct(q')$. Moreover,
  \begin{equation*}
|\Quad^c_{\A}(n)/\approx|= \left\{ \begin{array}{c l}
1, & \text{if $c=0$},\\
\left\lfloor \frac{n}{2} \right\rfloor, & \text{if $c=1$},\\
\left\lfloor \frac{(n-1)^2+15}{12} \right\rfloor, & \text{if $c=2$}.
\end{array} \right.
\end{equation*}
\end{corollary}
\begin{proof}
Observe first that $\CRnk(q)=\CRnk(q')=n-r$, and since $\Dyn(q)=\Dyn(q')=\A_r$, then $q$ and $q'$ are weakly Gram congruent (cf.~\cite{BP99}). Then the main claim follows directly by Theorem~\ref{MT2}, where it is shown that
\[
|[\Quad^c_{\A}(n)/\approx ]|=|\Part^c_1(m)|, \quad \text{where $\Part^c_1(m):=\{\pi \vdash m \mid 0 \leq c-(\ell(\pi)-1) \equiv 0 \mod 2\}$.}
\]
Denoting by $p_{\ell}(m)$ the number of partition of $m$ having exactly $\ell$ parts, then
\begin{equation}\label{EQQpart}
|\Part^c_1(m)|=\sum_{\dd =0}^{\lfloor c/2 \rfloor} p_{c-2\dd+1}(m).
\end{equation}
Using that $p_1(m)=1$ and $p_2(m)=\lfloor m/2 \rfloor$ for all $m \geq 2$, and that $m=n-c+1$, we get the claimed values of $|\Quad^c_{\A}(n)/\approx|$ for $c=0,1$.  If $c=2$ then $|\Part^2_1(m)|=p_3(n-1)+p_1(n-1)$. It can be easily shown that this coincides with the number of partitions of $n-1$ into $3$ or fewer distinct parts, a number known to be given by $\lfloor (n-1)^2/12+5/4 \rfloor$, cf. entry $A014591$ in~\cite{njaS09}.
\end{proof}

In general, the value of $|\Part^c_1(m)|$ can be found using~(\ref{EQQpart}) and the well-known recursive formula for $p_{\ell}(m)$, see for instance~\cite[pp. 345--348]{jA11},
\[
p_{\ell}(m)=p_{\ell-1}(m-1)+p_{\ell}(m-\ell),
\]
subject to the starting conditions $p_0(0)=1$ and $p_{\ell}(m)=0$ if $\ell \leq 0$, $m\leq 0$ and $\ell+m<0$.

\medskip
\noindent \textbf{Comments on generalizations and future work} \label{AUX3}
\smallskip

Following Simson's work, there are many interesting problems to consider in the setting of non-negative unit forms of arbitrary corank, for instance: the computation of isotropy mini-groups and Weyl group actions, the description of morsifications, Coxeter-translation quivers and mesh geometries, and applications to quadratic forms associated to posets. The combinatorial framework explored here is certainly useful for the analysis of such problems, at least for the class of connected non-negative unit forms of Dynkin type $\A_r$.

\medskip
As already mentioned in~\cite{jaJ2018}, the ideas presented here may be generalized to cover the Dynkin type $\D_s$ for $s \geq 4$, replacing loop-less quivers by certain loop-less bidirected graphs (in the sense of Zaslavsky~\cite{tZ08}, see also~\cite{tZ81}), satisfying certain cycle condition (cf.~\cite[\S 8]{jaJ2018}). Most of the constructions and results are similar to those for quivers, the main challenge now being the choice of an adequate family of representatives of strong Gram congruence (the corresponding standard unit forms of Dynkin type $\D_s$).  By admitting loops we include semi-unit forms, as well as a class of non-unitary connected non-negative quadratic forms to be classified yet. The seminal paper by Cameron, Goethals, Seidel and Shult~\cite{CGSS}, connecting classical root systems with the spectral analysis of signed graphs, contains fundamental ideas suitable for further generalizations (see also~\cite{tZ81}).

\medskip
In the following lemma we consider some of the ideas that originated the results of this paper, presented in a slightly more general context.
\begin{lemma}\label{Lgeneral}
Let $q,\wt{q}:\Z^n \to \Z$ be non-negative connected unit forms with $\CRnk(\wt{q})\leq 1$. Then the following are equivalent for a matrix $B \in \mathbb{M}_n(\Z)$,
\begin{itemize}
 \item[a)] $B$ is $\Z$-invertible and $q \approx^B \wt{q}$;
 \item[b)] $B$ satisfies $\wt{q}=qB$ and $G_{\wt{q}}B^*B=G_{\wt{q}}$, where $B^*:=\wc{G}_{\wt{q}}^{-1}B^{\tr}\wc{G}_q$.
\end{itemize}
\end{lemma}
\begin{proof}
Note first that if $(a)$ holds then $\wt{q}=qB$ and, by definition of $q \approx^B \wt{q}$, we have $B^*B=\Id$. Assume now that $(b)$ holds, which does not require the matrix $B$ to be $\Z$-invertible. Note that $\wt{q}=qB$ implies that $B^{\tr}G_qB=G_{\wt{q}}$, which shows that the matrix $M:=B^{\tr}\wc{G}_qB-\wc{G}_{\wt{q}}$ is skew-symmetric (recall that $G_q=\wc{G}_q+\wc{G}^{\tr}_q$ and similarly for $\wt{q}$). Hence, the rank of $M$ is a non-negative even number (cf.~\cite[XI, \S 4]{frG}). Since $\wc{G}_{\wt{q}}$ is $\Z$-invertible, then the rank of $N:=\wc{G}_{\wt{q}}^{-1}M=B^*B-\Id_n$ is also a non-negative even number. However, the condition $G_{\wt{q}}B^*B=G_{\wt{q}}$ guarantees that the columns of $N$ are radical vectors of $\wt{q}$, since
\[
G_{\wt{q}}N=G_{\wt{q}}[B^*B-\Id_n]=0.
\]
Then $\Rnk(N)\leq \CRnk(\wt{q}) \leq 1$, which shows that $N=0=M$, that is, $B^{\tr}\wc{G}_qB=\wc{G}_{\wt{q}}$. In particular, $\det(B)=\pm 1$ and $q \approx^B \wt{q}$, as claimed.
\end{proof}

In our context of Dynkin type $\A_r$ via connected loop-less quivers, any matrix $B$ satisfying the condition $(a)$ of Definition~\ref{DEF}, namely $Q \pseudo^B \wt{Q}$, satisfies $q_{\wt{Q}}=q_QB$ and $G_{\wt{Q}}B^*B=G_{\wt{Q}}$. In particular, if $\CRnk(q_{\wt{Q}})\leq 1$, then $B$ determines a strong Gram congruence $q_Q \approx^B q_{\wt{Q}}$, as claimed in the introduction (see comments after Step~1 on page~\pageref{StepOne}).

\subsection{Hints for an implementation}

We end the paper with some ideas and suggestions for an implementation of our main results, solving Problem~2$(ii)$ for the class of connected non-negative unit forms of Dynkin type $\A_r$. All algorithms are straightforward, and make use of well-known methods as \texttt{least squares} and \texttt{depth-first search}. The main construction is presented in Algorithm~\ref{AA} (resp. Algorithm~\ref{AB}), where for a connected non-negative unit form of Dynkin type $\A_{m-1}$ (resp. a connected loop-less quiver on $m$ vertices) we compute a matrix that determines a strong Gram congruence to the corresponding standard unit form (resp. standard quiver) representative of class, see Theorem~\ref{MT0} and Definition~\ref{DEF}. This construction is based on the correction Algorithms~\ref{A2} and~\ref{A3}, corresponding to Propositions~\ref{(P):inv} and~\ref{(P):strong} respectively, and on further auxiliary methods given in Algorithms~\ref{A6}, \ref{A10} and~\ref{A11}. An algorithmic approach to the \texttt{skew normal form} may be found in~\cite{VF21} or~\cite[Theorem~IV.1]{Newman:1972}.\smallskip

\begin{algorithm}\label{AA}$ $ Standard solution for quadratic forms. \par \smallskip
\textbf{Input:} A connected non-negative unit form $q$ in $n$ variables and of Dynkin type $\A_{r}$ for $r \geq 1$.

\textbf{Output:} A $n \times n$ matrix $B$ such that
\[
\wc{G}_{\vec{q}}=B^{\tr}\wc{G}_{q}B,
\]
where $\vec{q}$ is the standard non-negative unit form weakly congruent to $q$ and satisfying $\varphi_{\vec{q}}=\varphi_q$.
\end{algorithm}

 \underline{Step 1.} Find a quiver $Q$ with $n$ arrows and $m=r+1$ vertices such that $q=q_Q$ (see~\cite[Algorithm~1]{jaJ2020b}).

\underline{Step 2.} Apply Algorithm~\ref{AB} below to find a matrix $B$ satisfying $\wc{G}_{\vec{Q}}=B^{\tr}\wc{G}_{\rho \cdot Q}B$, where $\vec{Q}$ is a standard quiver and $\rho$ is a permutation such that $\Lambda_{\rho \cdot Q}=\Lambda_{\vec{Q}}$ (cf. Remark~\ref{RPerm}). Since $\wc{G}_q=\wc{G}_Q$, then $B$ is the wanted strong Gram congruence matrix, by taking $\vec{q}:=q_{\vec{Q}}$. Return matrix $B$.\smallskip

\begin{algorithm}\label{AB}$ $ Standard solution for quivers. \par \smallskip
\textbf{Input:} A connected loop-less quiver $Q$ with $n \geq 1$ arrows and $m \geq 2$ vertices.

\textbf{Output:} A permutation $\rho$ of $Q_0=\{1,\ldots,m\}$ and a $n \times n$ matrix $B$ such that $I(\rho \cdot Q)B=I(\vec{Q})$, and
\[
\wc{G}_{\vec{Q}}=B^{\tr}\wc{G}_{Q}B,
\]
where $\vec{Q}$ is the standard quiver with $\ct(\vec{Q})=\ct(Q)$ and same number of arrows as $Q$.
\end{algorithm}

 \underline{Step 1.} Compute the permutation of vertices $\xi^-_Q$ of $Q$. [Hint: determine the structural walks $\alpha_Q^-(v)$ for $v=1,\ldots,m$, see definition~(\ref{DEFperm}), or compute directly the Coxeter-Laplacian $\Lambda_Q$ of $Q$ using~(\ref{EQQ5})].

 \underline{Step 2.} Compute the cycle type $\ct(Q)$ of $Q$ by considering the cardinalities of the orbits of $\xi^-_Q$ (see also~\cite[Algorithm~2]{jaJ2020b}), and let $\ell$ be the length of $\ct(Q)$.

 \underline{Step 3.} Consider the degree of degeneracy $\dd=\frac{1}{2}(c-\ell+1)$ of $q_Q$, where $c=n-m+1$ is the corank of $q_Q$, and take the standard quiver $\vec{Q}:=\vec{\A}^{\dd}[\ct(Q)]$.

 \underline{Step 4.} Determine a permutation $\rho$ such that $\xi^-_Q\circ \rho=\xi^-(\vec{Q})$.

 \underline{Step 5.} Take $\wt{Q}:=\rho \cdot Q$ By Step~4 we have $\Lambda_{\wt{Q}}=\Lambda_{\vec{Q}}$. Find a matrix $B'$ such that $I(\wt{Q})B'=I(\vec{Q})$ and $I(\vec{Q})(B')^{*}=I(\wt{Q})$, that is, $\wt{Q} \pseudo^{B'} \vec{Q}$ [Hint: using the structural walks of Step~1 and apply equations~(\ref{EQQa}) and~(\ref{EQQb}) from the proof of Proposition~\ref{Ls}; we also need to find arbitrary connecting walks $\delta(t)$, for which we may use, for instance, a~\texttt{depth-first search} algorithm].

 \underline{Step 6.} Apply Algorithm~\ref{A2} below to the matrix $B'$ of Step~5 to find a matrix $M$ such that $B'+M$ is $\Z$-invertible and $\wt{Q} \pseudo^{B'+M} \vec{Q}$.

 \underline{Step 7.} Apply Algorithm~\ref{A3} below to the matrix $B'+M$ of Step~6 to find a matrix $C$ such that if $B:=(B'+M)C$, then $I(\wt{Q})B=I(\vec{Q})$ and $[B]^*[B]=\Id$, as wanted. Return the pair $(\rho,B)$.

\smallskip

Using Corollary~\ref{CinDet}, the following correction Algorithms~\ref{A2} and~\ref{A3}, based on the proofs of Propositions~\ref{(P):inv} and~\ref{(P):strong} respectively, work for arbitrary connected loop-less quivers $Q$ and $\wt{Q}$ with same number of vertices and arrows (see Remark~\ref{Rgeneral}).

\smallskip
\begin{algorithm}\label{A2}Invertible pseudo-morphism.\par\smallskip
\textbf{Input:} Two connected loop-less quivers $Q$ and $\wt{Q}$ with the same number of arrows $n$, and a $n \times n$ matrix $B$ such that $Q \pseudo^B \wt{Q}$.

\textbf{Output:} A $n \times n$ matrix $M$ such that $B+M$ is $\Z$-invertible and $Q \pseudo^{B+M} \wt{Q}$.
\end{algorithm}

 \underline{Step 1.} Fix kernel matrices $K$ and $\wt{K}$ of $I(Q)$ and $I(\wt{Q})$ respectively, and consider the skew-symmetric matrices $W=K^{\tr}\wc{G}_QK$ and $\wt{W}=\wt{K}^{\tr}\wc{G}_{\wt{Q}}\wt{K}$.

 \underline{Step 2.} Make sure that the kernel matrix $K$ has the shape $K=[K' , K'']$, where $K''$ is a kernel matrix of the reduced radical $\rad_{re}(q_Q)$, and assume that $\wt{K}$ has a similar shape $\wt{K}=[\wt{K}',\wt{K}'']$.

 \underline{Step 3.} Find a matrix $L$ such that $B\wt{K}=KL$ (for instance, using a \texttt{least squares} algorithm). As indicated by the partitions $K=[K' , K'']$ and $\wt{K}=[\wt{K}',\wt{K}'']$, the matrices $L$ and $\wt{W}$ have the following shapes,
 \[
 L=\begin{pmatrix} L_1&0\\L_2&L_3 \end{pmatrix} \quad \text{and} \quad \wt{W}=\begin{pmatrix} \wt{W}'&0\\0&0 \end{pmatrix}.
 \]
By Corollary~\ref{CinDet}, the matrix $\wt{W}'$ is $\Z$-invertible.

 \underline{Step 4.} Take $Y_1:=(L_1-\Id)(\wt{W}')^{-1}$ and $Y_2:=L_2(\wt{W}')^{-1}$. Take also $Y:=\begin{pmatrix} Y_1&0\\Y_2&0 \end{pmatrix}$ and $M:=KY\wt{K}^{\dagger \tr}$, where $\wt{K}^{\dagger}:=\wc{G}_{\wt{Q}}\wt{K}$. Then $B+M$ satisfies the claim, as in Proposition~\ref{(P):inv}. Return matrix $M$.

\smallskip
\begin{algorithm}\label{A3}Strong congruence matrix.\par\smallskip
\textbf{Input:} Two connected loop-less quivers $Q$ and $\wt{Q}$ with the same number of arrows $n$, and a $\Z$-invertible $n \times n$ matrix $B$ such that $Q \pseudo^B \wt{Q}$.

\textbf{Output:} A $\Z$-invertible $n \times n$ matrix $C$ such that $I(Q)(BC)=I(\wt{Q})$ and $(BC)^*(BC)=\Id$.
\end{algorithm}

 \underline{Step 1.} Fix kernel matrices $K$ and $\wt{K}$ of $I(Q)$ and $I(\wt{Q})$ as in Steps~1 and~2 of Algorithm~\ref{A2}.

 \underline{Step 2.}  Take $B^*:=\wc{G}_{\wt{Q}}^{-1}B^{\tr}\wc{G}_Q$, and use Algorithm~\ref{A6} below to compute $Z:=\Xi(B^{*}B)$ and $\wt{Z}:=\Xi((B^{*}B)^{-1})$.

 \underline{Step 3.} Apply Algorithm~\ref{A10} below to the matrix $\wt{Z}$ of Step~2 in order to find a matrix $Y$ such that
 \[
 \wt{Z}=Y-Y^{\tr}+Y^{\tr}\wt{W}Y,
 \]
 where $\wt{W}:=\wt{K}^{\tr}\wc{G}_{\wt{Q}}\wt{K}$.

 \underline{Step 4.} Define $C:=\Id-\wt{K}Y^{\tr}\wt{K}^{\dagger \tr}$ where $Y$ is the matrix of Step~3, and $\wt{K}^{\dagger}:=\wc{G}_{\wt{Q}}\wt{K}$. Then $BC$ satisfies the claim as in the proof of Proposition~\ref{(P):strong}. Return matrix $C$.

\smallskip

The following are auxiliary constructions for Algorithm~\ref{A3}.

\smallskip
\begin{algorithm}\label{A6}Construction of the bijection $\Xi$ of Lemma~\ref{(M):endo}.\par\smallskip
\textbf{Input:} A connected loop-less quiver $Q$ with $n$ arrows, a fixed $n \times c$ kernel matrix $K$ of $I(Q)$, and a pseudo-endomorphism $B$ of $Q$.

\textbf{Output:} A $c \times c$ matrix $Z:=\Xi(B)$ such that $B=\Id+KZK^{\dagger\tr}$, where $K^{\dagger}:=\wc{G}_QK$.
\end{algorithm}

 \underline{Step 1.} Find a solution $L$ to the equation $B-\Id=KL$ (for instance, using a \texttt{least squares} algorithm).

 \underline{Step 2.} Find similarly a solution $Z$ to the equation $L=ZK^{\dagger \tr}$ where $L$ is the matrix of Step~1. Return matrix $Z$.

\smallskip
\begin{algorithm}\label{A10} Special decomposition of a skew-symmetric matrix.\par\smallskip
\textbf{Input:} A pair $(Z,W)$ of skew-symmetric $c \times c$ matrices, where $W$ is pure and in \texttt{skew normal form}.

\textbf{Output:} A $c \times c$ matrix $Y$ such that $Z=Y^{\tr}-Y+Y^{\tr}WY$.
\end{algorithm}

 \underline{Step 1.} Assume first that $W$ is $\Z$-invertible and take $\wt{Z}:=Z+W$. Using Algorithm~\ref{A11} below we find a matrix $\wt{Y}$ such that $\wt{Z}=\wt{Y}^{\tr}W\wt{Y}$, and take $Y:=\wt{Y}-W$ (see the first part of the proof of Lemma~\ref{(I):skewsym}$(b)$).

 \underline{Step 2.} Assume now that $W$ is not $\Z$-invertible. Then $W=W'\oplus \mathbf{0}$ for a $\Z$-invertible skew-symmetric matrix $W'$ and $\mathbf{0}$ a zero matrix of adequate size, and $Z$ has the following shape
 \[
 Z=\begin{pmatrix} Z_1& -Z_2^{\tr} \\Z_2&Z_3 \end{pmatrix}.
 \]

 \underline{Step 3.} Since $Z_1$ and $W'$ are skew-symmetric, and $W'$ is $\Z$-invertible, using Step~1 we find a matrix $Y_1$ such that $Z_1=Y_1-Y_1^{\tr}+Y^{\tr}_1W'Y_1$.

 \underline{Step 4.} Take $Y_2:=Z_2$, and let $Y_3$ be the upper triangular part of $Z_3$ (so that $Z_3=Y_3-Y_3^{\tr}$, since $Z_3$ is also skew-symmetric). Take
 \[
 Y:=\begin{pmatrix} Y_1&0\\Y_2&Y_3 \end{pmatrix}.
 \]
 Then $Y$ satisfies the claim, as in the proof of Lemma~\ref{(I):skewsym}. Return matrix $Y$.

\smallskip
\begin{algorithm}\label{A11} Direct factorization of a skew-symmetric matrix.\par\smallskip
\textbf{Input:} A pair $(Z,W)$ of skew-symmetric $c \times c$ matrices, where $W$ is $\Z$-invertible and in \texttt{skew normal form}.

\textbf{Output:} A $c \times c$ matrix $Y$ such that $Z=Y^{\tr}WY$.
\end{algorithm}

 \underline{Step 1.} Since $W$ is $\Z$-invertible then $c=2r$ for some $r \geq 0$. Find a matrix $P$ such that
 \[
 P^{\tr}\wt{Z}P=f_1\begin{pmatrix}0&1\\-1&0\end{pmatrix} \oplus \ldots \oplus f_r\begin{pmatrix}0&1\\-1&0\end{pmatrix},
 \]
  for integers $f_1,\ldots,f_r$ (the product $P^{\tr}\wt{Z}P$ not necessarily in skew normal form).

 \underline{Step 2.} Consider the $2r \times 2r$ matrix $S=\diag(1,d_1,1,d_2,\ldots,1,d_r)$ and return $Y:=SP^{-1}$.

\acknowledgements{Thanks to the anonymous referees for their careful revision and useful suggestions. Thanks to the Instituto de Matem\'aticas UNAM, Mexico, for financial support. Part of this paper was completed during a research stay at the Faculty of Mathematics and Computer Science of Nicolaus Copernicus University in Toru\'n within project {\em University Centre of Excellence \lq\lq Dynamics, Mathematical Analysis and Artificial Intelligence''}. The author expresses his gratitude to NCU for the hospitality and for providing excellent working conditions.}


\bibliographystyle{fundam}

\end{document}